\documentclass[11pt]{amsart}

\usepackage{hyperref}
\usepackage{mathtools}
\usepackage{amsthm}
\usepackage{amssymb}
\usepackage{amsfonts}
\usepackage[utf8]{inputenc}
\usepackage{float}
\usepackage{verbatim}
\usepackage{enumitem}
\usepackage{booktabs}

\usepackage{tikz}

\numberwithin{equation}{section}
\newtheorem{theorem}{Theorem}[section]

\newtheorem{proposition}[theorem]{Proposition}
\newtheorem{prop}[theorem]{Proposition}

\newtheorem{corollary}[theorem]{Corollary}
\newtheorem{cor}[equation]{Corollary}

\newtheorem{lemma}[theorem]{Lemma}
\newtheorem{conjecture}[theorem]{Conjecture}

\theoremstyle{definition}
\newtheorem{definition}[theorem]{Definition}
\theoremstyle{remark}
\newtheorem{remark}[theorem]{Remark}

\newcommand{\bd}{\partial}

\newcommand{\supp}{\mathrm{supp}}

\newcommand{\eps}{\varepsilon}

\DeclareMathOperator{\area}{\mathrm{Area}}

\newcommand{\Id}{\mathrm{Id}}

\newcommand{\Sph}{\mathbb{S}}

\newcommand{\N}{\mathbb{N}}
\newcommand{\Z}{\mathbb{Z}}

\newcommand{\DD}{\mathbb{D}}

\newcommand{\D}{\mathcal{D}}
\newcommand{\Dcal}{\mathcal{D}}

\newcommand{\dist}{\mathrm{dist}}
\newcommand{\R}{\mathbb{R}}
\newcommand{\B}{\mathbb{B}}

\newcommand{\genus}{\mathrm{genus}}
\newcommand{\Isom}{\mathrm{Isom}}
\newcommand{\Conf}{\mathrm{Conf}}
\newcommand{\Area}{\mathrm{Area}}

\newcommand{\Diff}{\mathrm{Diff}}
\newcommand{\Met}{\mathrm{Met}}

\newcommand{\Fix}{\mathrm{Fix}}

\newcommand{\sigmabar}{\overline{\sigma}}
\newcommand{\ulow}{u_{\mathrm{low}}}
\newcommand{\uhigh}{u_{\mathrm{high}}}

\definecolor{light-gray}{gray}{.95}
\definecolor{dark-gray}{gray}{.7}

\begin{document}
\title[Asymptotics for minimal surfaces]{Large topology asymptotics for spectrally extremal minimal surfaces in $\mathbb{B}^3$ and $\mathbb{S}^3$}

\author[M.~Karpukhin]{Mikhail~Karpukhin}
\author[P.~McGrath]{Peter~McGrath}
\author[D.~Stern]{Daniel~Stern}
\date{}
\address{Department of Mathematics, University College London, 25 Gordon Street, London, WC1H 0AY, UK} \email{m.karpukhin@ucl.ac.uk}
\address{Department of Mathematics, North Carolina State University, Raleigh NC 27695} 
\email{pjmcgrat@ncsu.edu}
\address{Department of Mathematics, Cornell University, Ithaca NY, 14853} \email{daniel.stern@cornell.edu}

\begin{abstract}
In recent work with Kusner, we developed a method, based on the equivariant optimization of Laplace and Steklov eigenvalues, for producing minimal surfaces of prescribed topology in low-dimensional balls and spheres. We used the method to construct many new minimal embeddings in $\mathbb{S}^3$ with area below $8\pi$, and many new free boundary minimal embeddings in $\mathbb{B}^3$ with area below $2\pi$.
In this paper, we study the geometry of these surfaces in more detail, with an emphasis on studying sharp area estimates and varifold limits in the large Euler characteristic regime.  This allows us to confirm some well-known conjectures regarding the space of low-area minimal surfaces in $\Sph^3$ in this class of examples and the special role played by Lawson's $\xi_{\gamma,1}$ surfaces. We also confirm analogous statements in $\B^3$ and identify a family of free boundary minimal surfaces in $\B^3$ most closely resembling $\xi_{\gamma,1}$.
\end{abstract}

\maketitle

%\tableofcontents

\section{Introduction}
\label{Sintro}
In the last thirty years, an emerging dictionary between isoperimetric problems in spectral geometry and minimal surfaces in distinguished ambient spaces has led to striking developments in both spectral theory and the study of minimal surfaces, starting with Nadirashvili's discovery %observation 
 \cite{Nadirashvili} that extremal metrics for the area-normalized Laplacian eigenvalues  
 \[
 {\bar{\lambda}_k(M,g):=\lambda_k(M,g)\area(M,g)}
 \]
  on a closed surface $M$ are induced by minimal immersions into spheres. Initially, this connection was exploited to obtain sharp spectral inequalities and characterize extremal metrics, as in the identification of $\bar{\lambda}_1$-maximizing metrics on the torus and Klein bottle \cite{Nadirashvili, Jakobson},  $\bar{\lambda}_k$-maximizing metrics on $\mathbb{S}^2$ and $\mathbb{RP}^2$ \cite{KNPP, KarRP2}, and metrics maximizing the first length-normalized Steklov eigenvalue on the annulus and M\"obius band \cite{FraserSchoen}. 

More recently, there has been growing interest in using this connection in service of minimal surface theory, employing variational methods for eigenvalues as a tool for producing new minimal surfaces \cite{FraserSchoen, KSDuke, KKMS}. Though general constructions of this form yield surfaces of high codimension, possibly with branch points and self-intersections \cite{NaySh}, in \cite{KKMS} the authors and R. Kusner identified a large family of variational problems for Laplace and Steklov eigenvalues for which extremal metrics give rise to \emph{embedded} minimal surfaces of \emph{codimension one}, and by developing the corresponding existence theory, gave many new examples of minimal surfaces in $\mathbb{S}^3$ and free boundary minimal surfaces in $\mathbb{B}^3$ with prescribed topology and symmetries. 

The goal of the present paper is to gain an improved understanding of the \emph{geometry} of the minimal surfaces in $\mathbb{S}^3$ and $\mathbb{B}^3$ constructed in \cite{KKMS}, with an emphasis on \emph{sharp area estimates} and \emph{varifold limits} in the large Euler characteristic regime, similar in spirit to those available for doubling and desingularization constructions \cite{LDG}. As advertised in \cite[Section 1.6]{KKMS}, among other applications, these results allow us to confirm some well-known conjectures regarding the space of low-area minimal surfaces in $\mathbb{S}^3$ for the large class of new examples constructed in \cite{KKMS}, providing further evidence that the Lawson surfaces $\xi_{\gamma,1}$ have least area among minimal surfaces with genus $\gamma$ (cf. \cite[Conjecture 8.4]{KusnerWillmore}) and are the unique family of minimal surfaces desingularizing a pair of (non-identical) equatorial spheres (cf.  \cite[Question 4.4]{Kap:survey}). In $\mathbb{B}^3$, we also obtain analogs of these results for the many new free boundary minimal surfaces with area $<2\pi$ constructed in \cite{KKMS}, and we identify a family of free boundary minimal surfaces which we believe can be regarded as free boundary analog of the Lawson surfaces $\xi_{\gamma,1}$.

\subsection{$\bar{\lambda}_1$-extremal surfaces in $\mathbb{S}^3$}\label{intro.s3}

Let $M$ be a closed, oriented surface and let $\Gamma\leq \Diff(M)$ be a finite group of diffeomorphisms. Recall from \cite[Definition 5.4]{KKMS} that the pair $(M,\Gamma)$ is called a \emph{basic reflection surface} if $\Gamma$ contains an involution $\tau\in \Gamma$ whose fixed point set $M^\tau$ contains a curve, and whose quotient $M/\langle \tau\rangle$ has genus zero. As in \cite{KKMS}, we are particularly interested in basic reflection surfaces $(M,\Gamma)$ where $\Gamma$ has the form
\begin{equation}\label{gamma.prod}
\Gamma:=\langle \tau\rangle\times G,
\end{equation}
where $\tau$ is an involution as above, and $G$ is a finite reflection group--i.e., a finite group $G$ generated by elements whose fixed-point sets separate $M$. For pairs $(M,\Gamma)$ of this type, it was shown in \cite{KKMS} that $\bar{\lambda}_1(M,g)$ can always be maximized over $\Gamma$-invariant metrics $g\in \Met_{\Gamma}(M)$, giving rise to an extremal metric induced by a minimal embedding in $\mathbb{S}^3$. 

\begin{theorem}\label{kkms.thm1}{{\cite[Theorem 8.7]{KKMS}}}
For every orientable basic reflection surface $(M,\Gamma)$ with $\Gamma$ of the form $\Gamma=\langle \tau\rangle\times G$ for a finite reflection group $G$, there exists a metric $g\in \Met_{\Gamma}(M)$ realizing the supremum
$$\Lambda_1(M,\Gamma):=\sup\{\bar{\lambda}_1(M,h)\mid h\in \Met_{\Gamma}(M)\}$$
of $\bar{\lambda}_1$ over all $\Gamma$-invariant metrics $h\in \Met_{\Gamma}(M)$. Moreover, there is an orthogonal representation $\rho: \Gamma\to O(4)$ and a $\rho$-equivariant minimal embedding $\Phi: M\to \mathbb{S}^3$ satisfying $\Delta_{g}\Phi=\lambda_1(M, g)\Phi$, with 
$$\Area(\Phi)=\frac{1}{2}\Lambda_1(M,\Gamma)<8\pi.$$
\end{theorem}

As shown in \cite[Lemma 5.35]{KKMS}, these surfaces are \emph{doublings} of the equator $\Sph^2 \subset \Sph^3$ in the sense of \cite[Definition 1.1]{LDG}: the nearest-point projection $\pi$ to $\Sph^2$ is well-defined on $M$ and $M = M_1 \cup M_2$, where $M_1$ is a $1$-manifold, $M_2 \subset M$ is open, $\pi|_{M_1}$ is a diffeomorphism, and $\pi|_{M_2}$ is a $2$-sheeted covering map. 

We emphasize that these minimal doublings are determined not only by their topology and the isomorphism class of $G$, but by the specific action of $G$ on $M$. For instance, as discussed in [KKMS, Proposition 8.8], one obtains at least $\lfloor \frac{\gamma-1}{4}\rfloor+1$ distinct $\mathbb{Z}_2\times \mathbb{Z}_2$-symmetric minimal $\Sph^2$-doublings of genus $\gamma$ by applying Theorem \ref{kkms.thm1} with $G$ given by a single reflection with fundamental domain of prescribed topology.

Prior to \cite{KKMS}, other families of minimal surfaces in $\mathbb{S}^3$ with prescribed genus and area $<8\pi$ have been constructed by variational and perturbative methods \cite{Lawson, KPS, Kapouleas, KapMcG, LDG}; among these, the earliest and best studied are the Lawson surfaces $\xi_{\gamma,1}$. Several well-known conjectures assert that this family is distinguished among the others in various senses. First, Kusner conjectured \cite[Conjecture 8.4]{KusnerWillmore} that the surfaces $\xi_{\gamma,1}$ have least area among all genus $\gamma$ minimal surfaces in $\mathbb{S}^3$, and in fact the lowest Willmore energy among all genus $\gamma$ surfaces in $\mathbb{S}^3$. Note that recent work of \cite{HHT,CHHT} gives a sharp asymptotic expansion for this area $\Area(\xi_{\gamma,1})$, of the form
$$\Area(\xi_{\gamma,1})=8\pi-\frac{4\pi \log 2}{\gamma}+O(1/\gamma^2)$$
as $\gamma\to\infty$.

As an application of the tools introduced below, we confirm that for $\gamma$ sufficiently large, all of the $\sim \gamma/4$ minimal surfaces of genus $\gamma$ constructed in \cite{KKMS} indeed have area larger than $\Area(\xi_{\gamma,1})$, with $8\pi-\Area(M_{\gamma})$ decaying exponentially fast in $\gamma$ as $\gamma\to\infty$ for most of those families. More precisely, the basic reflection pairs $(M,\Gamma)$ can be split into two types, which exhibit different behaviors.

\begin{definition}\label{scherk.def}
We say a basic reflection pair $(M,\Gamma)$ with $M$ orientable is of \emph{Scherk type} if every $\Gamma$-invariant metric on $M$ is conformally equivalent to the double of $\mathbb{S}^2\setminus \mathcal{D}$, where $\mathcal{D}$ is a union of geodesic disks with centers on the equator. A basic reflection pair $(M,\Gamma)$ that is not of Scherk type is said to be  \emph{generic}. 
\end{definition}

\begin{remark}
It is a straightforward exercise to check that $(M,\Gamma)$ is of Scherk type whenever there exist two distinct reflections $\tau_1,\tau_2\in \Gamma$ for which $M/\langle \tau_i\rangle$ has genus zero. Indeed, for the reflection groups we consider, this is almost equivalent to the definition given above, with essentially one exceptional case when $\Gamma\cong \mathbb{Z}_2\times D_k$. See Lemma \ref{scherk.lem} below for an enumeration of all the Scherk-type basic reflection surfaces of the form $\Gamma=\langle \tau\rangle\times G$.
\end{remark}

To justify the use of `generic', note for instance that when $\Gamma\cong \mathbb{Z}_2\times \mathbb{Z}_2$,  for every genus $\gamma$, there is exactly one pair $(M,\Gamma)$ of Scherk type, up to equivariant homeomorphism, and $\lfloor \frac{\gamma+1}{2}\rfloor$ generic pairs. With this language in place, we have the following.

\begin{theorem}\label{s3.lbd}
There exist $\gamma_0\in \mathbb{N}$ such that for every pair $(M,\Gamma)$ as in Theorem \ref{kkms.thm1} with $genus(M)=\gamma\geq \gamma_0$,
$$\frac{1}{2}\Lambda_1(M,\Gamma)\geq \Area(\xi_{\gamma,1}).$$
Moreover, there exists $c>0$ such that $\frac{1}{2}\Lambda_1(M,\Gamma)\geq 8\pi-e^{-c\sqrt{\gamma}}$ for all generic basic reflection pairs $(M,\Gamma)$ of genus $\gamma$.
\end{theorem}

Since the Lawson surfaces have first Laplace eigenvalue equal to $2$ \cite{ChoeSoret}, for those pairs $(M,\Gamma)$ where $M$ has genus $\gamma$ and $\Gamma$ is conjugate to a subgroup of $\Isom(\xi_{\gamma,1})$, the lower bound 
$$\Lambda_1(M,\Gamma)\geq \bar{\lambda}_1(\xi_{\gamma,1})=2\Area(\xi_{\gamma,1})$$
follows immediately from the definition of $\Lambda_1(M,\Gamma)$; in particular, this provides the desired lower bound for all pairs $(M,\Gamma)$ of Scherk type. For generic pairs $(M,\Gamma)$, we prove by direct construction of test metrics that $\Lambda_1(M,\Gamma)\geq 16\pi-e^{-c\sqrt{\gamma}}$, and in fact most of these families have $\Lambda_1(M,\Gamma)$ converging to $16\pi$ at the slightly faster rate $16\pi - \Lambda_1(M,\Gamma)\leq e^{-c\gamma}$.

On the other hand, using stability methods similar in spirit to those of \cite{KNPS} and \cite{KSDuke}, we establish matching upper bounds, showing that these lower bounds on $\Lambda_1(M,\Gamma)$ are qualitatively sharp.

\begin{theorem}\label{s3.ubd}
There exists $C_1>0$ such that for every basic reflection surface $(M,\Gamma)$,
$$\frac{1}{2}\Lambda_1(M,\Gamma)\leq 8\pi-e^{-C_1|\chi(M)|}.$$
Moreover, there is $C_2>0$ such that if $(M,\Gamma)$ is of Scherk type, then
$$\frac{1}{2}\Lambda_1(M,\Gamma)\leq 8\pi-\frac{C_2}{|\chi(M)|}.$$
\end{theorem}

\begin{remark}
The universal upper bound $\frac{1}{2}\Lambda_1(M,\Gamma)\leq 8\pi-e^{-C_1|\chi(M)|}$ among all basic reflection surfaces is equivalent to the following statements, where $C_1$ and $C_2$ denote absolute constants.
\begin{itemize}
\item For any metric $g$ on the surface $\Omega$ with genus zero, the minimum $\mu_1:=\min\{\lambda_1^N(\Omega,g),\lambda_1^D(\Omega,g)\}$ of the Dirichlet and Neumann eigenvalues satisfies
$$\mu_1(\Omega,g)\cdot \Area(\Omega,g)\leq 8\pi-e^{-C_1 |\chi(\Omega)|}.$$
\item Every genus zero free boundary minimal surface $\Sigma\subset \mathbb{S}^3_+$ embedded by first eigenfunctions has area 
$$|\Sigma|\leq 4\pi-e^{-C_2|\chi(\Sigma)|}.$$
\end{itemize}
These estimates may be compared with those of \cite[Theorem 1.8]{KSDuke} for the Steklov optimization problem on genus zero surfaces. Per a well-known conjecture of Yau \cite[Problem 100]{Yau:Problems} the ``embedded by first eigenfunctions" condition in the latter statement can likely be weakened to embeddedness; it could be interesting to find an alternate proof of the upper bound in this case that bypasses the spectral estimates used in this paper.
\end{remark}

The preceding upper and lower area estimates may be compared with area expansions \cite[Theorem A(v)]{LDG} for $\Sph^2$-doublings constructed by PDE gluing methods \cite{Kapouleas, KapMcG, LDG}.  For example, using the formula \cite[Equation 6.23]{Kapouleas} with this area expansion shows the ``equator-poles" examples $M^{\mathrm{eq-pol}}_\gamma$ with genus $\gamma$ from \cite[Theorem 7.3]{Kapouleas} have area satisfying
\begin{align*}
8\pi - C \sqrt{\gamma} e^{-\sqrt{2\gamma}} < |M^{\mathrm{eq-pol}}_\gamma| < 8\pi - c \sqrt{\gamma}e^{ -\sqrt{2\gamma}}
\end{align*}
for appropriate constants $0< c< C$.  It follows from the formulas in \cite{Kapouleas, KapMcG, LDG} that all of the other families of $\Sph^2$-doublings constructed in those articles have area tending to $8\pi$ in the genus $\gamma$ at rates bounded between $8\pi - e^{-c \gamma}$ and $8\pi - e^{-C \gamma}$ for appropriate constants $c$.

For pairs $(M,\Gamma)$ of Scherk type, combining Theorems \ref{s3.lbd} and \ref{s3.ubd} gives
$$8\pi-\frac{8\pi \log 2}{|\chi(M)|}+O(1/|\chi(M)|^2)\leq \frac{1}{2}\Lambda_1(M,\Gamma)\leq 8\pi-\frac{C_2}{|\chi(M)|},$$
and we strongly suspect that $\Lambda_1(M,\Gamma)=2\Area(\xi_{\gamma,1})$ for all pairs $(M,\Gamma)$ of this type with genus $\gamma$. For $\Gamma$ conjugate to the \emph{full} isometry group of $\xi_{\gamma,1}$, this follows from a uniqueness result of Kapouleas--Wiygul \cite{KWsym}, and the desired equality would follow if the symmetry hypotheses in their result could be weakened to invariance under a certain $\mathbb{Z}_2\times \mathbb{Z}_2$ subgroup of $\xi_{\gamma,1}$. In spectral terms, this is equivalent to the following conjecture.

\begin{conjecture}
Let $\Omega_k$ be a $2k$-gon, and write $\partial\Omega_k:=\Gamma_1\cup \Gamma_2$, where each $\Gamma_i$ consists of $k$ disjoint segments. Given a metric $g$ on $\Omega_k$, let $\lambda_1^N(\Omega_k,g)$ denote the first nonzero Neumann eigenvalue of $\Delta_g$, and set 
$$\lambda_1^{\mathrm{mix}}(\Omega_k,g):=\inf\left\{\left.\frac{\int_{\Omega_k}|d\phi|_g^2dv_g}{\int_{\Omega_k}\phi^2dv_g}\right| \phi|_{\Gamma_i}\equiv 0\text{ for either }i=1\text{ or }2\right\}.$$
 Then for any metric $g$ on $\Omega_k$, 
$$\Area(\Omega_k,g)\cdot \min\{\lambda_1^N(\Omega_k,g),\lambda_1^{\mathrm{mix}}(\Omega_k,g)\}\leq \frac{1}{2}\Area(\xi_{k,1}),$$
with equality when $(\Omega_k,g)$ is homothetic to a quarter of $\xi_{k,1}$.
\end{conjecture}

Another interesting open problem about the space of minimal surfaces in $\mathbb{S}^3$ with area $<8\pi$ is that of identifying the \emph{boundary} of this space, with respect to the varifold topology. Given a sequence $M_k\subset \mathbb{S}^3$ of minimal surfaces with $\Area(M_k) \nearrow 8\pi$ and genus $\to\infty$, it is well-known that subsequences must converge as varifolds to a sum $S_1+S_2$ of two (multiplicity-one) great two-spheres $S_1,S_2$. A priori, the angle between $S_1$ and $S_2$ could be arbitrary, but in practice only two cases are known to occur: $S_1=S_2$, corresponding to doublings of an equator, and $S_1\perp S_2$, which arises as the large genus limit of the Lawson surfaces $\xi_{\gamma,1}$. In \cite[Section 4]{Kap:survey}, Kapouleas presents evidence that these are indeed the \emph{only} possible varifold limits,  suggesting the following conjecture. %% [NOT EXACTLY WHAT KAPOULEAS SAYS, SO MAY WANT TO PHRASE MORE CAREFULLY..]

\begin{conjecture}[cf. {\cite[Question 4.3]{Kap:survey}}]
\label{lim.conj} The boundary of the space of minimal surfaces in $\mathbb{S}^3$ with area $<8\pi$ consists only of multiplicity-two great spheres and unions of orthogonal great spheres. Moreover, the Lawson surfaces $\xi_{\gamma,1}$ are the only family desingularizing a pair of orthogonal great spheres.
\end{conjecture}

In other words, apart from the Lawson surfaces $\xi_{\gamma,1}$, it is expected that all families of minimal surfaces in $\mathbb{S}^3$ with area $<8\pi$ resemble doublings of the equator in the large topology limit. In the class of $\bar{\lambda}_1$-extremal basic reflection surfaces, we show that Conjecture \ref{lim.conj} is closely related to Kusner's conjecture that the Lawson surfaces $\xi_{\gamma,1}$ have least area for their genus: namely, any sequence $M_{\gamma}$ of genus $\gamma$ minimal surfaces of this type with $8\pi-|M_{\gamma}|\ll8\pi-|\xi_{\gamma,1}|$ must converge to a multiplicity-two sphere as $\gamma\to\infty$. More precisely, we show the following.

\begin{theorem}\label{var.close.thm}
Let $N$ be a genus zero free boundary minimal surface, embedded in the hemisphere $\Sph^3_+$ by first eigenfunctions. Then
$$\int_N \dist_{\partial \Sph^3_+} \leq C\sqrt{1+|\chi(N)|}\sqrt{4\pi-|N|}$$
for some universal constant $C$. In particular, any sequence $N_k$ with $4\pi-|N_k|=o(\frac{1}{1+|\chi(N_k)|})$ must converge to the boundary $\partial \mathbb{S}^3_+=\mathbb{S}^2\times\{0\}$ in the varifold sense as $k\to\infty$.
\end{theorem}

As a consequence, if $(M,\Gamma)$ is a generic basic reflection surface, this result, together with the bound $8\pi-\frac{1}{2}\Lambda_1(M,\Gamma)\leq e^{-c\sqrt{|\chi(M)|}}$ of Theorem \ref{s3.lbd} gives the following corollary.

\begin{corollary}
Let $M\subset \mathbb{S}^3$ be the minimal surface realizing $\Lambda_1(M,\Gamma)$ for a generic basic reflection pair $(M,\Gamma)$. Then there exists a great sphere $\Sph^2 \subset \mathbb{S}^3$ and universal constants $C<\infty$, $c_1>0$ such that
$$\int_M\dist_{\Sph^2} \leq Ce^{-c_1\sqrt{|\chi(M)|}}.$$
In particular, if $\{M_k\}$ is a sequence of such surfaces with $|\chi(M_k)|\to\infty$, then $M_k$ converges rapidly to a great sphere with multiplicity two.
\end{corollary}

For the pairs $(M,\Gamma)$ of Scherk type, Theorem \ref{var.close.thm} also gives some nontrivial information, leading in particular to the proof of the upper bound $\frac{1}{2}\Lambda_1(M,\Gamma)\leq 8\pi-\frac{C_2}{|\chi(M)|}$ in Theorem \ref{s3.ubd}. Roughly speaking, if we had such a pair $(M,\Gamma)$ with $8\pi-\frac{1}{2}\Lambda_1(M,\Gamma)=o(1/|\chi(M)|)$, Theorem \ref{var.close.thm} would imply that the associated minimal surface simultaneously lies very close to \emph{two orthogonal} great spheres corresponding to the two reflections $\tau_1,\tau_2$, which is evidently impossible. 

\subsection{$\bar{\sigma}_1$-extremal surfaces in $\mathbb{B}^3$}\label{intro.b3}

\subsubsection{$\bar{\sigma}_1$-extremal basic reflection surfaces}

Analogous to the case of closed surfaces, if $N$ is a compact oriented surface with boundary and $\Gamma\leq \Diff(N)$ is a finite group of diffeomorphisms, we say the pair $(N,\Gamma)$ is a \emph{basic reflection surface} if there is a reflection $\tau\in \Gamma$ for which the quotient $N/\langle \tau \rangle$ has genus zero and all but one boundary component contained in the fixed-point set $N^\tau$ (see \cite[Def. 5.3]{KKMS}). For any such pair, we consider the supremum
$$\Sigma_1(N,\Gamma):=\sup\{\bar{\sigma}_1(N,g)\mid g\in \Met_{\Gamma}(N)\}$$
of the length-normalized first Steklov eigenvalue $\bar{\sigma}_1(N,g)=L(\partial N,g)\cdot \sigma_1(N,g)$ on the space $\Met_{\Gamma}(N)$ of $\Gamma$-invariant metrics.

In \cite[Theorem 9.15]{KKMS}, existence of metrics realizing $\Sigma_1(N,\Gamma)$ is proved for a large class of basic reflection pairs $(N,\Gamma)$, and by \cite[Lemma 5.37]{KKMS}, these metrics are induced by $\Gamma$-equivariant free boundary minimal embeddings $(N,\partial N)\to (\mathbb{B}^3,\mathbb{S}^2)$ of area $\frac{1}{2}\Sigma_1(N,\Gamma)<2\pi$.  Just as in the closed case, these embeddings are doublings, this time of an equatorial disk $\mathbb{D}^2$.  Though for technical reasons \cite[Theorem 9.15]{KKMS} does not establish existence of a maximizing metric for every basic reflection pair $(N,\Gamma)$, for every $\gamma\geq 0$ and $k\geq 2$, \cite[Theorem 9.15]{KKMS} supplies at least $\lfloor \frac{\gamma-2}{4}\rfloor$ distinct free boundary minimal surfaces with genus $\gamma$, $k$ boundary components, and area $<2\pi$. 

\begin{remark}
In a recent preprint \cite{PetridesNew}, Petrides has announced a significant refinement of the main analytic tool used in the existence theory of \cite{KKMS}, which together with the results of \cite{KKMS} should show that $\Sigma_1(N,\Gamma)$ is indeed realized by a free boundary minimal embedding $(N,\partial N)\subset (\mathbb{B}^3,\mathbb{S}^2)$ for \emph{any} basic reflection pair $(N,\Gamma)$.
\end{remark}

As with the $\bar{\lambda}_1$-extremal closed basic reflection surfaces in $\mathbb{S}^3$, we are able to obtain qualitatively sharp area estimates for these $\bar{\sigma}_1$-extremal free boundary minimal surfaces in the large topology limit, and relate this to their varifold limits. One interesting distinction with the analogous results in $\mathbb{S}^3$, where topology is encoded by genus alone, is that the genus and the number of boundary components  play very different roles in determining the possible range of areas for these free boundary minimal surfaces in the large topology limit.

First, we record the following universal upper and lower bounds.

\begin{theorem}\label{main.stek.bds}
For every basic reflection surface $(N,\Gamma)$ with genus $\gamma$ and $k$ boundary components, we have an estimate of the form
$$4\pi-e^{-C_1(\gamma+k)}\geq \Sigma_1(N,\Gamma)\geq \max\{4\pi-e^{-C_2k},4\pi-C_3/\gamma\}$$
for universal constants $C_1,C_2,C_3>0$. In particular, the associated free boundary minimal surface $N\subset \mathbb{B}^3$ has area
$$2\pi-e^{-C_1'(\gamma+k)}\geq |N| \geq \max\{2\pi-e^{-C_2'k},2\pi-C_2'/\gamma\}.$$
\end{theorem}

In particular, note that the gap $4\pi-\Sigma_1(N,\Gamma)$ decays exponentially fast as the number of boundary components of $N$ increases, but, for a fixed number of boundary components, could vanish as slowly as $1/\gamma$ as the genus $\gamma$ grows but the number of boundary components remains fixed, and we present examples below for which the latter behavior does occur.

Note that the maximum of $\Sigma_1(N,\Gamma)$ over all basic reflection pairs $(N,\Gamma)$ on a surface $N$ of fixed topological type is simply $\Sigma_1(N,\langle \tau\rangle)$, the supremum of $\bar{\sigma}_1(N,g)$ over all metrics invariant under a fixed basic reflection $\tau:N\to N$. Existence of maximizing metrics in this case is established in \cite[Theorem 9.15]{KKMS}, corresponding to the free boundary basic reflection surfaces in $\mathbb{B}^3$ of largest area for a given topology. For these surfaces, we show that the upper bound in Theorem \ref{main.stek.bds} is sharp.

\begin{theorem}\label{max.stek.bds}
There exist $C_1,C_2$ such that for the surface $N_{\gamma,k}$ with genus $\gamma$ and $k$ boundary components.
$$4\pi-e^{-C_1(\gamma+k)}\geq \Sigma_1(N_{\gamma,k},\langle \tau\rangle)\geq 4\pi-e^{-C_2(\gamma+k)}.$$
\end{theorem}

\begin{remark}
For the surfaces $N_{0,k}$ with genus zero and $k$ boundary components, it is possible \emph{a priori} that $\bar{\sigma}_1$-maximization over any subgroup $\Gamma\leq \Diff(N_{0,k})$ containing a basic reflection yields the same $D_k$-symmetric free boundary minimal surface, coinciding with those produced in \cite[Theorem 5.1]{FrS} by min-max methods or in \cite{Zolotareva} via gluing techniques. Indeed, since the results of \cite{KusnerMcGrath} confirm that the latter two families are embedded by first Steklov eigenfunctions, one can obtain more precise lower bounds for $\Sigma_1(N_{0,k},\langle \tau\rangle)$ by a careful examination of the gluing construction in \cite{Zolotareva} in combination with the area expansion in \cite[Theorem A(v)]{LDG}, but we do not pursue this here.
\end{remark}

Next, we note that the lower bound in Theorem \ref{main.stek.bds} with respect to the genus is also sharp, at least when the number of boundary components is one or two.

\begin{theorem}\label{lawsy.bds}
If $N$ has $k=1$ or $2$ boundary components and any genus $\gamma$, there exists a subgroup $\Gamma_{12}=\langle \tau_1,\tau_2\rangle\leq \Diff(N)$ generated by two distinct reflections with respect to which $N$ is a basic reflection surface. In this case, there are universal constants $C_1,C_2>0$ such that 
$$4\pi-\frac{C_1}{\gamma}\geq \Sigma_1(N_{\gamma,k},\Gamma_{12})\geq 4\pi-\frac{C_2}{\gamma}.$$
\end{theorem}

Together with \cite[Theorem 9.15]{KKMS}, this gives the following. 

\begin{corollary}
For every integer $m\geq -1$, there exists a embedded free boundary minimal surface $N_m\subset \mathbb{B}^3$ with Euler characteristic $\chi(N_m)=m$ and area
$$2\pi-\frac{C_1}{m}\geq |N_m|\geq 2\pi-\frac{C_2}{m},$$
invariant with respect to reflection about two orthogonal planes. 
\end{corollary}

The behavior of these surfaces seems analogous in many ways to that of the Lawson surfaces $\xi_{\gamma,1}$ in $\mathbb{S}^3$. Indeed, in the large topology limit, we expect that they give a desingularization of a union of orthogonal disks, perhaps coinciding with the desingularization announced by Kapouleas-Li in \cite{KapLi}. Moreover, the area estimates obtained here suggest that this family has the slowest possible area growth as $|\chi(N)|$ becomes large; in light of these estimates, it is tempting to conjecture that these surfaces have least area among all free boundary minimal surfaces in $\mathbb{B}^3$ with prescribed Euler characteristic. Some care is needed for the case of two boundary components, since in this case it is possible to add an additional (non-basic) reflection $\tau_3\in \Diff(N)$ to obtain a larger group $\Gamma_{12}'=\Gamma_{12}\times\langle \tau_3\rangle \cong \mathbb{Z}_2\times \mathbb{Z}_2\times \mathbb{Z}_2$ for which $\Sigma_1(N,\Gamma_{12}')\leq \Sigma_1(N, \Gamma_{12})$, and numerical results presented in \cite{KFS} suggest that this inequality is in general strict. We then pose the following.

\begin{conjecture}
For any free boundary minimal surface $S\subset \mathbb{B}^3$, 
$$|S|\geq \frac{1}{2}\Sigma_1(N_{\gamma,1},\Gamma_{12})\text{ when }\chi(S)\text{ is odd and }\gamma=\frac{1-\chi(S)}{2}$$
and
$$|S|\geq \frac{1}{2}\Sigma_1(N_{\gamma,2},\Gamma_{12}')\text{ when }\chi(S)\text{ is even and }\gamma=\frac{-\chi(S)}{2}.$$
\end{conjecture}

The upper bounds in Theorem \ref{lawsy.bds} are obtained as a consequence of the following theorem, analogous to Theorem \ref{var.close.thm} in the closed setting, showing that $\bar{\sigma}_1$-extremal basic reflection surfaces with $2\pi-|N|\ll 1/(1+|\chi(N)|)$ lie close to the fixed-point plane of every basic reflection.

\begin{theorem}
\label{Tfbconv}
If $N\subset \mathbb{B}^3$ is a free boundary minimal embedding by first Steklov eigenfunctions such that $N$ is a basic reflection surface with respect to $\tau(x,y,z)=(x,y,-z)$, then
$$\int_{\partial N}|z|\leq C\sqrt{2\pi-|N|}\cdot \sqrt{1+|\chi(N)|}$$
for a universal $C<\infty$.
\end{theorem}

Since the families of Theorem \ref{lawsy.bds} are basic reflection surfaces with respect to two orthogonal planes of reflection, it follows that $2\pi-|N|$ cannot decay faster than $O(\frac{1}{1+|\chi(N)|})$. For the other families of basic reflection surfaces, most  have $2\pi-|N|=O(e^{-c|\chi(N)|})$, and Theorem \ref{Tfbconv} shows that such surfaces converge to a double disk in the large topology limit.

\subsection{Organization} In Section 2, we introduce a convenient equivalent formulation of the $\bar{\lambda}_1$- and $\bar{\sigma}_1$-optimization optimization problems over basic reflection surfaces in terms of optimization problems over subsets of $\mathbb{S}^2$ and $\mathbb{D}^2$. In Sections 3 and 4, we complete the proof of Theorem \ref{s3.lbd}, constructing test metrics for the relevant optimization problems via carefully chosen domains $\Omega\subset\mathbb{S}^2$, obtained by removing a collection disk whose total area $4\pi-|\Omega|$ is kept as small as possible while maintaining a lower bound $\lambda_1^D(\Omega)\geq 2$ on the first Dirichlet eigenvalue. In Section 5, we provide similar arguments in the Steklov setting, proving the lower bounds in Theorems \ref{main.stek.bds} and \ref{max.stek.bds}. In Section 6, we prove Theorem \ref{s3.ubd} and Theorem \ref{var.close.thm}, using as a starting point stability estimates modelled on those of \cite{KNPS, KSDuke}. In Section 7, we then implement analogous arguments in the Steklov setting, completing the proofs of the upper bounds in Theorems \ref{main.stek.bds} and \ref{max.stek.bds}, proving Theorem \ref{Tfbconv}, and deducing Theorem \ref{lawsy.bds} as a corollary.

\subsection*{Acknowledgements} The authors thank Rob Kusner, Mario Schulz, and Antoine Song for interesting discussions related to this work. During the completion of this work, D.S. was supported in part by the NSF grant DMS 2404992 and the Simons Foundation award MPS-TSM-00007693.

\section{Equivalent formulations of the optimization problem}

To obtain effective estimates for the suprema $\Lambda_1(M,\Gamma)$ and $\Sigma_1(N,\Gamma)$ among basic reflection pairs, it is convenient to work with an equivalent formulation via shape optimization problems over domains in $\mathbb{S}^2$ and $\mathbb{D}^2$, respectively.

Given a compact oriented surface $(\Omega,g)$, denote by $\lambda_1^D$ and $\lambda_1^N$ the first nonzero eigenvalues of the Laplacian $\Delta_g$ with homogeneous Dirichlet and Neumann boundary conditions, respectively, on $\partial\Omega$. We then set
$$\mu_1(\Omega,g):=\min\{\lambda_1^D(\Omega,g),\text{}\lambda_1^N(\Omega,g)\},$$
and
$$\bar{\mu}_1(\Omega,g)=\Area(M,g)\mu_1(\Omega,g).$$

Now, it is easy to see that there is a one-to-one correspondence between oriented basic reflection pairs $(M,\Gamma)$ of the form $\Gamma=\langle \tau\rangle\times G$ as in Theorem \ref{kkms.thm1} and pairs $(\Omega,G)$ consisting of compact surfaces $\Omega$ with genus zero and $(\genus(M)+1)$ boundary components together with an action of $G\leq \Diff(\Omega)$. Given a pair $(\Omega,G)$ of the latter type, define
$$\mathcal{M}_1(\Omega,G):=\sup\{\bar{\mu}_1(\Omega,g)\mid g\in \Met_G(\Omega)\}.$$
We then have the following.
\begin{lemma}\label{double.lem}
Given a closed oriented basic reflection pair $(M,\Gamma=\langle \tau\rangle \times G)$ as above with associated genus zero pair $(\Omega,G)$ corresponding to a fundamental domain of $\tau$ in $M$, we have
$$\Lambda_1(M,\Gamma)=2\mathcal{M}_1(\Omega,G).$$
\end{lemma}
\begin{proof}
First, note that if $g\in \Met_G(\Omega)$ is a metric on $\Omega$ with totally geodesic boundary, then doubling $\Omega$ across its boundary gives rise to a $\Gamma$-invariant metric $\bar{g}\in \Met_{\Gamma}(M)$, and any $\Gamma$-invariant metric on $M$ in turn restricts to a $G$-invariant metric on $\Omega$ with totally geodesic boundary.  In this setup, it is easy to check that the $\Delta_g$-eigenfunctions on $\Omega$ with homogeneous Neumann condition are precisely the restrictions of $\tau$-even eigenfunctions for $\Delta_{\bar{g}}$ on $M$, while the Dirichlet eigenfunctions on $\Omega$ are the restrictions of the $\tau$-odd eigenfunctions on $M$, and it follows immediately that
$$\lambda_1(M,\bar{g})=\mu_1(\Omega,g),$$
and therefore $\bar{\lambda}_1(M,\bar{g})=2\bar{\mu}_1(\Omega,g)$. As an immediate consequence, taking the supremum over all $\bar{g}\in \Met_{\Gamma}(M)$ gives
$$\Lambda_1(M,\Gamma)\leq 2\mathcal{M}_1(\Omega,G).$$

More generally, even if $g\in \Met_G(\Omega)$ does not have totally geodesic boundary, we can write $g$ in the form $g=e^{2f}h$ where $h$ has totally geodesic boundary \cite[Theorem 1(b)]{Sarnak}, and denoting by $\bar{h}\in \Met_{\Gamma}(M)$ the metric obtained by doubling $(\Omega,h)$ across its totally geodesic boundary, we can regard the double $\bar{g}$ of $g$ to $M$ as a Lipschitz metric in $\Met_{\Gamma}(M)$ conformal to $\bar{h}$, again satisfying
$$\lambda_1(M,\bar{g})=\mu_1(\Omega,g).$$
Approximating $\bar{g}$ by smooth metrics in $\Met_{\Gamma}(M)$, well-known continuity properties of Laplace eigenvalues in the space of Lipschitz metrics imply that $\bar{\lambda}_1(M,\bar{g})\leq \Lambda_1(M,\Gamma),$ whence $2\bar{\mu}_1(\Omega,g)\leq \Lambda_1(M,\Gamma)$ for every $g\in \Met_G(\Omega)$; in other words, $2\mathcal{M}_1(\Omega,G)\leq \Lambda_1(M,\Gamma),$ as desired.
\end{proof}

Similarly, if $(N,\Gamma)$ is a basic reflection pair where $\partial N\neq \varnothing$ and $\Gamma=\langle \tau\rangle\times G$, a fundamental domain $\Omega$ for $\tau$ can be identified conformally with the complement $\Omega=\mathbb{D}\setminus \mathcal{D}$ of a finite, disjoint union of disks $\mathcal{D}$ in the disk $\mathbb{D}$, where $\partial \mathcal{D}$ corresponds to the fixed point set of $\tau$. Given a metric $g$ on $\Omega$, we then define the mixed Steklov-Neumann and Steklov-Dirichlet eigenvalues for $(\Omega,g)$ by
$$
\sigma_1^N(\Omega,g):=\inf\left\{\left.\frac{\int_{\Omega}|d\phi|^2}{\int_{\partial\Omega\setminus \mathcal{D}}\phi^2}\right| \int_{\partial\Omega\setminus\mathcal{D}}\phi=0,\text{ }\phi\neq 0\right\}
$$
and
$$
\sigma_1^D(\Omega,g):=\inf\left\{\left.\frac{\int_{\Omega}|d\phi|^2}{\int_{\partial\Omega\setminus \mathcal{D}}\phi^2}\right| \phi|_{\partial \mathcal{D}}=0,\text{ }\phi\neq 0\right\},
$$
respectively. Similar to the case of Laplace eigenvalues above, it is straightforward to check that for $g\in \Met_G(\Omega)$, doubling $\Omega$ across $\partial \mathcal{D}$ gives a (Lipschitz) metric $\Met_{\Gamma}(N)$ for which
$$\bar{\sigma}_1(N,\bar{g})=2L_g(\partial\Omega\setminus \partial\mathcal{D})\cdot \min\{\sigma_1^N(\Omega,g), \sigma_1^D(\Omega,g)\}.$$
In particular, arguments entirely analogous to the proof of Lemma \ref{double.lem} give the following.

\begin{lemma}\label{stek.double}
Given an oriented basic reflection pair $(N,\Gamma=\langle \tau\rangle \times G)$ with nonempty boundary, and $(\Omega,G)$ the pair corresponding to a fundamental domain for $\tau$, we have
$$\Sigma_1(N,\Gamma)=2\sup\{L_g(\partial\Omega\setminus \partial\mathcal{D})\cdot \min\{\sigma_1^N(\Omega,g), \sigma_1^D(\Omega,g)\}\mid g\in \Met_G(\Omega)\}.$$
\end{lemma}

In particular, to obtain the lower bounds in Theorems \ref{s3.lbd} and \ref{main.stek.bds}, it suffices to construct pairs $(\Omega,G)$ as above for which $\bar{\lambda}_1^N$ and $\bar{\lambda}_1^D$ (or $\bar{\sigma}_1^N$ and $\bar{\sigma}_1^D$) are simultaneously large. As we will see in the next sections, the conditions needed to ensure that $(\Omega,G)$ has large Neumann eigenvalue $\bar{\lambda}_1^N$ are quite flexible, while obtaining lower bounds for $\bar{\lambda}_1^D$ is somewhat more delicate.

\section{Neumann estimates}

The goal of this section is to prove the following result.

\begin{prop}
\label{Pneu}
There exist constants $C>0$ and $\epsilon_0> 0$ such that if $X \subset \Sph^2$ is a finite set, $ \{ D_{2r_x}(x) \}_{x \in X}$ is a collection of pairwise disjoint disks, and $\Dcal := \cup_{x \in X} D_{r_x} (x)$ has area at most $\epsilon_0$, then the domain $\Omega : = \Sph^2 \setminus \Dcal$ satisfies
\[
\lambda_1^N(\Omega) |\Omega| \geq 8\pi - C |\Dcal|.
\]
\end{prop}

Before giving the proof, we will need the following simple lemma. 

\begin{lemma}
\label{Lneu}
There exists a constant $\epsilon_0>0$ such that if $X, \Dcal$, and $\Omega$ are as in Proposition \ref{Pneu} with $|\Dcal| \leq \epsilon_0$, then $\lambda^N_4(\Omega)\geq 4.$
\end{lemma}
\begin{proof}
First, consider the harmonic extension operator 
\[
W^{1,2}(D_2\setminus D_1)\to W^{1,2}(D_1),
\]
where here $D_r$ is the disk of radius $r$ in $\R^2$.  After rescaling and applying the result to the annuli $D_{2r_x}(x)\setminus D_{r_x}(x) \subset \Sph^2, x \in X$, we observe that the harmonic extension operator
$$H: W^{1,2}(\Omega)\to W^{1,2}(\Sph^2)$$
satisfies an estimate of the form
\begin{equation}\label{h.est}
\|d(H\phi)\|_{L^2(\Sph^2)}\leq C\|d\phi\|_{L^2(\Omega)}
\end{equation}
for a universal constant $C.$

Now, let $V\subset W^{1,2}(\Omega)$ be any $5$-dimensional subspace.  Because the coordinate and constant functions span a $4$-dimensional subspace of $L^2(\Sph^2)$, there exists a nonzero $\phi\in V$ such that $\hat{\phi} :=H\phi$ is $L^2(\Sph^2)$-orthogonal to the coordinate and constant functions. In particular, we can decompose
$$\hat{\phi}=\sum_{j=4}^{\infty}a_j u_j,$$
where $a_j\in \mathbb{R}$ and $u_j$ a $\lambda_j$-eigenfunction for $\Delta_{\Sph^2}$ with $\langle u_j,u_i\rangle_{L^2(\Sph^2)}=\delta_{ij}$. Fix an arbitrary $\Lambda>6$, and write
$$\hat{\phi}_{\Lambda}:=\sum_{\lambda_j<\Lambda}a_ju_j.$$
For the high frequency projection $\hat{\phi}-\hat{\phi}_{\Lambda}$, it then follows by \eqref{h.est} that
$$\|\hat{\phi}-\hat{\phi}_{\Lambda}\|_{L^2(\Sph^2)}^2\leq \frac{1}{\Lambda}\|d\hat{\phi}\|_{L^2(\Sph^2)}^2\leq \frac{C}{\Lambda}\|d\phi\|_{L^2(\Omega)}^2.$$
 In particular, we see that
$$\|\phi\|_{L^2(\Omega)}^2\leq \|\hat{\phi}\|_{L^2(\Sph^2)}^2\leq \|\hat{\phi}_{\Lambda}\|_{L^2(\Sph^2)}^2+\frac{C}{\Lambda}\|d\phi\|_{L^2(\Omega)}^2,$$
and since $\hat{\phi}_{\Lambda}$ is a sum of eigenfunctions with eigenvalue $\geq \lambda_4(\Sph^2)=6$, it follows that
\begin{align}
\label{Ephiup}
\|\phi\|_{L^2(\Omega)}^2\leq \frac{1}{6}\|d\hat{\phi}_{\Lambda}\|_{L^2(\Sph^2)}^2+\frac{C}{\Lambda}\|d\phi\|_{L^2(\Omega)}^2.
\end{align}

On the other hand, for the low frequency component $\hat{\phi}_{\Lambda}$, we have a $C^1$ estimate of the form
$$\|\hat{\phi}_{\Lambda}\|_{C^1(\Sph^2)}\leq C(\Lambda)\|d\hat{\phi}_{\Lambda}\|_{L^2(\Sph^2)}\leq C'(\Lambda)\|d\phi\|_{L^2(\Omega)},$$ 
where the second inequality uses \eqref{h.est}, so that
\begin{eqnarray*}
\int_{\Sph^2}|d\hat{\phi}_{\Lambda}|^2&=&\int_{\Sph^2}\langle d\hat{\phi},d\hat{\phi}_{\Lambda}\rangle\\
&\leq &\int_{\Omega}\langle d\phi,d\hat{\phi}_{\Lambda}\rangle+C'(\Lambda)\|d\hat{\phi}\|_{L^2(\Sph^2)}|\mathcal{D}|^{1/2}\|d\phi\|_{L^2(\Omega)},
\end{eqnarray*}
and therefore
$$\|d\hat{\phi}_{\Lambda}\|_{L^2(\Sph^2)}^2\leq (1+C''(\Lambda)|\mathcal{D}|^{1/2})\|d\phi\|_{L^2(\Omega)}^2.$$
Combining this with the upper bound \eqref{Ephiup} for $\|\phi\|_{L^2(\Omega)}^2$, we see that
$$\|\phi\|_{L^2(\Omega)}^2\leq \left(\frac{1}{6}+C(\Lambda)|\mathcal{D}|^{1/2}+\frac{C}{\Lambda}\right)\|d\phi\|_{L^2(\Omega)}^2.$$
Now, fix $\Lambda$ large enough such that $\frac{C}{\Lambda}<\frac{1}{100}$, and choose $\epsilon_0> 0$ small enough that $|\Dcal| < \epsilon_0$ implies $C(\Lambda)|\mathcal{D}|^{1/2}<\frac{1}{100}$. Then
$$\|\phi\|_{L^2(\Omega)}^2\leq \left(\frac{1}{6}+\frac{1}{50}\right)\|d\phi\|_{L^2(\Omega)}^2\leq \frac{1}{4}\|d\phi\|_{L^2(\Omega)}^2,$$
and since $V\subset W^{1,2}(\Omega)$ was an arbitrary $5$-dimensional space of functions, it follows that $\lambda^N_4(\Omega)\geq 4.$
\end{proof}

We are now ready to prove Proposition \ref{Pneu}.
\begin{proof}[Proof of Proposition \ref{Pneu}]
Let $u$ be a coordinate function on $\Sph^2$.  Because $u$ is a $\Delta_{\Sph^2}$-eigenfunction with eigenvalue $2$, we have
$\int_{\Sph^2} \langle d u_i , d w \rangle - 2u_i w = 0$ for any test function $w \in W^{1, 2}(\Sph^2)$.  Fixing $\phi \in W^{1,2}(\Omega)$ and using the harmonic extension $\hat{\phi} = H\phi$ as a test function shows that
\begin{equation}
\label{Ecoordt}
\begin{aligned}
\left|\int_{\Omega} \langle d u, d \phi \rangle - 2u \phi \right| &= \left| \int_{\Sph^2 \setminus \Omega} \langle d u, d \hat{\phi} \rangle - 2 u \hat{\phi}\right|\\
&\leq C\| u \|_{C^1(\Sph^2)} |\Dcal |^{1/2} \| \phi \|_{W^{1,2}(\Omega)}
\\
&\leq C |\Dcal|^{1/2} \| \phi\|_{W^{1,2}(\Omega)}. 
\end{aligned}
\end{equation}
Since $\int_{\mathbb{S}^2}u=0$ and $|u|\leq 1$, we must have $|\int_{\Omega}u|=|\int_{\mathcal{D}}u|\leq |\mathcal{D}|,$ and it follows that we can write
\begin{align*}
%u = \sum_{j=1}^3 a_{j} v_j + v_{L},
u =c+ \ulow + \uhigh,
\end{align*}
where $c$ is a constant with $|c|\leq C|\mathcal{D}|$, $\ulow$ is in the span of the first three $L^2$-normalized nonconstant Neumann eigenfunctions of $\Omega$ and $\uhigh$ is a sum of eigenfunctions with eigenvalue $\geq \lambda^N_4(\Omega) \geq 4$, where the second inequality follows from Lemma \ref{Lneu}.   Testing \eqref{Ecoordt} with $\phi = \uhigh$ gives an estimate of the form
\begin{align*}
\int_{\Omega} | d\uhigh |^2 - 2\uhigh^2 \leq C | \Dcal|^{1/2} \| \uhigh \|_{W^{1,2}(\Omega)},
\end{align*}
which together with the lower bound 
\begin{align*}
\int_\Omega |d \uhigh |^2 \geq 4 \int_\Omega \uhigh^2
\end{align*}
implies that 
\begin{align}
\label{EvL}
\| \uhigh \|_{W^{1,2}(\Omega)} \leq C | \Dcal|^{1/2}.
\end{align}

Next, note that the restrictions of the coordinate functions to $\Omega$ are linearly independent, so span a $3$-dimensional subspace of $W^{1,2}(\Omega)$; we may therefore choose $u$ so that $\ulow$ is a $\lambda^N_1(\Omega)$-eigenfunction.
Moreover, using that $\|u\|^2_{L^2(\Omega)} \geq \frac{4\pi}{3} - | \Dcal|$, $c^2\leq C|\mathcal{D}|^2$, and $\| \uhigh \|^2_{L^2(\Omega)} \leq C |\Dcal |$ by \eqref{EvL}, it follows that 
\begin{align}
\label{Ealow}
\| \ulow \|^2_{L^2(\Omega)} \geq \frac{4\pi}{3}- C | \Dcal|.
%a^2_1 \geq \frac{4\pi}{3}- C | \Dcal|. 
\end{align}
Finally, note that
\begin{align*}
 |\lambda_1^N(\Omega) - 2 | \| \ulow\|^2_{L^2(\Omega)}&= \left| \int_{\Omega} \langle du , d\ulow\rangle - 2 u \ulow\right|\\
&=  \left| 2c^2|\Omega|+\int_\Omega  | du|^2 - 2u^2 - \int_\Omega \langle du , d\uhigh \rangle - 2u \uhigh \right|\\
&\leq  \left| \int_\Dcal |du|^2 - 2u^2\right| +C|\Dcal|^2+ C | \Dcal|^{1/2} \| \uhigh\|_{W^{1,2}(\Omega)},
\end{align*}
where in the last line we've used that $u$ is an eigenfunction with eigenvalue $2$ and used the estimate \eqref{Ecoordt} with $\phi = \uhigh$. Finally, using \eqref{EvL} and the bound $\| u \|_{C^1(\Sph^2)} \leq 2$, it follows that
\begin{align*}
\| \ulow\|^2_{L^2(\Omega)} | \lambda_1^N(\Omega) - 2 | \leq C | \Dcal|,
\end{align*}
which together with the lower bound \eqref{Ealow} completes the proof. 
\end{proof}

\section{Dirichlet lower bounds}

Next, we collect some computational tools useful for estimating the first Dirichlet eigenvalue $\lambda_1^D$ of domains in $\mathbb{S}^2$.

\begin{lemma}
\label{Lmono}
Let $A_{r, R}: = D_R(0) \setminus D_r(0)$ be an annulus in $\R^2$.  There is a universal constant $C$ such that if $u \in W^{1,2}(A_{r, R})$, then
\begin{enumerate}[label=\emph{(\roman*)}]
\item $\left| \frac{d}{ds}\left(\frac{1}{s} \int_{\partial D_s} u^2\right)^{1/2} \right| \leq \left(\frac{1}{s} \int_{\partial D_s} |\frac{\partial u}{\partial \nu} |^2 \right)^{1/2}$ for a.e. $s \in (r, R)$. %, where $Q(s) = \frac{1}{s} \int_{\partial D_s} u^2$.
\end{enumerate}
If moreover $u|_{\partial D_r} = 0$, then 
\begin{enumerate}
\item[\emph{(ii)}] $\| u\|^2_{L^2(\partial D_s)} \leq C R\log(R/r) \| du \|^2_{L^2(A_{r, R})}$ for a.e. $s \in (r,R)$.
\item[\emph{(iii)}] $\| u\|^2_{L^2(A_{r, R})} \leq C R^2\log (R/r) \|du\|^2_{L^2(A_{r, R})}$.
%\item $\int_{A_{r, R}} u^2 \leq R^2 \log(R/r) \int_{A_{r, R}} | du |^2$. 
\end{enumerate}
The same estimates in (ii) and (iii) hold with a modified constant $C$ for any family of domains uniformly bi-Lipschitz to $A_{r,R}$.  

Finally, if $u \in W^{1,2}(A^+_{r, R})$ and $u|_{\partial D^+_r} = 0$, where $A^+_{r, R} : = A_{r, R} \cap \R^2_+$ and  $D^+_s := D_s \cap \R^2_+$, then (ii)-(iii) hold with $A_{r, R}$ and $\partial D_s$ replaced by $A^+_{r, R}$ and $\partial D^+_s$, respectively. 
\end{lemma}
\begin{proof}
Set $Q(s)= \frac{1}{s} \int_{\partial D_s} u^2$.  Calculating and estimating shows that
\begin{align}
\label{ineq:lvlsetder}
Q'(s) = \frac{2}{s} \int_{\partial D_s} u \frac{\partial u}{\partial \nu} 
\leq 2 Q(s)^{1/2} \left(\frac{1}{s}\int_{\partial D_s}\left|\frac{\partial u}{\partial \nu}\right|^2\right)^{1/2},
\end{align}
where the inequality follows from the Cauchy-Schwarz inequality.  Item (i) follows.  Item (ii) follows from integrating the inequality in (i) over $[r, R]$, applying the Cauchy-Schwarz inequality to the right-hand side, and squaring, and item (iii) follows by integrating the inequality in (ii) over $[r, R]$.  

Finally, the estimates on the half-annulus $A^+_{r, R}$ follow from (i)-(iii) by a simple reflection argument. 
\end{proof}

\begin{lemma}
\label{Lsphlower}
If $X \subset \Sph^2$ is a finite set and $R>0$ is such that
\begin{enumerate}[label=\emph{(\alph*)}]
\item the disks $\{ D_R(x)\}_{x \in X}$ are pairwise disjoint and 
\item the disks $\{ D_{6R}(x)\}_{x \in X}$ cover $\Sph^2$,
\end{enumerate}
then for any $r<R/2$,  $\lambda^D_1(\Omega)^{-1} \leq C R^2 \log (R/ r)$, where $\Omega = \Sph^2 \setminus \cup_{x \in X} D_r(x)$. 
\end{lemma}
\begin{proof}
Let $\{V_x\}_{x \in X}$ be the Voronoi tesselation with respect $X$, defined by 
\begin{align*}
V_x = \{y \in \Sph^2 : d(y, x) \leq d(y, x') \,  \, \forall  x' \in X, x'\neq x\}.
\end{align*}
for each $x \in X$.  As long as $|X| \geq 2$, note each $V_x$ is convex and satisfies $D_{R}(x) \subset V_x \subset D_{6R}(x)$.  It is not difficult to see that each $\Omega \cap V_x$, $x\in X$ is bi-Lipschitz to the annulus $A_{r, R}$, with Lipschitz constants independent of $x, r, R$.

Consequently, if $u \in W^{1,2}_0(\Omega)$, Lemma \ref{Lmono} implies for each $x \in X$ that
\begin{align*}
\| u \|^2_{L^2(\Omega \cap V_x)} \leq C R^2 \log (R/r) \| du \|^2_{L^2(\Omega \cap V_x)},
%\int_{\Omega \cap V_x} u^2 \leq C R^2 \log (R/ r) \int_{\Omega \cap V_i} |du|^2,
\end{align*}
and summing over $i$ gives
\begin{align*}
\| u \|^2_{L^2(\Omega)} \leq C R^2 \log (R/r) \| du \|^2_{L^2(\Omega)},
\end{align*}
proving the desired estimate. 
\end{proof}

\begin{lemma}
\label{Lmonobd}
Let $K \subset \Sph^2$ be a convex polygonal domain, $X\subset \partial K$ be a finite subset, and $0<R<R_0(K)$ be such that 
\begin{enumerate}[label=\emph{(\alph*)}]
\item $d(x, y) \geq R$ for distinct $x,y \in X$; and 
\item The disks $\{ D_{10R}(x)\}_{x \in X}$ cover $\partial K$.
\end{enumerate}
For any $r < R/10$, let $U = K \setminus D_r(X)$, let $E_{t} : = U \cap D_{20t}(\partial K)$ be a tubular neighbourhood of $\bd K$ in U, and $F_{t} : = U \cap \partial D_{20t}(\partial K)$ be its interior boundary.  If $u \in W^{1,2}(U)$ and $u|_{\partial D_r(X)} = 0$, there exists a constant $C=C_K$ depending only on $K$ such that
\begin{enumerate}[label=\emph{(\roman*)}]
\item $\| u\|^2_{L^2(F_R)} \leq C R \log (R/r) \| du \|^2_{L^2(E_R)}$,
\item $\| u\|^2_{L^2(\partial K)} \leq C R \log (R/r) \| du \|^2_{L^2(E_R)}$,
\item $\|u\|^2_{L^2(E_R)} \leq C R^2 \log (R/r) \| du \|^2_{L^2(E_R)}$. 
\end{enumerate}
\begin{proof}
If $u \in W^{1, 2}(A^+_{r, R})$ and $u = 0$ on $\partial D^+_r$, then Lemma \ref{Lmono} implies
\begin{align}
\label{Eestp}
\| u\|^2_{L^2(\partial D^+_s)}  \leq C R \log (R/r) \| du \|^2_{L^2(A^+_{r, R})}
\end{align}
for a.e. $s\in [r,R]$. By the hypothesis of the lemma, $E_{R}=U\cap D_{20R}(\partial K)$ can be decomposed into a union $\cup_{x \in X} E_{R, x}$, where $E_{R, x}=E_R\cap V_x$ is uniformly bi-Lipschitz to $A^+_{r, R}$, so if $u \in W^{1,2}(U)$ with $u |_{ \partial U \setminus \partial K} = 0$, we may apply \eqref{Eestp} on each $E_{R, x}$ to see that 
\begin{align*}
\| u\|^2_{L^2(\partial E_{R, x} \setminus \partial K)} \leq C R \log (R/r) \| u\|^2_{L^2( E_{R, x})}.
\end{align*}
Since $F_R \subset \cup_{x \in X} \partial E_{R, x} \setminus \partial K$, summing over $x \in X$ then gives (i).  

Similarly to the proof of~\eqref{ineq:lvlsetder}, we obtain that for $s<R_0(K)$ sufficiently small 
\[
\left|\frac{d}{ds}\left(\int_{F_s} u^2\right)^{1/2}\right|\leq C\left(\int_{F_s} |du|^2\right)^{1/2}
\]
Integrating on $s\in [0,R]$, rearranging and taking squares yields
\begin{equation}
\label{ineq:rect_der}
 \int_{\partial K} u^2\leq C R \int_{E_R} |du|^2 + C\int_{F_R} u^2,
\end{equation}
which together with the estimate in (i) proves the estimate (ii).  Similarly, item (iii) follows from Lemma \ref{Lmono}.
\end{proof}
\end{lemma}

\begin{lemma}
Let $K \subset \Sph^2$ be a convex geodesic polygon with piecewise smooth geodesic boundary.  Then there exists a constant $C(K)> 0$ such that for any $u \in W^{1,2}(K)$, 
\begin{align*}
\lambda^D_1(K) \| u\|^2_{L^2(K)} \leq \| du\|^2_{L^2(K)} + C \| u\|_{L^2(\partial K)} \| du \|_{L^2(K)} + C \| u\|^2_{\partial K}.
\end{align*}
\end{lemma}
\begin{proof}
This follows easily by decomposing $u$ into $u = u_0+ u_h$ where $u_0 \in W^{1,2}_0(K)$ and $u_h$ is the harmonic extension of $u$, and noting that any harmonic function $\phi : K \rightarrow \R$ satisfies an estimate of the form
\[
\| \phi \|_{L^2(K)} \leq C(K) \| \phi\|_{L^2(\partial K)}.\qedhere
\]
\end{proof}

Combining this with Lemma \ref{Lmonobd}, we arrive at the following. 
\begin{prop}
\label{Pbd1}
Let $K \subset \Sph^2$, $R< R_0(K)$, and $U = K \setminus \cup_{x \in X} D_r(x)$ be as in Lemma \ref{Lmonobd}.  Then for each $u \in W^{1,2}(K)$ vanishing on $K \setminus U$, we have
\begin{align*}
\lambda_1^D(K) \int_{K} u^2 \leq (1+ C\sqrt{R\log (R/r)}) \int_{K} |du|^2.
\end{align*}
\end{prop}

\subsection{Application to symmetric domains}

We recall some definitions and notation from \cite[Section 5]{KKMS}.  A \emph{reflection} on $\Sph^2$ is a smooth involution $\tau: \Sph^2 \rightarrow \Sph^2$ whose fixed-point set $(\Sph^2)^\tau$ separates $\Sph^2$.  If $G$ is a finite group acting properly, smoothly, and effectively on $\Sph^2$ and $G$ is generated by reflections, $G$ is called a \emph{reflection group} on $\Sph^2$.  Given $x \in \Sph^2$ denote by $R(x)$ the set of all reflections in $G$ fixing $x$.  A point $x \in \Sph^2$ is called \emph{nonsingular} if $R(x) = \emptyset$.  A \emph{chamber} of $G$ on $\Sph^2$ is the closure of a connected component of nonsingular points. 

The classification of reflection groups on $\Sph^2$ is well-known \cite[page 53]{Conway}; they are isomorphic to finite subgroups of $O(3)$ generated by reflections and consist of: the \emph{tetrahedral}, \emph{octahedral}, and \emph{icosahedral} groups corresponding to the symmetries of the Platonic solids, the prismatic dihedral groups $\Z_2\times D_k$, the dihedral group $D_k$, $\Z_2$, and the trivial group. 

Given such a group $G$,
The fixed-point sets of all reflections in $G$ induce a tesselation of $\Sph^2$ by fundamental chambers whose edges are geodesic segments.   

\begin{definition} 
\label{dtype}
Let $G$ be a reflection group on $\Sph^2$ and $\Omega \subset \Sph^2$ be a $G$-invariant domain.  The \emph{type} associated to $\Omega$ is the formal linear combination
\begin{align*}
b = f + \sum_i e_i \rho_i + \sum_{i<j} v_{ij} \rho_i \rho_j
\end{align*}
where $f, e_i$, and $v_{ij}$ are the number of components of $\partial \Omega$ meeting a chamber $V$ satisfying respectively $\mathrm{Stab}_G(S) = 1$, $\mathrm{Stab}_G(S) = \langle \rho_i\rangle$, and $\mathrm{Stab}_G(S) = \langle \rho_i, \rho_j\rangle$.  The indices in the sums are understood to run over the indices corresponding to the generators for $G$.
\end{definition}

In other words, if $V$ is a chamber for the action of $G$ on $\Sph^2$, a domain of type $b$ can be constructed by removing a $G$-invariant set $\Dcal$ of small geodesic disks with the properties that $f$ are centered in the interior of $V$, $e_i$ are centered along the edge of $V$ fixed by $\rho_i$, and $v_{ij}$ are centered at the intersection of the edges fixed by $\rho_i$ and $\rho_j$.  

As a first step in the proof of Theorem \ref{s3.lbd}, we enumerate all of the types for each reflection group $G$ whose associated basic reflection pair $(M,\Gamma)$ is of Scherk type, and observe that every such pair automatically satisfies $\frac{1}{2}\Lambda_1(M,\Gamma)\geq \Area(\xi_{\gamma,1})$.

\begin{lemma}\label{scherk.lem}
For the following groups $G$ and types $b$, the associated basic reflection pair $(M,\Gamma)$ is of Scherk type.
\begin{itemize}
\item $G=\mathbb{Z}_2$ of type $b=m\rho_1$ for any integer $m\geq 2$.
\item $G=D_k$ of type $b=1$ or $b=\rho_1$ or $b=\rho_2$.
\item $G=\mathbb{Z}_2\times D_k$ of type $e_3\rho_3+v_{13}\rho_1\rho_3+v_{23}\rho_2\rho_3$ for any $e_3\in \mathbb{N}\cup \{0\}$ and $v_{13},v_{23}\in \{0,1\}$, where $\rho_3$ is the reflection generating the $\mathbb{Z}_2$ factor. Moreover, in each of these cases, we have
$$\Lambda_1(M,\Gamma)\geq 2\Area(\xi_{\gamma,1}),$$
where $\gamma$ is the genus of $M$.
\end{itemize}
\end{lemma}
\begin{proof}
Let $(M,\Gamma=\langle \tau\rangle \times G)$ be a basic reflection surface associated to one of these types equipped with a $\Gamma$-invariant metric $g$, and let $N \subset M$ be one of the components of $M \setminus M^\tau$, so that $N$ is a fundamental domain for $G$, acting by isometries of $g$. To verify that Definition \ref{scherk.def} is satisfied, we simply need to check that $(N,g)$ is conformally equivalent to the complement $\mathbb{S}^2\setminus \mathcal{D}$ of a union of geodesic disks $\mathcal{D}$ in the standard sphere $\mathbb{S}^2$ with centers on the equator.

To this end, recall that by a standard variant of the uniformization theorem for compact surfaces with boundary (cf. \cite[Theorem 2.1]{KSDuke}), there exists a union of geodesic disks $\mathcal{D}\subset \mathbb{S}^2$ whose complement $\Omega=\mathbb{S}^2\setminus \mathcal{D}$ admits a conformal map $F: \Omega\to N$, and this pair $(F,\mathcal{D})$ is unique up to conformal automorphisms of $\mathbb{S}^2$. In particular, it follows that $G$ acts on $\mathbb{S}^2$ by conformal automorphisms such that $G$ preserves $\mathcal{D}$ and $F$ is equivariant. 

Moreover, it is straightforward to check that, for each reflection group $G$ considered here, any action of $G$ on $\mathbb{S}^2$ by conformal automorphisms must in fact be conjugate to the standard action by isometries of the associated type, and therefore the domain $\Omega=\mathbb{S}^2\setminus \mathcal{D}$ is a domain of the prescribed type with respect to the standard isometric action of $G$ on $\mathbb{S}^2$.

To complete the proof that the given configurations are of Scherk type, we note that when $G=\mathbb{Z}_2$ and $\Omega$ is of type $me_1$, then all of the disks $\mathcal{D}$ must be invariant under a reflection, and therefore have centers along the equatorial fixed point set. Likewise, when $G=D_k$ and $\Omega$ is of type $1$, $\rho_1,$ or $\rho_2$, then $\mathcal{D}$ consists of the $D_k$-orbit of a single disk; after applying a conformal dilation centered at a fixed point $p$ of the $D_k$ action, the center of this disk can be taken without loss of generality to lie on the equator of distance $\pi/2$ from $p$, in which case the full $D_k$-orbit $\mathcal{D}$ lies on this equator as well. Finally, when $G=\mathbb{Z}_2\times D_k$ and $\Omega$ is of type $e_3\rho_3+v_{13}\rho_1\rho_3+v_{23}\rho_2\rho_3$, by forgetting the $D_k$ factor, we see that $\Omega$ is a domain of type $(2ke_3+kv_{13}+kv_{23})\rho_3$ for the $\mathbb{Z}_2$ action generated by $\rho_3$, which we have already seen is of Scherk type.

Finally, it is straightforward to check that for each of the symmetry types $(M,\Gamma)$ listed above, there is a diffeomorphism $\phi: M\to \xi_{\gamma,1}$ to the Lawson surface of the corresponding genus, such that $\phi^*g_{\mathbb{S}^3}\in \Met_{\Gamma}(M)$. Combining this observation with the work of Choe-Soret \cite{ChoeSoret}, it follows that
$$\Lambda_1(M,\Gamma)\geq \bar{\lambda}_1(\xi_{\gamma,1})=2\Area(\xi_{\gamma,1}),$$
as claimed.
\end{proof}

To establish the lower bounds in Theorem \ref{s3.lbd}, by Lemma \ref{double.lem}, it remains to show that $\mathcal{M}_1(\Omega,G)\geq 8\pi-e^{-c\sqrt{k}}$ for all the reflection groups $G$ and types not included in Lemma \ref{scherk.lem}, whenever the number of boundary components $k$ of $\Omega$ is sufficiently large. More precisely, we will establish this lower bound in all of the remaining cases:
\begin{itemize}
\item $G=1$, $\Omega$ of any type,
\item $G$ one of the platonic groups, $\Omega$ of any type,
\item $G=\mathbb{Z}_2$, $\Omega$ of type $f+e_1\rho_1$ with $f>0$,
\item $G=\mathbb{Z}_2\times D_k$, $\Omega$ of type $f+\sum_i e_i\rho_i+\sum_{i<j}v_{ij}\rho_i\rho_j$ where $f+e_1\rho_1+e_2\rho_2+v_{12}\rho_1\rho_2\neq 0,$
\item $G=D_k$, $\Omega$ of type $f+e_1\rho_1+e_2\rho_2+v_{12}\rho_1\rho_2$, where $f+e_1+e_2+v_{12}\geq 2$.
\end{itemize}

As a first step, we obtain a lower bound of the form $\mathcal{M}_1(\Omega,G)\geq 8\pi-e^{-ck}$ in the case $G=1$, or more generally, whenever the number $f$ of `free' holes in the interior of a chamber is large relative to the order $|G|$.

\begin{lemma}
\label{Lfdense}
There is a universal constant $C$ such that for each reflection group $G$ on $\Sph^2$ and each number $f_0 \geq C| G|$, there exists a $G$-invaraint domain $\Omega \subset \Sph^2$ of type $f$ with $f \in [C^{-1} f_0, Cf_0]$ satisfying the hypotheses of Lemma \ref{Lsphlower} with $R = 1/ \sqrt{ f_0 | G|}$.
\end{lemma}
\begin{proof}
Let $V$ be a chamber for the action of $G$, and fix $R<c/|G|$ for a constant $c> 0$.  By choosing $c>0$ small enough in absolute terms, the Vitali covering lemma implies there exists a finite set $X_R \subset V \setminus D_R(\partial V)$ such that 
\begin{enumerate}[label={(\alph*)}]
\item the disks $\{ D_R(x)\}_{x \in X_R}$ are pairwise disjoint and contained in $V$; and
\item the disks $\{D_{5R}(x)\}_{x \in X_R}$ cover $ V \setminus D_R(\partial V)$.
\end{enumerate}
By the triangle inequality, the disks $\{D_{6R}(x)\}_{x \in X_R}$ then cover $V$, so that
\begin{align*}
C |X_R| R^2 \geq |V| = \frac{4\pi}{|G|},
\end{align*}
while condition (a) implies a matching upper bound
\begin{align*}
|X_R| R^2 \leq \frac{C'}{|G|}.
\end{align*}
In particular, fixing $R = 1/\sqrt{f_0|G|}$ for $f_0> C | G|$, it follows that $|X_R| \in [C^{-1} f_0, Cf_0]$ for a suitable constant $C$. 

Letting $\Dcal$ be the orbit under $G$ of $D_r(X_R)$ for $r< R/2$ and taking $\Omega=\mathbb{S}^2\setminus \Dcal$ then gives the $G$-invariant domain satisfying the desired conditions. 
\end{proof}

Using this, we can show the following.

\begin{prop}
\label{Pfmany}
There exist universal constants $C, c \in (0, \infty)$ such that for each reflection group $G$ and each type $b = f + \sum_i e_i \rho_i + \sum_{i<j} v_{ij} \rho_i \rho_j $ with $f\geq C| G|$, there exists a $G$-invariant domain $\Omega \subset \Sph^2$ of type $b$ such that 
\begin{align*}
\lambda^D_1(\Omega) \geq 2
\quad
\text{and} 
\quad
|\Dcal| \leq C \sqrt{f | G| } e^{-c f | G|}.
\end{align*}
In particular, there is a constant $c'\in (0,\infty)$ such that
$$\mathcal{M}_1(\Omega,G)\geq 8\pi-e^{-c'f|G|}$$
whenever $f\geq C|G|$.
\end{prop}
\begin{proof}
Applying Lemma \ref{Lfdense} with $C^{-1} f$ in place of $f_0$, we see that there is a $G$-invariant domain $\Omega'= \Sph^2 \setminus \Dcal'$ of type $f' \in [C^{-2}f, f]$ satisfying the hypotheses of Lemma \ref{Lsphlower} with $R = C'/ \sqrt{f|G|}$ and any $r<R/2$.  Applying Lemma \ref{Lsphlower}, we then see that
\begin{align*}
\lambda^D_1(\Omega') \geq \frac{1}{CR^2 \log (R/r)} = 2
\end{align*}
by taking
\begin{align*}
r = R e^{-\frac{1}{2CR^2}}= \frac{C'}{\sqrt{f|G|}} e^{-c f |G|},
\end{align*}
so that 
\begin{align*}
|\Dcal'| \leq C'' \sqrt{f |G|} e^{-2cf|G|}.
\end{align*}
Since $f' \leq f$, by punching additional holes of arbitrarily small area as needed, we obtain a $G$-invariant domain $\Omega \subset \Omega'$ of type $b$ for which $\lambda^D_1(\Omega) \geq \lambda^D_1(\Omega') \geq 2$ and $|\Dcal|$ is as close as desired to $|\Dcal'|$, completing the proof of these bounds. Finally, note that since $\lambda_1^D(\Omega)\geq 2$ and, by Proposition \ref{Pneu}, 
$$\lambda_1^N(\Omega)\geq 2-C|\Dcal|\geq 2-C'\sqrt{f|G|}e^{-cf|G|},$$
we deduce that
$$\bar{\mu}_1(\Omega)\geq (4\pi-C\sqrt{f|G|}e^{-cf|G|})(2-C'\sqrt{f|G|}e^{-cf|G|})\geq 8\pi-e^{-c'f|G|}$$
for a suitable constant $c'<c$. Since $\mathcal{M}_1(\Omega,G)\geq \bar{\mu}_1(\Omega)$, this completes the proof.
\end{proof}

Next, we focus on $G$-invariant domains $\Omega \subset \Sph^2$ with type $b$ for which the excised disks are arranged along one-dimensional graphs in $\Sph^2$ corresponding to the orbit of one or more fixed-point sets $(\Sph^2)^{\rho_i}$ under the action of $G$.  Together with Proposition \ref{Pfmany}, this will yield the desired lower bound whenever $G$ is one of the tetrahedral, octahedral, or icosahedral groups.

\begin{prop}
\label{Pedges}
Let $G$ be one of the Platonic groups.  Then provided $\sum_i e_i$ is large enough in absolute terms, there exists a $G$-invariant domain $\Omega = \Sph^2 \setminus \Dcal$ of type $b = f+ \sum_{i} e_i \rho_i + \sum_{i<j} v_{ij} \rho_i \rho_j$ such that
\begin{align*}
\lambda^D_1(\Omega) \geq 2
\quad{and} 
\quad
|\Dcal| \leq \frac{C}{\sum_i e_i} e^{-C \sum_i e_i}.
\end{align*}
\end{prop}
\begin{proof}
Up to changing the constant $C$ by a factor of $\frac{1}{3}$, by arguing as in the proof of Proposition \ref{Pfmany}, it suffices to prove the Proposition in the special case where $b = e_i \rho_i$ for some $i \in \{1,2,3\}$. 

Denote by $S_i \subset \Sph^2$ the geodesic graph whose edges correspond to the orbit of $(\Sph^2)^{\rho_i}$ under the action of $G$, and observe that $S_i$ partitions $\Sph^2$ into a union of less than $|G|$ copies of a convex polygon $K = K_{G, i}$ strictly contained in a hemisphere bounded by the fixed-point set $(\Sph^2)^{\rho_i}$.   In particular, since there are finitely many such groups $G$, there is a universal lower bound
\begin{align*}
\lambda^D_1( K ) \geq (1+\delta) 2
\end{align*}
where $\delta>0$ is independent of $G$ and $i \in \{1,2,3\}$.

We can then get a domain $\Omega$ of type $e_i \rho_i$ by removing disks of radius $r$ at evenly spaced points along each edge of $S_i$.  With $K$ as above, the intersection $U = \Omega \cap K$ satisfies the hypotheses of Lemma \ref{Lmonobd} with $R= C/ e_i$, and it follows from Proposition \ref{Pbd1} that
\begin{align*}
\lambda^D_1(\Omega) \geq \frac{\lambda^D_1(K)}{1+ C \sqrt{R\log (R/r)}} \geq 2 \frac{1+\delta}{1+ C\sqrt{R \log(R/r)}}.
\end{align*}
In particular, taking $r = \frac{1}{e_i}e^{-ce_i}$ for a suitable small constant $c>0$ gives $\lambda^D_1(\Omega)>2$, while
\begin{align*}
|\Dcal| \leq \frac{C}{e_i} e^{-2ce_i},
\end{align*}
as desired. 
\end{proof}

\begin{cor}\label{plat.cor}
There is a constant $c >0$ such that if $G$ is trivial or one of the three platonic groups, and $\Omega_0 \subset \Sph^2$ is a $G$-invariant domain with $m$ boundary components, then $\Omega_0$ is $G$-equivariantly homeomorphic to a domain $\Omega$ with
\begin{align*}
\lambda^D_1(\Omega) \geq 2
\quad
\text{and}
\quad
|\Sph^2 \setminus \Omega | \leq e^{-c m}.
\end{align*}
In particular, $\mathcal{M}_1(\Omega,G)\geq 8\pi-Ce^{-c m}.$
\end{cor}
\begin{proof}
This follows by combining Proposition \ref{Pfmany} in the case where $G$ is a Platonic group and $m\leq Cf$ with Proposition \ref{Pedges} when $m\leq C\sum_i e_i$, together with Proposition \ref{Pneu} and the observation that $\sum_{i<j}v_{ij}$ is bounded independent of $m$. 
\end{proof}

As another application of Proposition \ref{Pbd1}, we can also obtain bounds of the desired form for $D_k$ symmetric domains of type $f+e_1\rho_1+e_2\rho_2$ and $\mathbb{Z}_2\times D_k$-symmetric domains of type $f+\sum_ie_i\rho_i+\sum_{i<j}v_{ij}\rho_i\rho_j$ with $\max\{e_1,e_2\}$ sufficiently large relative to $k$. 

\begin{prop}\label{dk.edges}
If $\Omega$ is a $G=D_k$-invariant domain of type $f+e_1\rho_1+e_2\rho_2$ or a $G=\mathbb{Z}_2\times D_k$-invariant domain of type $f+\sum_ie_i\rho_i+\sum_{i<j}v_{ij}\rho_i\rho_j$ with $\max\{e_1,e_2\}=n\geq k$, then
$$\mathcal{M}_1(\Omega,G)\geq 8\pi-e^{-cn}$$
for some constant $c>0$ independent of $k$ and $n$.
\end{prop}
\begin{proof}
For the standard action of $D_k$ on $\mathbb{S}^2$, let $W_k\subset \mathbb{S}^2$ denote one of the $k$ wedges formed by reflecting the fundamental domain over $\Fix(\rho_2)$, so that $W_k$ is bounded by two geodesic segments in the orbit of $\Fix(\rho_1)$, meeting at the poles with an angle of $2\pi/k$. 

Suppose for the moment that $k\geq 11$, and let $K_1$ be the wedge made up of $\lfloor \frac{k}{3}\rfloor$ copies of $W_k$, and let $K_2$ be the wedge made up of $k-2\lfloor \frac{k}{3}\rfloor$ copies of $W_k$, so that $\mathbb{S}^2$ can be partitioned into two copies of $K_1$ and one copy of $K_2$, where each $K_i$ is bounded by two geodesic segments in $\Fix(\rho_1)$ meeting at the poles with an angle in the interval $[\frac{2\pi}{3}-\frac{4\pi}{3k},\frac{2\pi}{3}+\frac{8\pi}{3k}]$. In particular, we see that the domains $K_1$ and $K_2$ satisfy the hypotheses of Lemma \ref{Lmonobd} and Proposition \ref{Pbd1} with $\lambda_1^D(K_i)\geq 2+\delta_0$ for some $\delta_0>0$, such that the constant $C=C_{K_i}$ is bounded independent of $k$.

Now, we can construct a $D_k$-invariant domain $\Omega$ of type $f+e_1\rho_1+e_2\rho_2$ by removing the orbit of $n=e_1$ evenly spaced small disks $D_r(x_1),\ldots,D_r(x_n)$ with centers on the intersection of $\Fix(\rho_1)$ with a fundamental domain, so that the orbit of $D_{10R_n}(x_1)\cup \cdots\cup D_{10R_n}(x_n)$ contains the orbit of $\Fix(\rho_1)$ for $R_n=\frac{C}{n}$, then removing an additional $f$ arbitrarily small disks from the interior of the fundamental domain and $e_2$ arbitrarily small disks from $\Fix(\rho_2)$ in the fundamental domain. For $n$ sufficietly large, we then see that $K_1$ and $K_2$ satisfy the hypotheses of Lemma \ref{Lmonobd} and Proposition \ref{Pbd1} with $\Omega\cap K_i\subset U$ and $R_n=\frac{C}{n}$, so that an application of Proposition \ref{Pbd1} gives
$$(2+\delta_0)\int_{K_i}u^2\leq (1+C\sqrt{\log(1/nr)/n})\int_{K_i}|du|^2$$
for all $u$ vanishing on $\mathbb{S}^2\setminus\Omega$, and therefore
$$\lambda_1^D(\Omega)\geq (1+C\sqrt{\log(1/nr)/n})^{-1}(2+\delta_0),$$
with $C$ and $\delta_0$ independent of $n$ and $k$. Taking $r=e^{-cn}$ for $c$ sufficiently small then gives $\lambda_1^D(\Omega)\geq 2$, while $|\mathbb{S}^2\setminus\Omega|\leq Ck ne^{-cn}\leq e^{-c_1n}$ for $n\geq k$, and therefore
$$\bar{\mu}_1(\Omega)\geq 8\pi-e^{-c_1n},$$
by Proposition \ref{Pneu}. In particular, this gives a bound of the desired form
$$\mathcal{M}_1(\Omega,D_k)\geq 8\pi-e^{-cn}$$
for $D_k$-invariant domains of the form $f+e_1\rho_1+e_2\rho_2$ with $\max\{e_1,e_2\}=n\geq k$.

The same argument--with an additional reflection symmetry enforced while choosing the location of the disks $D_r(x_i)$--yields a lower bound of the form
$$\mathcal{M}_1(\Omega,\mathbb{Z}_2\times D_k)\geq 8\pi-e^{-cn}$$
for domains $\Omega$ of type $f+\sum_ie_i\rho_i+\sum_{i<j}v_{ij}\rho_i\rho_j$ with $\max\{e_1,e_2\}=n\geq k\geq 11$ as well.

Finally, we remark that the remaining cases $k=3,4,\ldots,10$ are easily handled by an identical argument, applying Proposition \ref{Pbd1} with $K$ one of the $k\leq 10$ wedges $W_k$.
\end{proof}

The estimates from the preceding proposition become useful when a majority of the boundary components of $\Omega$ lie along the fixed point sets of $\rho_1$ and $\rho_2$, so that the total number $m$ of boundary components satisfies $m\leq Ckn$. In this case, note that $kn\leq n^2$, so that $e^{-cn}\leq e^{-c'\sqrt{m}}$. By a slight refinement of the arguments above, the estimate $\mathcal{M}_1(\Omega, G)\geq 8\pi-e^{-c\sqrt{m}}$ could probably be improved to one of the form $\mathcal{M}_1(\Omega,G)\geq 8\pi-e^{-cm}$ in this case. We do not pursue this improvement, though, since it is clear that a lower bound of the form $8\pi-Ce^{-c\sqrt{m}}$ is sharp for some other basic reflection surfaces of generic type, as in the case treated by the following lemma, where all but two of the holes are arranged along the equator.

\begin{lemma}
\label{Leqpol}
There exists $c>0$ such that if $k \in \N$ is large enough and $\Omega = \Sph^2 \setminus D_r(X)$, where 
\begin{align*}
X = \{p_N, - p_N, x_1, \dots, x_k\},
\quad 
r = e^{-c \sqrt{k}},
\end{align*}
$p_N$ is the north pole, and $\{x_j\}$ are arranged in a symmetric way along the equator of $\Sph^2$, then we have $\lambda^D_1(\Omega) \geq 2$. 
\end{lemma}
\begin{proof}
Applying Lemma \ref{Lmonobd} with $K$ the hemisphere $\Sph^2_+ := D_{\pi/2}(p_N)$ and $R = 2\pi /k$, for any $u \in W^{1,2}_0(\Omega)$ we see that
\begin{align}
\label{Eur1}
\| u\|^2_{L^2(F_R)} \leq C c k^{-1/2} \| du \|^2_{L^2(E_R)}
\end{align} 
and
\begin{align}
\label{Eur2}
\| u\|^2_{L^2(E_R)} \leq Cc k^{-3/2}  \| du \|^2_{L^2(E_R)},
\end{align}
where $F_R := K\cap\Omega \cap \partial D_{20R}(\partial K) = \partial D_{\pi/2- 20R}(p_N)$ and $E_R = K\cap\Omega \cap D_{20R}(\partial K)$.  

Now let  %Note that $F_R$ and $D_R(\partial K)$ are small perturbations of $\{x_3 = R\}$ and $\{-R \leq x_3 \leq R\}$, respectively. Now let
\begin{align*}
V := K\cap \Omega \setminus E_R = D_{\pi/2- 20R}(p_N) \setminus D_r(p_N) % \{R \leq x_3 \leq \sqrt{1-r^2}\}
\end{align*}
and define
\begin{align*}
Q(V) : = \inf \left\{ \left.\frac{ \int_V | d\phi|^2 - 2 \phi^2}{\| \phi\|^2_{L^2(F_R)}}\right| \phi|_{\partial D_{r}(p_N)}= 0\right\}.
\end{align*}
Since the first Dirichlet eigenvalue of $V$ is $>2$, it's not hard to see that $Q(V) > - \infty$, and the infimum must be realized by a radially symmetric function of the form $\phi = f (d(p_N, \cdot))$, solving
\begin{align*}
\begin{gathered}
\csc t \frac{d}{dt}[ \sin t f'(t) ] + 2f(t) = 0,\\
f(r) = 0, \quad f(\pi/2 -20R) = 1,
\end{gathered}
\end{align*}
for which
\begin{align*}
\int_V |d\phi|^2 - 2\phi^2 = \int_{\partial V} \phi \frac{\partial \phi}{\partial \nu}=  -|F_R| f'(\pi/2- 20R)
\end{align*}
and $\int_{F_R} \phi^2 =  |F_R| $, so that
\begin{align}
\label{EQV}
Q(V) = - f'(\pi/2-20R).
\end{align}
A direct computation %using that $\Delta f = \csc t \, \partial_t ( \sin t f'(t))$ 
shows that $f$ must have the form
\begin{align*}
f(t) = a \cos t + b \left( 1 + \cos t \log  \tan \frac{t}{2} \right)  %at + b [ t \log (1+t) - t\log (1-t) -2]
\end{align*}
for suitable $a$ and $b$, so that  
\begin{align*}
f'(t) = - a \sin t + b \left( \cot t - \sin t \log \tan \frac{t}{2}\right),
\end{align*}
and hence
\begin{align*}
f'(\pi/2-20R) &= a \cos 20R + b \cdot  O(R).
\end{align*}

Substituting the values for $f$ at $r$ and $\pi/2-R$ yields
\begin{align*}
0 &= a ( 1+ O(r^2) ) + b ( \log r + O(1)) ,\\
1 &= a \sin 20R + b (1+ O(R^2)), 
\end{align*}  
and simplifying using the definitions shows that
\begin{align*}
a = - \log r + O(1), 
\quad
b = 1+ O(\sqrt{R}). 
\end{align*}
Returning to \eqref{EQV}, we see using the definitions of $R$ and $r$ that
\begin{align*}
Q(V) \geq C \log r \geq -C c \sqrt{k}.
\end{align*}

The preceding estimates and \eqref{Eur1} together imply 
\begin{align*}
\int_{K\cap \Omega \setminus E_R} |du|^2 - 2u^2 \geq C Q(V) \|u\|^2_{L^2(F_R)} \geq -Cc \int_{E_R} |du|^2.
\end{align*}
Since $K\cap \Omega = (K\cap \Omega \setminus E_R) \cup E_R$, this together with \eqref{Eur2} and the analogous inequality for the southern hemisphere gives
\begin{align*}
\int_{\Omega} |du|^2 - 2u^2 \geq 
(1- C c k^{-3/2} -Cc ) \int_{E_R} |du|^2.
\end{align*}
By choosing $c>0$ small enough, the term in the parentheses is positive, and hence $\lambda^D_1(\Omega) \geq 2$. 
\end{proof}

As a corollary, we obtain the desired lower bounds for all remaining types in the case $G=\mathbb{Z}_2$, and some additional types when $G=D_k$ or $\mathbb{Z}_2\times D_k$.

\begin{prop}\label{eq.poles}
There is a constant $c\in (0,\infty)$ such that a lower bound of the form
$$\mathcal{M}_1(\Omega,G)\geq 8\pi-e^{-c\sqrt{m}}$$
holds in each of the following cases, where $m$ is the number of boundary components of $\Omega$: 
\begin{itemize}
\item $G=\mathbb{Z}_2$ and $\Omega$ is of type $f+e_1\rho_1$ with $f>0$; 
\item $G=D_n$ and $\Omega$ is of type $f+e_1\rho_1+e_2\rho_2+2\rho_1\rho_2$ with $\max\{f,e_1,e_2\}= 1$; 
\item $G=\mathbb{Z}_2\times D_n$ and $\Omega$ is of type $f+\sum_i e_i\rho_i+\sum_{i<j}v_{ij}\rho_i\rho_j$, where $f+e_1+e_2+v_{12}\geq 1$ and $m\leq C n(e_3+v_{13}+v_{23})$. 
\end{itemize}
\end{prop}
\begin{proof}
For $G=\mathbb{Z}_2$, when $m=2f+e_1\leq 4f$, the stronger bound $\mathcal{M}_1(\Omega,G)\leq 8\pi-e^{-cm}$ already follows from Proposition \ref{Pfmany}, so it suffices to consider the case $m\leq 2e_1$. In this case, we can obtain a domain of the desired type by applying Lemma \ref{Leqpol} with $k=e_1$ and adding $f-1$ additional arbitrarily small holes on either side of the equator in a symmetric way, yielding a $\mathbb{Z}_2$-symmetric domain $\Omega$ of type $f+e_1\rho_1$ with $\lambda_1^D(\Omega)\geq 2$ and $|\mathbb{S}^2\setminus \Omega|\leq Ce_1 e^{-c\sqrt{e_1}}$, which together with Proposition \ref{Pneu} gives a bound of the form
$$\mathcal{M}_1(\Omega,\mathbb{Z}_2)\geq \bar{\mu}_1(\Omega)\geq 8\pi-Cme^{-c\sqrt{m/2}},$$
giving the desired bound after replacing $c$ with a slightly smaller constant.

Similarly, when $\max\{f,e_1,e_2\}=1$, taking $k=n(2f+e_1+e_2)$, we can realize the domain of Lemma \ref{Leqpol} as a $D_n$-symmetric domain of type $f+e_1\rho_1+e_2\rho_2+2\rho_1\rho_2$ with $\bar{\mu}_1(\Omega)\geq 8\pi-e^{-c'\sqrt{k}}$, giving a bound of the desired form in this case, since $m=k+1$. 

Likewise, taking $k=n(2e_3+v_{13}+v_{23})$, we can realize the domain of Lemma \ref{Leqpol} as a $\mathbb{Z}_2\times D_n$ symmetric domain of type $e_3\rho_3+v_{13}\rho_1\rho_3+v_{23}\rho_2\rho_3+\rho_1\rho_2$ with $\bar{\mu}_1(\Omega)\geq 8\pi-e^{-c'\sqrt{k}}$, yielding a bound of the desired form when $m\leq Cn(e_3+v_{13}+v_{23})$. More generally, we can realize this domain as a limit of domains of type $f+\sum_ie_i\rho_i+\sum_{i<j}v_{ij}\rho_i\rho_j$ whenever $f+e_1+e_2+v_{12}>0$, giving the desired lower bound in this case provided $m\leq Cn(e_3+v_{13}+v_{23})$ holds.

\end{proof}

To treat $D_n$-symmetric domains of type $f+e_1\rho_1+e_2\rho_2+1\rho_1\rho_2$ with $\max\{f,e_1,e_2\}=1$, we next consider a similar construction with one hole at a single pole and the others arranged symmetrically on a great circle slightly perturbed from the equator.

\begin{lemma}
\label{Lpollat}
There exists a constant $c>0$ such that for every $k\geq 2$, one can find a domain $\Omega\subset \mathbb{S}^2$ whose complement is the disjoint union of one disk centered at the north pole $p_N$ and $k$ disks arranged symmetrically along a great circle centered at $p_N$, such that 
\begin{align*}
\lambda^D_1(\Omega) \geq 2
\quad
\text{and}
\quad 
|\Sph^2 \setminus \Omega| \leq e^{-c \sqrt{k}}. 
\end{align*}
\end{lemma}
\begin{proof}
Let $X = \{p_N\} \cup \{x_1,\ldots,x_k\}$, where $x_1,\ldots,x_k$ is a collection of points of distance $d(x_i,p_N)=\pi/2+ 1/\sqrt{k}$ from the north pole $p_N$ such that $d(p_{i-1},p_i)=d(p_i,p_{i+1})$ for every $1\leq i\leq k$ .  In what follows, we set $t: = 1/\sqrt{k}$. 

Now define $\Omega = \Sph^2 \setminus D_r(X)$ for
\begin{align}
\label{Erpollat}
r = e^{-c \sqrt{k}},
\end{align}
with $c>0$ a fixed small constant to be chosen later; note that this immediately implies an area bound of the desired form for $\Sph^2 \setminus \Omega$.

Next we consider the lower bound on $\lambda_1^D(\Omega)$. Fixing $u \in W^{1, 2}_0(\Omega)$ and applying Lemma \ref{Lmonobd} with $K$ taking the role of the closure of either of the domains bounded by the circle $S = \partial D_{\pi/2 +t}(p_N)$, and with $R = |S|/k$, we see using \eqref{Erpollat} that 
\begin{equation}
\label{Eprepmono}
\begin{aligned}
\| u\|^2_{L^2(F_R)} &\leq C c k^{-\frac{1}{2}} \| du\|^2_{L^2(E_R)}
\quad
\text{and}\\
\| u\|^2_{L^2(E_R)} &\leq C c k^{-\frac{3}{2}}\| du\|^2_{L^2(E_R)},
 \end{aligned}
 \end{equation}
  where $F_R : =\Omega \cap \partial D_20R(S)$ and $E_R : = \Omega \cap D_20R(S)$.

%Note that
%$\partial D_R(S) \approx \{x_3 = t \pm R\}$ and $D_R(S) \approx \{ t-R \leq x_3 \leq t+R\}$. 

Now let $V_1 = \Sph^2 \setminus D_{\pi/2+ t + 20R}(p_N) $. % \{ x_3 < t- R\}$.  
It is easy to see that
\begin{align*}
Q(V_1) : = \inf \left\{ \left. \frac{\int_{V_1} |d\phi|^2 - 2 \phi^2}{\| \phi \|^2_{L^2(\bd V_1)}} \right| \phi \in W^{1,2}(V_1)\right\}
\end{align*}
must be realized by a multiple of $u = \cos d(p_N, \cdot) $, so that
\begin{align*}
Q(V_1) = \cot (t+R)  \leq C \sqrt{k}
\end{align*}
and therefore that
\begin{equation}
\label{EV1}
\begin{aligned}
2 \int_{V_1} u^2 & \leq \int_{V_1} |du|^2 + Q({V_1}) \int_{F_R} u^2\\
&= \int_{V_1} |du|^2 + C c \int_{E_R} |du|^2.
\end{aligned}
\end{equation}

Now let $V_2 = D_{\pi/2+t+20R}(p_N) \setminus D_{\pi/2}(p_N)$.  %\{ t- R \leq x_3 \leq 0\}$. 
 On $V_2$, similarly to~\eqref{ineq:rect_der}, one can see that for $s\in[0,t+R]$
\begin{equation}
\begin{aligned}
\label{EW1p}
\int_{\partial D_{\pi/2+ s}(p_N) } u^2 &\leq C | t| \int_{V_2} |du|^2 + C\int_{F_R} u^2\\
&\leq \frac{C}{\sqrt{k}} \int_{V_2} | du|^2 + \frac{C c}{\sqrt{k} } \int_{E_R} |du|^2
\end{aligned}
\end{equation}
and integrating over $s$ gives
\begin{equation}
\begin{aligned}
\label{EV2}
\int_{V_2} u^2 &\leq C|t|^2 \int_{V_2} |du|^2 + C | t| \int_{F_R} u^2
\\
& \leq 
\frac{C}{k} \int_{V_2} | du|^2 + \frac{C c}{k} \int_{E_R} |du|^2,
\end{aligned}
\end{equation}
where the second line inequalities use \eqref{Eprepmono}.

Meanwhile, on $V_3 = D_{\pi/2}(p_N) \setminus D_{p_N}(r)$, arguing as in the proof of Lemma \ref{Leqpol} and combining with \eqref{Erpollat} and \eqref{EW1p}  shows that 
\begin{equation}
\label{EV3}
\begin{aligned}
2\int_{V_3} u^2 &\leq 
\int_{V_3} |du|^2 + C|\log r | \int_{ \partial D_{\pi/2}(p_N)} u^2
\\ 
&\leq \int_{V_3} |du|^2 + C c \int_{V_2} |du|^2 + C c^2 \int_{E_R} |du|^2. 
\end{aligned}
\end{equation}

It follows by combining \eqref{EV1}, \eqref{EV2}, and \eqref{EV3} and taking $c>0$ small enough that 
\begin{align*}
2 \int_{\Omega} u^2 &= 2\sum_{i=1}^3 \int_{V_i} u^2
\leq \int_{\Omega} |du|^2,
%
%&\leq \int_{V_1 \cup V_3} |du|^2 + \left( C c +\frac{C}{k}\right) \int_{V_2} |du|^2 + C c \int_{E_R} |du|^2.
\end{align*}
which implies the desired eigenvalue bound. 
\end{proof}

As an immediate corollary, we have the following.

\begin{proposition}\label{offeq1pole}
There is a constant $c>0$ such that if $\Omega$ is a $D_n$-symmetric domain of type $f+e_1\rho_1+e_2\rho_2+1\rho_1\rho_2$ with $\max\{f,e_1,e_2\}=1$, then $\mathcal{M}_1(\Omega,D_n)\geq 8\pi-e^{-c\sqrt{m}}$, where $m=n(2f+e_1+e_2)+1$ is the number of boundary components of $\Omega$.
\end{proposition}
\begin{proof}
To see this, simply observe that, taking $k=n(2f+e_1+e_2)$, the domain $\Omega$ of Lemma \ref{Lpollat} can be realized as a $D_n$-symmetric domain of type $f+e_1\rho_1+e_2\rho_2+1\rho_1\rho_2$; since Lemma \ref{Lpollat} together with Proposition \ref{Pneu} imply that this domain satisfies 
$$\mathcal{M}_1(\Omega,D_n)\geq \bar{\mu}_1(\Omega)\geq 8\pi-Ce^{-c\sqrt{k}},$$
we obtain a bound of the desired form.
\end{proof}

\subsection{Segmented domains}

To complete the proof of the lower bounds for $\mathcal{M}_1(\Omega,G)$ for all basic reflection surfaces, essentially two cases remain: $D_k$-symmetric domains of type $f+e_1\rho_1+e_2\rho_2+v_{12}\rho_1\rho_2$ with $2\leq \max\{f,e_1,e_2\}\leq Ck$, and $\mathbb{Z}_2\times D_k$-symmetric domains of type $f+\sum_i e_i\rho_i+\sum_{i<j}v_{ij}\rho_i\rho_j$ with total number of boundary components $m\leq Ck\max\{f,e_1,e_2\}\leq C'k^2$. As we will see, both cases can be handled with the following construction.

\begin{lemma}
\label{Lseg}
Let $k>2$ be given, and let $D_k \leq O(3)$ be equipped with its standard action on $\Sph^2$.  Given $c_1 >0$ and a natural number $n$ with $2 \leq n \leq c_1 k$, there exists a $D_k$-invariant domain $\Omega \subset \Sph^2$ of type $n \rho_1$ such that 
\begin{align*}
\lambda^D_1(\Omega) \geq 2
\quad
\text{and} 
\quad
|\Sph^2 \setminus \Omega| \leq  e^{-C(c_1) nk}.
\end{align*}
Moreover, if $n\in 2\mathbb{N}$, $\Omega$ possesses an additional reflection symmetry, making it a $\mathbb{Z}_2\times D_k$-symmetric domain of type $\frac{n}{2}e_1$.
\end{lemma}
\begin{proof}
For $i=0, \dots, n+1$ and $j = 1, \dots, n+1$, define
\begin{align*}
t_i = -1 + \frac{2i}{n+1}
\quad
\text{and}
\quad
A_i = \Sph^2 \cap \{x_3 \in [t_{i-1}, t_i]\}.
\end{align*}
As long as $n\geq 2$, is easy to see there is are constants $c, \delta_0> 0$ such that 
\begin{align}
\label{eAlower}
\lambda^D_1(A_i) \geq \max\{ 2+ \delta_0, c \cdot  n \}.
\end{align}
Let $X \subset \Sph^2$ be a $G$-invariant collection of $2kn$ points with the property that each $G$-orbit lies on one of the $\{x_3 = t_i\}$ and $X$ meets the boundary of a fundamental chamber $V$ for $G$ along a single side of $\partial V$. 

Now define $\Omega = \Sph^2 \setminus D_r(X)$, where $r$ is defined by 
\begin{align}
\label{Erseg}
r = \frac{1}{k} e^{-c nk},
\end{align}
with $c>0$ a small constant to be chosen at the end of the proof. This implies the area bound $|\Sph^2 \setminus \Omega| \leq Ce^{-2c nk}$ in the Lemma, so it only remains to prove the eigenvalue bound $\lambda^D_1(\Omega)\geq 2$.  To this end, we argue similarly to the proof of Proposition \ref{Pedges}, taking care to ensure that the estimates are uniform in $n$ and $k$.

For each $i=1, \dots, n+1$, let
\begin{align*}
U_i = A_i \cap \Omega,
\end{align*}
and take $R_i \sim |\partial A_i|/(10 k) $, so that the distance between nearest points of $X \cap A_i$ is proportional to $R_i$.  Note also that $\frac{c}{k\sqrt{n} } \leq R_i \leq \frac{c}{k}$ for all $i$.  Observe that the $R_i$-neighborhood of $\partial A_i$ can be partitioned into $2k$ rectangles $Q_1, \dots, Q_{2k}$ uniformly (with constants independent of $i,n,k$) bi-Lipschitz to $B_{R_i}( 0) \subset \R^2$, such that $Q_j \cap U_i$ is uniformly bi-Lipchitz to $B_{R_i}(0) \setminus B_r( 0)$.  

Now fix $u \in W^{1,2}_0(\Omega)$, extend $u$ to $A_i$ by requesting $u = 0$ on $A_i \setminus U_i$, and define the decomposition $u= u_0 + u_h$ on $A_i$, where $u_h$ is harmonic and $u_0 \in W^{1,2}_0(A_i )$.  Arguing just as in the proof of Lemma \ref{Lmonobd} shows that
\begin{align}
\label{Eaest}
\| u_h\|^2_{L^2(\partial A_i)} = \| u\|^2_{L^2(\partial A_i)} \leq C R_i \log (R_i / r)\| du \|^2_{L^2(A_i)}.
\end{align}
Also, a straightforward computation gives
\begin{align}
\label{Eaharm}
\int_{A_i} u^2_h \leq C \frac{t_i - t_{i-1}}{|\partial A_i|} \int_{\partial A_i} u^2_h \leq \frac{C}{nk R_i} \int_{\partial A_i} u^2_h.
\end{align}
Indeed, let $\phi$ be a harmonic function on the cylinder $A_T = \Sph^1\times [0,T]$ and let $f = \frac{1}{2}(Tt - t^2)\geq 0$. Then $f$ satisfies $\Delta f = 1$ and $f|_{\bd A_T} = 0$, so that
\begin{equation*}
\begin{split}
\int_{A_T}\phi^2\Delta f = -\int_{A_T}|d\phi|^2f - \int_{\bd A_T}u^2\bd_{\nu}f\leq \frac{T}{2}\int_{\bd A_T}\phi^2.
\end{split}
\end{equation*}
Inequality~\eqref{Eaharm} follows from this and the conformal identification of $A_i$ with cylinders.

Combining \eqref{Eaest} and \eqref{Eaharm} with Cauchy's inequality with some $\delta> 0$ gives  
\begin{align*}
 \int_{A_i} u^2 &=  \int_{A_i} u^2_0 + 2 \int_{A_i} u_0 u_h +  \int_{A_i} u^2_h\\
&\leq (1+\delta)\int_{A_i} u^2_0 + \frac{C}{\delta}\int_{A_i} u^2_h\\
&\leq \frac{1+\delta}{\lambda_1(A_i)} \int_{A_i} |du_0|^2 + \frac{C}{\delta n k } \log (R_i / r) \int_{A_i} |du|^2\\
&\leq  \left( \frac{1+\delta}{\lambda_1(A_i)} + \frac{C}{\delta n k } \log (R_i / r)\right) \int_{A_i} |du|^2,
\end{align*}
Finally, note by \eqref{eAlower} and \eqref{Erseg} that
$$\frac{1+\delta}{\lambda_1(A_i)}+\frac{C}{\delta n k}\log(R_i/r)\leq \frac{1+\delta}{2+\delta_0}+\frac{C'}{\delta}c,$$
which can be made $\leq \frac{1}{2}$ by taking $\delta=\frac{\delta_0}{10}$ and setting $c=\frac{\delta_0^2}{100C'}$. With these choices, the preceding inequality then gives
$\| u\|^2_{L^2(A_i)} \leq \frac{1}{2} \| du \|^2_{L^2(A_i )}$.  Because $u$ vanishes on $A_i \setminus \Omega$, it follows that $\|u\|^2_{L^2(A_i \cap \Omega)} \leq \frac{1}{2} \| du \|^2_{L^2(A_i \cap \Omega)}$; by summing over $i$, we conclude finally that $\| u \|^2_{L^2(\Omega)} \leq \frac{1}{2} \| du \|^2_{L^2(\Omega)}$, and this completes the proof. 
\end{proof}

With the preceding lemma in hand, we can now provide the desired lower bounds in the following cases.

\begin{prop}
\label{Pseg}
For any given $c_1>0$, there exists a constant $C(c_1)>0$ such that if either
\begin{itemize}
\item $G = \Z_2 \times D_k$ and $m \leq c_1k ( f+ e_1 + e_2)\leq Cc_1k^2$, or
\item $G = D_k$ and $ 2 \leq (f+e_1+e_2) < c_1k$ with $m\leq Ck\max\{f,e_1,e_2\}$ , 
\end{itemize}
there exists a $G$-invariant domain $\Omega \subset \Sph^2$ of type $f + \sum_i e_i \rho_i + \sum_{i< j} v_{ij} \rho_i \rho_j$ with $m$ total boundary components such that
\begin{align*}
\lambda^D_1(\Omega) \geq 2
\quad
\text{and} 
\quad
|\Sph^2 \setminus \Omega | \leq e^{-C m}.
\end{align*}
In particular, there is a constant $C>0$ such that $\mathcal{M}_1(\Omega,G)\geq 8\pi-e^{-Cm}$ in each of these cases.
\end{prop}
\begin{proof}
For $G=\mathbb{Z}_2\times D_k$, taking $n=2\max\{f,e_1,e_2\}$, we see that the domain $\Omega$ constructed in Lemma \ref{Lseg} can be viewed as a $G$-symmetric domain of type $\frac{n}{2}f$ or $\frac{n}{2}\rho_1$ or $\frac{n}{2}\rho_2$, and by removing an additional collection of arbitrarily small disks, we obtain a $G$-symmetric domain $\Omega'$ of the desired type $f+\sum_ie_i\rho_i+\sum_{i<j}v_{ij}\rho_i\rho_j$ with $\lambda_1^D(\Omega')\geq 2$ and $|\Sph^2\setminus\Omega'|\leq e^{-C_1nk}\leq e^{-C_2m}$.

Similarly, taking $n=f+e_1+e_2$, by removing a suitable collection of arbitrarily small disks from the domain $\Omega$ of Lemma \ref{Lseg}, we obtain a $D_k$-invariant domain of type $f+e_1\rho_1+e_2\rho_2+v_{12}\rho_1\rho_2$ with $\lambda_1^D(\Omega)\geq 2$ and $|\mathbb{S}^2\setminus\Omega|\leq Ce^{-Ckn}\leq e^{-C_2m}$, provided $Ck\max\{f,e_1,e_2\}\geq m$. 

In both cases, the desired bound on $\mathcal{M}_1(\Omega,G)\geq \bar{\mu}_1(\Omega)$ then follow from an application of Proposition \ref{Pneu}.
\end{proof}

By collecting all of the estimates above, we can finally prove Theorem \ref{s3.lbd}.

\begin{proof}[Proof of Theorem \ref{s3.lbd}]

By Lemma \ref{double.lem}, the desired conclusion is equivalent to the statement that for all finite reflection groups $G\leq O(3)$ and all $G$-invariant domains with $m$ bounday components, we have
\begin{equation}\label{scherk.laws}
\mathcal{M}_1(\Omega,G)\geq \Area(\xi_{m-1,1})
\end{equation}
if $(\Omega,G)$ corresponds to a Scherk-type basic reflection surface, and 
\begin{equation}\label{gen.bd}
\mathcal{M}_1(\Omega,G)\geq 8\pi-e^{-C\sqrt{m}}
\end{equation}
for some universal constant $C>0$ if $(\Omega,G)$ is a basic reflection surface of generic type. That \eqref{scherk.laws} holds for surfaces of Scherk type is the content of Lemma \ref{scherk.lem}, so it remains to discuss the generic case.

For $G=1$ or $G$ one of the platonic groups, Corollary \ref{plat.cor} gives a lower bound of the form $\mathcal{M}_1(\Omega,G)\geq 8\pi-e^{-c'm}$, which clearly implies the weaker bound \eqref{gen.bd}. For $G=\mathbb{Z}_2$, the $G$-invariant domain $\Omega$ of type $f+e_1\rho_1$ is generic--for $m=2f+e_1\geq 4$--precisely when $f>0$, in which case the bound \eqref{gen.bd} is given in Proposition \ref{eq.poles}.

For $G=D_k$, a $D_k$-invariant domain $\Omega$ of type $f+e_1\rho_1+e_2\rho_2+v_{12}\rho_1\rho_2$ is generic whenever $f+e_1+e_2+v_{12}\geq 2$ and $m=k(f+e_1+e_2)+v_{12}$ is sufficiently large. Combining Proposition \ref{Pfmany} and Proposition \ref{dk.edges}, we obtain a lower bound of the form \eqref{gen.bd} whenever $f+e_1+e_2\geq C_0k$ for some constant $C_0$. When $2\leq f+e_1+e_2\leq C_0k$, \eqref{gen.bd} follows from Proposition \ref{Pseg} above. Finally, if $f+e_1+e_2=1$ and $v_{12}=1$, the desired bound follows from Proposition \ref{offeq1pole}, while if $f+e_1+e_2=1$ and $v_{12}=2$, it follows from Proposition \ref{eq.poles}.

It remains to confirm \eqref{gen.bd} for $\mathbb{Z}_2\times D_k$-invariant domains $\Omega$ of type $f+\sum_{i=1}^3e_i\rho_i+\sum_{i<j}v_{ij}\rho_i\rho_j$ with $f+e_1+e_2+v_{12}>0$. In the case $m\leq Cn(e_3+v_{13}+v_{23})$, this follows from Proposition \ref{eq.poles}; otherwise, we have $m\leq C(f+e_1+e_2+v_{12})$. In the latter case, if $\max\{f,e_1,e_2\}\geq Ck$, the desired bound follows from a combination of Propositions \ref{Pfmany} and \ref{dk.edges}, while if $\max\{f,e_1,e_2\}\leq Ck$, it follows from Proposition \ref{Pseg}.  
\end{proof}

\section{Lower Bounds in the Steklov setting}
\label{sec:lowbounds_St}

We turn next to the lower bounds in Theorem \ref{main.stek.bds}, showing that every basic reflection surface with boundary $(N,\Gamma)$ has a lower bound of the form $\Sigma_1(N,\Gamma)\geq \max\{4\pi-e^{-Ck}, 4\pi-C/\gamma\}$, where $\gamma$ is the genus of $N$ and $k$ the number of boundary components, and also proving the lower bound in Theorem \ref{max.stek.bds}, showing that $\Sigma_1(N,\mathbb{Z}_2)\geq 4\pi-e^{-C(\gamma+k)}$ for the largest basic reflection surfaces of prescribed topology. 

Similar to the case of Laplace eigenvalues on closed surfaces, the results of Section 2 allow us to establish lower bounds simply by constructing suitable domains in the disk with prescribed topology and symmetries, corresponding to fundamental domains for the defining reflection of the basic reflection surface. As a first step, we prove the following analog of Proposition \ref{Pneu} for domains of this type.

\begin{prop}
\label{Pstekn} 
There is a constant $C>0$ such that if $X \subset \DD \cup \partial \DD$ is a finite set and $\{D_{2r_x}(x)\}_{x \in X}$ is a pairwise disjoint collection of disks and half-disks, then
\begin{align*}
\sigma_1^N(\Omega) \geq 1 - C ( | \Dcal | + | \partial \DD \cap \Dcal |),
\end{align*}
where $\Dcal = \cup_{x \in X} D_{r_x} (x)$, $\Omega = \DD \setminus \Dcal$, and $\sigma_1^N(\Omega)$ is the mixed Steklov-Neumann eigenvalue
\begin{align*}
\sigma^N_1(\Omega) : = \inf \left\{ \left.\frac{ \int_{\Omega} | d\phi |^2}{\int_{\partial \DD \setminus \Dcal} \phi^2} \, \right| \, \int_{\DD \cap  \partial \Dcal} \phi = 0, \phi \neq 0\right\}.
\end{align*}
\end{prop}

\begin{lemma}
\label{Lstekn}
There exists a universal constant $\epsilon_0> 0$ such that if $\Omega \subset \DD$ is as in Proposition \ref{Pstekn} and $|\Dcal| < \epsilon_0$, then $\sigma^N_1(\Omega) \geq 3/2$.
\end{lemma}
\begin{proof}
As in the closed case (recall Lemma \ref{Lneu}), the pairwise disjointness of $\{D_{2r_x}(x)\}_{x \in X}$ implies uniform bounds
\[ 
\| d(H \phi) \|_{L^2(\DD)} \leq C \| d \phi \|_{L^2( \Omega)}
\] 
for the harmonic extension operator $H : W^{1,2}(\Omega) \rightarrow W^{1,2}(\DD)$.

Now, let $V \subset C^\infty( \Omega)$ be any $4$-dimensional subspace.  Because the coordinate and constant functions span a $3$-dimensional subspace of $L^2(\partial \DD)$, there exists a nonzero $\phi \in V$ such that $\hat{\phi} = H \phi$ is $L^2(\partial \DD)$-orthogonal to the constants and coordinate functions.  In particular, we can decompose 
\begin{align*}
\hat{\phi} = \psi + R = \sum_{j=3}^\infty a_j u_j + R,
\end{align*}
where $\{u_j\}$ is an $L^2(\partial \DD)$-orthonormal basis of Steklov eigenfunctions with eigenvalue $\sigma_j(\DD ) \geq \sigma_3(\DD) = 2$ and $R \in W^{1,2}_0(D)$.  Fixing $\Lambda> 2$, denote by $\psi_\Lambda$ the projection of $\psi$ onto the Steklov eigenspaces with eigenvalue $\sigma_j < \Lambda$, so that
\begin{align*}
\| \hat{\phi} - \psi_\Lambda\|^2_{L^2(\partial \DD)} \leq \frac{1}{\Lambda} \| d \hat{\phi}\|^2_{L^2(\DD)} \leq \frac{C}{\Lambda} \| d \phi \|^2_{L^2(\Omega)},
\end{align*}
while the low-frequency projection $\psi_\Lambda$ satisfies a $C^1$ estimate
\begin{align*}
\| \psi_\Lambda \|_{C^1(\DD)} \leq C(\Lambda) \| d \psi_{\Lambda}\|_{L^2(\DD)},
\end{align*}
so that
\begin{align*}
\int_{\DD} | d\psi_\Lambda|^2 &= \int_{\DD} \langle d \hat{\phi} , d \psi_{\Lambda}\rangle \\
&\leq \int_{\Omega} \langle d \phi, d \psi_{\Lambda}\rangle + C(\Lambda) |\Dcal|^{1/2} \| d \psi_{\Lambda}\|_{L^2(\DD)} \| d \phi \|_{L^2( \Omega)}\\
&\leq (1+ C(\Lambda) \sqrt{\epsilon_0}) \| d \phi \|^2_{L^2(\Omega)}.
\end{align*}
Putting this together, we see that
\begin{align*}
\| \phi\|^2_{L^2(\partial \DD \cap \Dcal)} 
&\leq \| \hat{\phi} \|^2_{L^2(\partial \DD)}\\
&= \| \psi_\Lambda\|^2_{L^2(\partial \DD)} + \| \hat{\phi} - \psi_{\Lambda}\|^2_{L^2(\partial \DD)}\\
&\leq \frac{1}{2} \| d \psi_\Lambda\|^2_{L^2(\DD)} + \frac{C}{\Lambda} \| d \phi \|^2_{L^2(\Omega)}\\
&\leq \left( 1+ C(\Lambda) \sqrt{\epsilon_0}  + \frac{C}{\Lambda} \right) \| d \phi \|^2_{L^2(\Omega)}.
\end{align*}
Fixing $\Lambda$ sufficiently large that $C/\Lambda < 1/9$, then choosing $\epsilon_0$ such that $(1+C(\Lambda) \sqrt{\epsilon_0})< 10/9$, we find that $V$ contains an element with
\begin{align*}
\frac{ \int_\Omega |d\phi|^2 }{\int_{\partial \DD \cap \Dcal} \phi^2 } \geq \frac{3}{2},
\end{align*}
as claimed.
\end{proof}

We are now ready to prove Proposition \ref{Pstekn}. 

\begin{proof}[Proof of Proposition \ref{Pstekn}]
Let $u$ be a coordinate function on $\DD$.  Because $u$ is a Steklov eigenfunction with eigenvalue $1$, we have $\int_{\DD} \langle du , dw\rangle - \int_{\partial \DD} u w =0$ for any test function $w \in C^\infty (\DD )$.  Fixing $\phi \in W^{1,2}(\Omega)$ and using the harmonic extension $\hat{\phi} = H \phi$ as a test function shows 
\begin{equation}
\label{Eteststek}
\begin{aligned}
\left| \int_\Omega \langle du , d\phi \rangle - \int_{\partial \DD \setminus \Dcal} u \phi \right| &= \left| \int_{\Dcal} \langle d u, d \hat{\phi} \rangle - \int_{\partial \DD \cap \Dcal} u \hat{\phi}\right|\\
&\leq C |\Dcal|^{1/2} \| \phi \|_{W^{1,2}(\Omega)} + C  | \partial \DD \cap \Dcal|^{1/2} \| \phi \|_{W^{1,2}(\Omega)}.
\end{aligned}
\end{equation}

Now let $\{v_i\}$ be an $L^2( \partial \DD \setminus \Dcal)$-orthonormal sequence of $\sigma^N_j$-eigenfunctions on $\Omega$, and suppose without loss of generality that $\int_{\partial \DD \setminus \Dcal} u v_2 = 0$.  On $\Omega$ we can then write
\begin{align*}
u = \ulow + \uhigh + R% = a_1 v_1 + u_L + R,
\end{align*}
where $\uhigh$ is a sum of $\sigma^N_j(\Omega)$-eigenfunctions with $j \geq 3$ and $R$ vanishes on $\partial \DD \setminus \Dcal$.  We now test \eqref{Eteststek} with $\phi = \uhigh$; first note that
\begin{align*}
\int_{\Omega} \langle du, d\uhigh \rangle - \int_{\partial \DD \setminus \Dcal} u \uhigh 
&= \int_{\Omega} | d\uhigh |^2 - \int_{\partial \DD \setminus \Dcal} \uhigh^2
\\
& \geq (1 -1/ \sigma^N_3(\Omega) ) \| d\uhigh \|^2_{L^2(\Omega)} \geq c \| \uhigh\|^2_{W^{1,2}},
\end{align*}
so the inequality \eqref{Eteststek} yields
\begin{align*}
\| \uhigh\|^2_{W^{1,2}(\Omega)} \leq C ( |\Dcal|^{1/2} + |\partial \DD \cap \Dcal|^{1/2} ) \| \uhigh \|_{W^{1,2}(\Omega)},
\end{align*}
hence 
\begin{align}
\label{Euhbd}
\| \uhigh \|^2_{W^{1,2}(\Omega)} \leq C ( | \Dcal | + |\partial \DD \cap \Dcal |). 
\end{align}
Since the restrictions of the coordinate functions to $\Omega$ are linearly independent, we may therefore choose $u$ so that $\ulow=v_1$ is a $\sigma^N_1(\Omega)$-eigenfunction.  Noting that $\| \ulow\|^2_{\partial \DD \cap \Dcal} = \int_{\partial \DD \cap \Dcal} u v_1= \int_{\partial \DD \cap \Dcal} (u^2 - u\uhigh)$, it follows from \eqref{Euhbd} that
\begin{align*}
\|\ulow\|^2_{L^2(\partial \DD \cap \Dcal)} \geq \pi - C\epsilon_0 > 3
\end{align*}
once $|\Dcal |^{1/2} + | \partial \D \cap \Dcal |^{1/2} < \epsilon_0$.  In particular, it follows that
\begin{equation}
\begin{split}
3 &|\sigma^N_1(\Omega) - 1 | = \left| \int_\Omega \langle du, dv_1\rangle - \int_{\partial \DD \cap \Dcal} u v_1\right|\\
\leq &\left| \int_{\Omega} | du|^2 - \int_{\partial \DD \cap \Dcal} u^2| + | \int_\Omega \langle du, d\uhigh \rangle - \int_{\partial \DD \cap \Dcal } u \uhigh \right|\\
\leq &| (\pi - |\Dcal|) - (\pi - c |\partial \DD \cap \Dcal|)| + C ( | \Dcal|^{1/2} + | \partial \DD \cap \Dcal|^{1/2}) \| u_L\|_{W^{1,2}(\Omega)}\\
\leq &C ( |\Dcal| + |\partial \DD \cap \Dcal|).
\end{split}
\end{equation}
Combined with the preceding, this gives the desired estimate. 
\end{proof}

\subsection{Lower bounds for $\sigma^D_1(\Omega)$}

With Proposition \ref{Pstekn} in place, the desired lower bounds for $\Sigma_1(N,\Gamma)$ will follow if we can exhibit domains $\Omega$ of prescribed topology and symmetry type satisfying the hypotheses of Proposition \ref{Pstekn} with $|\mathcal{D}|+|\partial \mathbb{D}\cap \mathcal{D}|$ sufficiently small and $\sigma_1^D(\Omega)\geq 1$, where, recall 
\begin{align}
\label{Estekdir}
\sigma^D_1(\Omega) : = \inf \left\{ \left.\frac{ \int_\Omega | d \phi |^2}{\int_{\partial \DD \setminus \DD} \phi^2 } \, \right| \, \phi|_{\DD \cap \partial\mathcal{D}} =0, \phi \neq 0\right\}.
\end{align}
In this section, we carry out these constructions, first obtaining estimates of the desired form in terms of the number $k$ of boundary components.  In what follows, for domains as in Proposition \ref{Pstekn}, we write
$$\bar{\sigma}(\Omega):=L(\partial\Omega\setminus \partial \mathcal{D})\cdot \min\{\sigma_1^D(\Omega),\text{ }\sigma_1^N(\Omega)\}.$$

\begin{prop}
\label{Pstekk}
There are constants $C, c>0$ such that, given $k\in \mathbb{N}$, taking $X=\{e^{2\pi i \ell/k}\mid 0\leq \ell \leq k-1\}$ and $r_x=r_k>0$ in Proposition \ref{Pstekn} for a suitable small radius $r_k>0$, the domain $\Omega$ satisfies
\begin{align*}
\sigmabar(\Omega) \geq 2\pi - e^{-c k},
\end{align*}
where $\Dcal : = \DD \setminus \Omega$. 
\end{prop}
\begin{proof}
Let $\Omega = \DD \setminus D_r (X)$, where
\begin{align}
\label{Erstekk}
X = \{ (\cos \textstyle{\frac{2\pi j}{k}}, \sin \frac{2\pi j}{k}) : j=1, \dots, k\} \subset \partial \DD, 
\quad
r = \frac{1}{k}e^{-ck},
\end{align}
and $c> 0$ will be determined later.

Choose $R = C/k$ such that the disks $\{D_R(x)\}_{x \in X}$ are disjoint and $\{D_{4R}(x)\}_{x \in X}$ cover $\partial \DD$.  By Lemma \ref{Lmono} one has
\begin{align*}
\| u \|^2_{L^2(\partial D_R(x) \cap \DD)} &\leq C R\log (R/r) \| du \|^2_{L^2(D_R(x) \cap \DD)}\\
&\leq C c \| du \|^2_{L^2(D_R(x) \cap \DD )}
\end{align*}
when $u$ is a $\sigma^D_1(\Omega)$-eigenfunction, extended trivially into $\Dcal$.  A scale invariant trace inequality then gives
\begin{align*}
\| u\|^2_{L^2(D_R(x) \cap \partial \DD)} &\leq C \| u\|^2_{L^2(\partial D_R(x) \cap \DD)} + C R \| du \|^2_{L^2(D_R(x))}\\
&\leq C c  \| du \|^2_{L^2(D_R(x) \cap \DD )}.
\end{align*}
Summing over $x \in X$ gives
\begin{align*}
\| u \|^2_{L^2(\partial \DD \setminus \Dcal)} \leq C c \| du \|^2_{L^2(\Omega)} = C c \sigma^D_1(\Omega) \| u \|^2_{L^2(\partial \DD \setminus \Dcal)} ,
\end{align*}
where the last step follows from the variational characterization \eqref{Estekdir} for $\sigma^D_1(\Omega)$.  By taking $c> 0$ small enough, this shows $\sigma_1^D(\Omega) \geq 1$, and the conclusion follows by combining with Proposition \ref{Pstekn}.
\end{proof}

In the next proposition, we obtain a strong lower bound for $\bar{\sigma}(\Omega)$ in terms of the genus, when the genus is arranged symmetrically along a circle close to $\partial\mathbb{D}$.

\begin{prop}\label{pstekmbdry}
There exist constants $c,c'> 0$ such that if $\Omega \subset \DD$ is a domain as in Proposition \ref{Pstekn} with $X\cap\bd\DD$ arbitrary and 
\begin{align*}
X\cap \DD = \{ \textstyle{\frac{m-1}{m}}(\cos \textstyle{\frac{2\pi j}{m}}, \sin \frac{2\pi j}{m}) : j=1, \dots, m\}
\quad
r = \frac{1}{m} e^{-c m },
\end{align*}
then
\begin{align*}
\sigmabar(\Omega) \geq 2\pi - e^{-c'm}. 
\end{align*}
\end{prop}
\begin{proof}
We can divide the annulus $D_1(0) \setminus D_{1- 2/m} (0)$ into a collection $\{Q_x\}_{x \in X}$ of polar ``rectangles" such that for each $x \in X$, there is a bi-Lipschitz correspondence 
\begin{align*}
Q_x \setminus D_r(x) \rightarrow D_{1/m} (0) \setminus D_r(0)
\end{align*}
with Lipschitz constants independent of $m$ and $r$.  If $u$ is a $\sigma^D_1$-eigenfunction on $\Omega$, then Lemma \ref{Lmono} gives an estimate of the form 
\begin{align*}
m \| u \|_{L^2(\partial Q_x)}^2 &\leq C \log (1/m r) \| du \|^2_{L^2(\Omega \cap Q_x)}\\
&\leq C c m  \| du \|^2_{L^2(\Omega \cap Q_x)}.
\end{align*}
and summing over all $x \in X$ gives 
\begin{align*}
\| u\|^2_{L^2(\partial \DD) } \leq C c \| du \|^2_{L^2(\Omega)} = C c \sigma^D_1(\Omega) \|u\|^2_{L^2(\DD)},
\end{align*}
and hence $\sigma^D_1(\Omega)>1$ by taking $c$ small enough. 
\end{proof}

The next proposition provides a similar lower bound in the case where genus is arranged in a $D_n$-symmetric way along the fixed point sets of conjugacy classes in $D_n$, with $n$ sufficiently large and genus large relative to $n$.

\begin{prop}\label{pstekwedges}
There exist constants $c,c'>0$ and $n_0\in \mathbb{N}$ such that if $\Omega\subset \DD$ is a domain as in Proposition \ref{Pstekn} with $X\cap \bd \DD$ arbitrary and 
$$
X\cap \DD=\left\{\left.\frac{\ell}{a+1}\left(\cos \frac{2\pi j}{n},\sin \frac{2\pi j}{n}\right)\right| \ell=1,\ldots,a;\text{ }j=1,\ldots,n\right\}
$$
for some $a\geq n\geq n_0$, and $r=\frac{1}{a}e^{-ca}$, then
$$\bar{\sigma}(\Omega)\geq 2\pi-e^{-c'a}.$$
\end{prop}
\begin{proof}
Using the assumption that $a\geq n$, we observe that the polar rectangles
$$
Q_{\ell j}:=\left\{\left.\frac{t}{a+1}(\cos (s),\sin(s))\right| \ell-\frac{1}{2}\leq t \leq \ell+\frac{1}{2};\text{ }\frac{2\pi j}{n}-\frac{1}{a}\leq \theta\leq \frac{2\pi j}{n}+\frac{1}{a}\right\}
$$
with obvious minor modifications for $\ell=1$ and $\ell=a$, cover the union of rays
$$S=\mathbb{D}\cap\bigcup_{j=1}^n\mathrm{Span}\{e^{2\pi j/n}\}.$$
Moreover, for each $\ell,j$, we see that there are bi-Lipschitz maps
$$F_{\ell,j}: Q_{\ell j}\cap\Omega\to D_{1/a}(0)\setminus D_r(0)$$
satisfying $\mathrm{Lip}(F_{\ell,j})+\mathrm{Lip}(F_{\ell,j})^{-1}\leq C$ for $C<\infty$ independent of $\ell,j,n,$ and $a$.

Arguing as in the proof of Proposition \ref{pstekmbdry}, we let $u\in W^{1,2}(\Omega)$ be an eigenfunction corresponding to the eigenvalue $\sigma_1^D(\Omega)$, so that $u\equiv 0$ on $\mathcal{D}$, and apply Lemma \ref{Lmono} to the domains $Q_{\ell j}\cap \Omega$ to deduce that
$$\|u\|_{L^2(\partial Q_{\ell j})}^2\leq \frac{C}{a}\log(1/ra)\|du\|_{L^2(Q_{\ell j})}^2.$$
Next, by considering standard estimates for the embedding $W^{1,2}(D_s(0))\to L^2([-s,s]\times\{0\})$, we observe that we have an estimate of the form
$$\int_{Q_{\ell j}\cap S}u^2\leq C\|u\|_{L^2(\partial Q_{\ell j})}^2+\frac{C}{a}\|du\|_{L^2(Q_{\ell j})}^2\leq \frac{C'}{a}\log(1/ra)\|du\|_{L^2(Q_{\ell j})}^2.$$
In particular, summing over $\ell=1,\ldots,a$, and denoting by $W_j$, $W'_j$ the wedges
$$W_j:=\left\{re^{i\theta}\mid 0\leq r\leq 1,\text{ }\frac{2\pi j}{n}\leq \theta \leq \frac{2\pi(j+1)}{n}\right\},$$
$$W'_j:=\left\{re^{i\theta}\mid 0\leq r\leq 1,\text{ }\frac{2\pi (j-1)}{n}\leq \theta \leq \frac{2\pi(j+2)}{n}\right\},$$
we arrive at an estimate of the form
$$\int_{\partial W_j\setminus \bd\mathbb{D}}u^2\leq \frac{C}{a}\log(1/ra)\|du\|_{L^2(W'_j)}^2.$$

On the polar rectangle $R_j:=W_j\cap (\mathbb{D}\setminus D_{1-1/n}(0))$, it's easy to see that we have an estimate of the form
$$\int_{\partial R_j}v^2\leq C\int_{\partial W_j\setminus \bd\mathbb{D}}v^2+\frac{C}{n}\int_{W_j}|dv|^2$$
for any $v\in W^{1,2}(W_j)$, and applying this to the eigenfunction $u$, we deduce that
$$\int_{\partial R_j}u^2\leq \left(\frac{C'}{a}\log(1/ra)+\frac{C}{n}\right)\|du\|_{L^2(W'_j)}^2.$$
Summing over $j$, we deduce in particular that
$$\int_{\partial \mathbb{D}}u^2\leq \left(\frac{C'}{a}\log(1/ra)+\frac{C}{n}\right)\|du\|_{L^2(\mathbb{D})}^2=\left(\frac{C'}{a}\log(1/ra)+\frac{C}{n}\right)\sigma_1^D(\Omega)\int_{\partial \mathbb{D}}u^2.$$
For $n>\frac{2}{C}$ and $r=\frac{1}{a}e^{-a/2C'}$, we deduce that $\sigma_1^D(\Omega)>1$, and the desired estimates follow.
\end{proof}

Finally, we consider the problem of obtaining lower bounds in terms of the genus in the case where genus distributed along an axis. Basic reflection surfaces of this type are the analog of the Steklov setting of the ``Scherk-type" surfaces described in Definition \ref{scherk.def}, and the following lower bounds are qualitatively sharp.

\begin{prop}\label{pstekmaxis}
For any $m \in \N$, there exists a domain $\Omega \subset \DD$ as in \ref{Pstekn} with $X\cap\bd\DD$ arbitrary and $X\cap\DD$ arranged along a diameter of $\DD$, such that
\begin{align*}
\sigmabar(\Omega) \geq 2\pi - \frac{C}{m} - C | \partial \DD \cap \Dcal|. 
\end{align*}
\end{prop}
\begin{proof}
Fix a diameter $\ell=\DD\cap (\mathbb{R}\times \{0\})$ of $\DD$, let $X$ consist of $m$ points evenly spaced on the diameter $\ell$, so that $\dist(x,y)\geq \frac{1}{m}$ for distinct $x,y\in X$, and consider the domains $\Omega_m : = \DD \setminus D_{r_m}(X)$, where $r_m=\frac{1}{2m}$, so that the hypotheses of Proposition \ref{Pstekn} hold.  Moreover, we claim that $\sigma_1^D(\Omega_m)$ satisfies a lower bound of the form 
\begin{equation}\label{sig.d.bd}
\sigma_1^D(\Omega_m)\geq 1-\frac{C}{m}
\end{equation}
with $C$ independent of $m$; together with Proposition \ref{Pstekn}, it the follows that a domain of the desired type can be obtained by removing an additional collection of arbitrarily small disks with centers in $\bd\DD$.

To prove \eqref{sig.d.bd}, first observe that for any $\phi \in W^{1,2}(\Omega)$ with $\phi=0$ on $\partial D_{r_m}(X)$, on the union of line segments
$$S=S_+\cup S_-,\qquad S_\pm:=\partial D_{1/m}(\ell)\cap\{\pm y>0\}=\left\{(x,y)\in \DD\mid y=\pm\frac{1}{m}\right\},$$
we have an estimate of the form
\begin{equation}\label{tr.est}
\int_S\phi^2\leq \frac{C}{m}\int_{\Omega}|d\phi|^2.
\end{equation}
To see this, apply Lemma \ref{Lmono} to $\phi$ on the annuli $D_{3/m}(x)\setminus D_{r_m}(x)$, and noting that $(3/m)/r_m=6$, we deduce that 
$$\int_{D_{3/m}(x)}\phi^2\leq \frac{C}{m^2}\int_{D_{3/m}(x)}|d\phi|^2.$$
By rescaling standard estimates for the trace embedding $W^{1,2}(D_3(0))\to L^2(D_3(0)\cap\{y=\pm 1\})$, we then see that
$$\int_{D_{3/m}(x)\cap S}\phi^2\leq \frac{C'}{m} \int_{D_{3/m}(x)}|d\phi|^2.$$
In particular, since the disks $D_{3/m}(x)$ cover $S$, and any point is contained in at most two of the disks $D_{3/m}(x)$, summing this estimate over $x\in X$ yields a bound of the form \eqref{tr.est}.

Next, let $f_1,f_2\in W_0^{1,2}(\Omega)$ be the Lipschitz functions given by $f_1(x,y)=(y-1/m)^+$, $f_2(x,y)=(-y-1/m)^+$; for any $\phi\in W_0^{1,2}(\Omega)$, we then compute
\begin{eqnarray*}
\int_{\Omega}\langle df_1,d\phi\rangle&=&\int_{\{y>1/m\}}\frac{\partial \phi}{\partial y}\\
&=&\int_{\partial \mathbb{D}\cap \{y>1/m\}}y\phi-\int_S\phi\\
&=&\int_{\partial \mathbb{D}}f_1\phi +\frac{1}{m}\int_{\partial \mathbb{D}\cap \{y>1/m\}}\phi-\int_{S_+}\phi,
\end{eqnarray*}
and similarly
$$\int_{\Omega}\langle df_2,d\phi\rangle=\int_{\partial \DD}f_2\phi+\frac{1}{m}\int_{\partial \mathbb{D}\cap \{-y>1/m\}}\phi -\int_{S_-}\phi.$$
In particular, for any $f=af_1+bf_2\in \mathrm{Span}\{f_1,f_2\}$, we obtain
\begin{equation}
\left|\int_{\Omega}\langle df,d\phi\rangle-\int_{\partial \DD}\phi f_2\right|\leq \|f\|_{L^{\infty}}\left(\frac{1}{m}\int_{\partial \mathbb{D}}|\phi|+\int_S|\phi|\right),
\end{equation}
and, using \eqref{tr.est} and the usual trace embedding $W^{1,2}(\DD)\to L^2(\partial \DD)$, we deduce that
\begin{equation}\label{f.quasi}
\left|\int_{\Omega}\langle df,d\phi\rangle-\int_{\partial\Omega}f\phi\right|\leq \frac{C\|f\|_{L^{\infty}}}{\sqrt{m}}\|\phi\|_{W^{1,2}}
\end{equation}
for every $\phi\in W^{1,2}(\Omega)$ with $\phi=0$ on $\partial D_{r_m}(X)$. 

Next, given $f\in \mathrm{Span}\{f_1,f_2\}$, we can write $f$ as a sum $f=\phi_0+\sum_{k=1}^{\infty}\phi_k$, where $\phi_k$ is an eigenfunction corresponding to the $k$th mixed Steklov-Dirichlet eigenvalue $\sigma_k^D(\Omega)$, and $\phi_0=0$ on $\partial\Omega$; note that this decomposition is orthogonal for the inner products $\langle d\phi, d\psi\rangle_{L^2(\Omega)}$ and $\langle \phi,\psi\rangle_{L^2(\partial\Omega)}$. Moreover, since $\mathrm{Span}\{f_1,f_2\}$ is $2$-dimensional, we can find $f\in \mathrm{Span}\{f_1,f_2\}$ such that $\phi_2=0$ in the decomposition above; after normalizing, we then have $f=af_1+bf_2$ of the form
$$f=\phi_0+\phi_1+\sum_{k=3}^{\infty}\phi_k\text{ with }\|f\|_{L^{\infty}}=\max\{|a|,|b|\}-\frac{1}{m}=1.$$
Applying \eqref{f.quasi} with $\phi=\phi_0$, we see that
$$\|d\phi_0\|_{L^2(\Omega)}^2=\int_{\Omega}\langle df,d\phi_0\rangle \leq \frac{C}{\sqrt{m}}\|\phi_0\|_{W^{1,2}},$$
from which it follows that
$$\|\phi_0\|_{W^{1,2}}^2\leq \frac{C}{m}.$$

Next, suppose that for $m$ sufficiently large,
\begin{equation}\label{s3.gap}
\sigma_3^D(\Omega_m)\geq 1+\delta_0
\end{equation}
for a fixed $\delta_0>0$. Writing $\psi=\sum_{k=3}^{\infty}\phi_k$, it follows that
$$\int_{\Omega}\langle df,d\psi\rangle-\int_{\partial\Omega}f\psi\geq C\delta_0\|\psi\|_{W^{1,2}}^2,$$
which together with \eqref{f.quasi} implies that
$$\|\psi\|_{W^{1,2}}^2 \leq \frac{C}{\delta_0^2m},$$
and therefore
\begin{equation}\label{f.phi1.close}
\|f-\phi_1\|_{W^{1,2}}^2\leq 2\|f_0\|_{W^{1,2}}^2+2\|\psi\|_{W^{1,2}}^2\leq \frac{C(1+\delta_0^{-2})}{m}.
\end{equation}

Now, since $\Delta \phi_1=0$ in $\Omega$ with $\phi_1=0$ on $\partial D_{r_m}(X)$ and $\frac{\partial \phi_1}{\partial \nu}=\sigma_1^D(\Omega)\phi_1$ on $\partial\Omega\setminus \partial D_{r_m}(X)$, we compute
\begin{eqnarray*}
\sigma_1^D(\Omega)\int_{\partial\Omega}\phi_1f&=&\int_{\Omega}\langle df,d\phi_1\rangle\\
%&=&\int_{\partial \mathbb{D}\cap \{|y|>1/m\}}|y|\phi_1-\int_S\phi_1\\
&=&\int_{\partial \DD\cap\{|y|>1/m\}}\left(f+\frac{1}{m}\right)\phi_1-\int_S\phi_1.
\end{eqnarray*}
For the last term, combining \eqref{tr.est} and \eqref{f.phi1.close} gives an estimate of the form
$$\int_S\phi_1=\int_S(\phi_1-f)\leq \frac{C}{\sqrt{m}}\|d(\phi-f_1)\|_{L^2(\Omega)}\leq \frac{C'(1+\delta_0^{-1})}{m},$$
while it's also clear that
$$\int_{\partial \DD\cap\{|y|>1/m\}}\frac{1}{m}\phi_1\leq \frac{C'}{m}.$$
Applying these two estimates in the preceding computation, we see that
$$|\sigma_1^D(\Omega)-1|\int_{\partial\Omega}\phi_1^2=|\sigma_1^D(\Omega)-1|\int_{\partial\Omega}\phi_1f\leq \frac{C(1+\delta_0^{-1})}{m}.$$
In particular, since
$$\|\phi_1\|_{L^2(\partial\Omega)}^2\geq \|f\|_{L^2(\partial\Omega)}^2-\frac{C(1+\delta_0^{-2})}{m}\geq \frac{1}{2}-\frac{C(1+\delta_0^{-2})}{m}$$
for $m$ sufficiently large, it will follow that $\sigma_1^D(\Omega)$ satisfies an estimate of the desired form
$$|\sigma_1^D(\Omega)-1|\leq \frac{C'}{m},$$
proving \eqref{sig.d.bd} and completing the proof of the proposition, once we confirm that \eqref{s3.gap} holds for a fixed $\delta_0>0$ independent of $m$ as $m\to\infty$.

To prove \eqref{s3.gap}, first observe that on the half-disk $\DD_+=\DD\cap \{y>0\}$, we have $\sigma_1^D(\DD_+)=1$ and 
\begin{equation}\label{s2.model}
\sigma_2^D(\DD_+)=1+2\delta_0>\sigma_1^D(\DD_+)
\end{equation}
for the mixed Steklov-Dirichlet problem with Dirichlet condition on $\ell$ and Steklov on $\partial \DD_+\setminus \ell$. Now, setting 
$$Y:=\{\phi\in W^{1,2}(\DD)\mid \phi\equiv 0\text{ on }D_{r_m}(X)\}$$
and
$$Z_+:=\{\psi \in W^{1,2}(\DD_+)\mid \psi \equiv 0\text{ on }\ell\},$$
we define a linear map $T_+: Y\to Z_+$ as follows. Given $\phi\in Y$, let 
$$h\colon\DD_+\to \mathbb{R}$$
minimize energy with respect to the constraint $h|_{\ell}=\phi_{\ell}$, so that $\Delta h=0$ on $\DD_+$ and $\frac{\partial h}{\partial\nu}=0$ on $\partial \DD_+\setminus \ell$, and set 
$$T_+(\phi):=\phi-h.$$
An argument identical to the proof of \eqref{tr.est} then gives the estimate
$$\|h\|_{L^2(\ell)}^2=\|\phi\|_{L^2(\ell)}^2\leq \frac{C}{m}\|d\phi\|_{L^2(\DD_+)}^2,$$
while a simple contradiction argument using compactness of the trace embedding $W^{1,2}(\DD_+)\to L^2(\partial \DD_+)$ gives an estimate of the form
$$\|h\|_{L^2(\partial \DD_+)}^2\leq \delta \|dh\|_{L^2(\DD_+)}^2+C_{\delta}\|h\|_{L^2(\ell)}^2\leq \left(\delta+\frac{C_{\delta}}{m}\right)\|d\phi\|_{L^2(\DD_+)}^2$$
for any fixed $\delta>0$. With these estimates in hand, observe next that
$$\|d(T_+\phi)\|_{L^2(\DD^+)}^2=\|d\phi\|_{L^2(\DD^+)}^2-\|dh\|_{L^2(\DD^+)}^2\leq \|d\phi\|_{L^2(\DD^+)}^2.$$
%while
%$$\|T_+\phi\|_{L^2(\partial \DD^+)}\geq \|\phi\|_{L^2(\partial \DD_+\setminus \ell)}-(\delta+C_{\delta}/m)^{1/2}\|d\phi\|_{L^2(\DD_+)}.$$
In particular, we see that
\begin{eqnarray*}
\|\phi\|_{L^2(\partial \DD_+\setminus \ell)}^2&\leq & (1+[\delta+C_{\delta}/m]^{1/2})\|T_+\phi\|_{L^2(\partial \DD^+)}^2\\
&&+2(\delta+C_{\delta}/m)^{1/2}\|d\phi\|_{L^2(\DD_+)}^2.
\end{eqnarray*}
Likewise, on the lower half-disk $\DD_-=\DD\cap\{y<0\}$, we set
$$Z_-:=\{\psi \in W^{1,2}(\DD_-)\mid \psi \equiv 0\text{ on }\ell\}$$
and define an identical map $T_-:Y\to Z_-$, which we can combine to get a map $T:Y\to Z$ to the space
$$Z:=\{\psi \in W^{1,2}(\mathbb{D})\mid \psi\equiv 0\text{ on }\ell\},$$
satisfying 
\begin{equation}\label{t.ener.bd}
\|d(T\phi)\|_{L^2(\DD)}^2\leq \|d\phi\|_{L^2(\DD)}^2
\end{equation}
and
\begin{equation}\label{t.est}
\|\phi\|_{L^2(\partial \DD)}^2\leq (1+[\delta+C_{\delta}/m]^{1/2})\|T\phi\|_{L^2(\partial \DD)}^2+2(\delta+C_{\delta}/m)^{1/2}\|d\phi\|_{L^2(\DD)}^2.
\end{equation}

Next, by virtue of \eqref{s2.model} and its counterpart $\sigma_2^D(\DD_-)=1+2\delta_0$ on $\DD_-$, we see that 
\begin{equation}\label{s3.char}
\sigma_3^D(\DD_-\cup \DD_+)=\sigma_2^D(\DD_{\pm})=1+2\delta_0;
\end{equation}
that is, for any $3$-dimensional subspace $W\subset Z$, there is a nonzero $\psi\in W$ for which
$$(1+2\delta_0)\|\psi\|_{L^2(\partial\DD)}^2\leq \|d\psi\|_{L^2(\DD)}^2.$$
Now, since every $\phi\in Y$ vanishes on $D_{r_m}(X)$ by assumption, it follows from the definition of $T$ and a simple unique continuation argument that $T: Y\to Z$ must be injective. In particular, if $V\subset Y$ is a $3$-dimensional subspace, then $T(V)$ defines a $3$-dimensional subspace in $Z$, so it follows from \eqref{s3.char} that there exists a nonzero $\phi\in V$ such that
$$(1+2\delta_0)\|T\phi\|_{L^2(\partial\DD)}^2\leq \|d(T\phi)\|_{L^2(\DD)}^2.$$
On the other hand, combining this with \eqref{t.est} and \eqref{t.ener.bd}, we see that this $\phi\in V$ satisfies
$$\|\phi\|_{L^2(\partial \DD)}^2\leq \frac{(1+[\delta+C_{\delta}/m]^{1/2})}{1+2\delta_0}\|d\phi\|_{L^2(\DD)}^2+2(\delta+C_{\delta}/m)^{1/2}\|d\phi\|_{L^2(\DD)}^2.$$
Fixing $\delta$ sufficiently small relative to $\delta_0$ and taking $m$ sufficiently large, we can arrange that
$$\frac{(1+[\delta+C_{\delta}/m]^{1/2})}{1+2\delta_0}+2(\delta+C_{\delta}/m)^{1/2}\leq \frac{1}{1+\delta_0}.$$
It then follows from the preceding argument that every $3$-dimensional subspace $V\subset Y$ contains a nontrivial $\phi\in V$ for which
$$(1+\delta_0)\|\phi\|_{L^2(\partial \DD)}^2\leq \|d\phi\|_{L^2(\DD)}^2;$$
in other words, $\sigma_3^D(\Omega_m)\geq 1+\delta_0$, confirming \eqref{s3.gap} and completing the proof. \qedhere
\end{proof}

Combining the ingredients above, we can now prove the lower bounds of Theorem \ref{main.stek.bds} and Theorem \ref{max.stek.bds}, which we collect in the following theorem.

\begin{theorem}
For every compact basic reflection surface $(N,\Gamma)$ with boundary, we have
$$\Sigma_1(N,\Gamma)\geq \max\left\{4\pi-e^{-ck},4\pi-\frac{C}{\gamma}\right\}$$
for some constants $C,c>0$, where $k$ is the number of boundary components of $N$ and $\gamma$ is the genus. Moreover, in the case where $\Gamma$ corresponds to a single reflection $\langle \tau\rangle$, we have
$$\Sigma_1(N,\langle \tau\rangle)\geq 4\pi-e^{-c(k+\gamma)}.$$
\end{theorem}
\begin{proof}
Writing $\Gamma=\tau\times G$ where $G\leq O(2)$ is a reflection group--i.e., $G=1,\mathbb{Z}_2$, or $D_n$ for some $n\geq 3$--by Lemma \ref{stek.double}, it suffices to prove a lower bound of the form
$$\bar{\sigma}(\Omega)\geq \max\left\{2\pi-e^{-ck},2\pi-\frac{C}{m}\right\}$$
for \emph{some} $G$-invariant domain $\Omega\subset \mathbb{D}$ with $k$ boundary components intersecting $\partial \mathbb{D}$ and $m$ boundary components in the interior of $\mathbb{D}$, with fundamental domain of prescribed topological type.

Given the symmetries of the construction in Proposition \ref{Pstekk}, we observe that for any reflection group $G\leq O(2)$, every topological type of $G$-invariant domain $\Omega\subset \mathbb{D}$ with $k$ total boundary components intersecting $\partial \mathbb{D}$ can be realized by removing a suitable collection of arbitrarily small disks from the interior of the domain $\Omega_k$ described in Proposition \ref{Pstekk}. This observation, together with Proposition \ref{Pstekk}, gives the lower bound
$$\Sigma_1(N,\Gamma)\geq 4\pi-e^{-ck}.$$

Next, we turn to basic reflection pairs $(N,\Gamma=\langle \tau\rangle\times G)$ given by doubling $G$-invariant domains $\Omega$ with $m$ interior boundary components. If $G=1$, so that there is no constraint on the locations of the $m$ boundary components, Proposition \ref{pstekmbdry} gives a domain of the desired type for which
$$\sigmabar(\Omega)\geq 2\pi-e^{-cm},$$
so that
$$\Sigma_1(N,\langle \tau\rangle)\geq 4\pi-e^{-cm},$$
giving the desired bound in this case. If $G=\mathbb{Z}_2$, with a majority of the $m$ interior boundary components of $\Omega$ lying away from the fixed point set, taking the domain from Proposition \ref{pstekmbdry} (with $m$ replaced by $m'\geq m/2$) and removing a required number of arbitrarily small disks centered on a line through the origin, we again obtain a lower bound of the form $\Sigma_1(N,\Gamma)\geq 4\pi-e^{-c'm}\geq 4\pi-\frac{C}{m}$ in this case. If, on the other hand, a majority of the boundary components are prescribed to lie on a fixed point set for $G=\mathbb{Z}_2$, we instead remove a prescribed number of arbitrarily small disks from the domain in Proposition \ref{pstekmaxis} to obtain a lower bound of the form $\Sigma_1(N,\Gamma)\geq 4\pi-\frac{C}{m}.$

The case of $G=D_n$-symmetric domains with $3\leq n\leq n_0$ can be handled similarly, appealing either to Proposition \ref{pstekmaxis} when a majority of the boundary components in a fundamental domain are prescribed to lie along the fixed point sets of generators of $G$, and appealing instead to Proposition \ref{pstekmbdry} when a majority of boundary components lie in the interior of a fundamental domain. 

Finally, consider the case $G=D_n$ with $n\geq n_0$. If we consider domains $\Omega$ with a majority of the $m$ interior boundary components prescribed to lie in the interior of a fundamental domain, Proposition \ref{pstekmbdry} can again be used to obtain a lower bound of the form $\Sigma_1(N,\Gamma)\geq 4\pi-e^{-c'm}\geq 4\pi-\frac{C}{m}$ in this case. Alternatively, suppose that at least $\frac{m}{4}$ of the $m$ boundary components correspond to the $D_n$-orbit of $a$ boundary components prescribed to lie along the fixed point set of a generating reflection in the fundamental domain. The $m\leq 4na$, and we subdivide into two more cases: if $a\leq n-1$, then applying Proposition \ref{pstekmbdry} with $m=n$ and deleting an additional collection of arbitrarily small disks at the required locations, we obtain a lower bound of the form
$$\Sigma_1(N,\Gamma)\geq 4\pi-e^{-cn}\leq 4\pi-e^{-c'\sqrt{m}}\leq 4\pi-\frac{C}{m};$$
if, on the other hand, $n\leq a$, then removing a collection of arbitrarily small disks from the domain considered in Proposition \ref{pstekwedges} gives a domain of the desired type satisfying $\bar{\sigma}(\Omega)\geq 2\pi-e^{-c'a}\geq 2\pi-e^{-c''\sqrt{m}},$ so that
$$\Sigma_1(N,\Gamma)\geq 4\pi-e^{-c\sqrt{m}}\geq 4\pi-\frac{C}{m}$$
in this case as well.
\end{proof}

\begin{remark}
In the arguments above, our techniques for estimating $\Sigma_1(N,\Gamma)$ from below in terms of the genus are evidently non-sharp in many cases. Indeed, analogous to the Scherk-type/generic dichotomy in the closed case, it seems plausible that the lower bound $\Sigma_1(N,\Gamma)\geq 4\pi-\frac{C}{\gamma}$ via the genus is qualitatively sharp only in the case where $\Gamma\cong \mathbb{Z}_2\times \mathbb{Z}_2$, and all of the necks and (at most two) boundary components of $N$ are forced to intersect the fixed point sets of both reflections. We do not pursue the problem of refining these estimates further in the present paper.
\end{remark}

\section{Stability and upper bounds}

Let $(N_k,g)$ be a compact genus zero surface with $k$ boundary components, which we can identify isometrically with a domain $\Omega\subset \Sph^2$ of the form
\[
\Omega=\Sph^2\setminus \mathcal{D}:=\Sph^2\setminus \bigcup_{j=1}^kD_{r_j}(x_j)
\]
with a conformal metric $g=\rho g_0$. Moreover, we can arrange that the standard inclusion map $I\colon \Sph^2\to \mathbb{R}^3$ satisfies
\[
\int_{\Omega}I dv_g=0.
\]
Let $\bar{\mu}(N_k,g):=\min\left\{\bar{\lambda}_1^{D}(N_k,g),\bar{\lambda}_1^{N}(N_k,g)\right\}$, and consider the gap
\[
\delta(N_k,g):=8\pi-\bar{\mu}(N_k,g).
\]
Using techniques developed in \cite{KNPS, KSDuke}, we prove the following.
%As discussed in the earlier note in the folder, straightforward applications of the Hersch trick and the Cauchy-Schwarz inequality give the following.

\begin{lemma}\label{stab.ineqs}
The set $\mathcal{D}=\bigcup\limits_{j=1}^kD_{r_j}(x_j)$ satisfies an estimate of the form
\[
\area(\mathcal{D},g_0)\leq \frac{1}{2}\delta(N_k,g),
\]
and for any $\psi\in W^{1,2}(\Omega),$
\begin{equation}\label{meas.close}
\int_{\Omega}\psi\left(2dv_{g_0}-\lambda_1^{N}(g)dv_g\right)\leq C\sqrt{\delta(N_k,g)}\|\psi\|_{W_{g_0}^{1,2}(\Omega)}.
\end{equation}
\end{lemma}

\begin{proof}
Since all of the coordinate functions $x_1,x_2,x_3$ are test functions for the Rayleigh quotient definition of $\lambda_1^{N}(N_k,g)$, we have
\begin{equation*}
\begin{split}
\area(\Omega,g_0)=\frac{1}{2}\int_{\Omega}|dI|_{g}^2\,dv_{g}
&\geq \frac{1}{2}\lambda_1^{N}(N_k)\int_{\Omega}|I|^2dv_{g}\\
&\geq  \frac{1}{2}\bar{\mu}(N_k,g) = 4\pi-\frac{1}{2}\delta(N_k,g),
\end{split}
\end{equation*}
giving the estimate for $\area(\mathcal{D},g_0)$.

Moreover, on the space of maps $F\in W^{1,2}(\Omega_k,\mathbb{R}^3)$ with 
\[
\bar{F}:=\frac{1}{\area(N_k,g)}\int_{\Omega_k}Fdv_{g}=0,
\]
the quadratic form
\[
Q_k(F,G):=\int_{\Omega}\langle dF,dG\rangle_{g_0}dv_{g_0}-\lambda_1^{N}(N_k,g)\int_{\Omega}\langle F,G\rangle dv_{g}
\]
is positive semidefinite. The Cauchy-Schwarz inequality then gives
\[
Q_k(F,I)\leq Q_k(F,F)^{1/2}Q_k(I,I)^{1/2}\leq \|dF\|_{L^2(\Omega)}\sqrt{\delta(N_k,g)}.
\]
In fact, even for maps with $\bar{F}\neq 0$, since $Q_k(\bar{F},I)=0$, it is easy to see that
\[
Q_k(F,I)=Q_k(F-\bar{F},I)\leq \|dF\|_{L^2(\Omega)}\sqrt{\delta(N_k,g)}.
\]
Taking $F=\psi I$ for some $\psi\in W^{1,2}(\Omega)$, 
note that 
\[
\langle d(\psi I),dI\rangle_{g_0}=2\psi,
\qquad
|d(\psi I)|_{g_0}^2=2\psi^2+|d\psi|_{g_0}^2,
\]
and $\langle (\psi I),I\rangle=\psi$, so that
\[
Q_k(\psi I,F)=\int_{\Omega}2\psi dv_{g}-\lambda_1^{N}(N_k,g)\int_{\Omega}\psi dv_{g_0},
\]
and $\|d(\psi I)\|_{L^2(\Omega)}\leq \|\psi\|_{W^{1,2}_{g_0}(\Omega)}$ so that~\eqref{meas.close} follows.
\end{proof}

As a corollary, we have the following, which together with Lemma \ref{double.lem} completes the proof of the first estimate in Theorem \ref{s3.ubd}.

\begin{cor}
\label{cor:Lap_lbd}
There exists a universal constant $c>0$ such that
\[
\delta(N_k,g)\geq e^{-ck}.
\]
\end{cor}
\begin{proof}
Writing $\delta=\delta(N_k,g)$, the area estimate in Lemma~\ref{stab.ineqs} implies
\[
\sum_{j=1}^kr_j^2\leq C\delta.
\]
Set $S=\{x_1,x_2,\ldots,x_k\}$, and without loss of generality, suppose 
\[
1>\sqrt{C\delta}>r_1\geq r_2\geq \cdots\geq r_k.
\]
Consider the log cutoff function $\phi \in \mathrm{Lip}(\Omega)$ given by 
\begin{equation*}
\phi(x)= 
\begin{cases}
1\quad&\text{ for }\mathrm{dist}(x,S)\geq \sqrt{r_1},\\
\frac{2}{|\log r_1|}\log(d(x,S)/r_1)&\text{ for }\mathrm{dist}(x,S)\in [r_1,\sqrt{r_1}],\\
0\quad&\text{ for }\mathrm{dist}(x,S)\leq r_1,
\end{cases}
\end{equation*}
where distance is measured with respect to the round metric on $\Sph^2$.
Then standard computations give
\[
\|d\phi\|_{L^2(\Omega)}^2\leq \frac{Ck}{\log(1/\delta)},
\]
while $0\leq \phi\leq 1$, and $\phi\equiv 1$ on the complement of the set $\bigcup\limits_{j=1}^kD_{\sqrt{r_1}}(x_j)$, so that
\[
\|\phi\|_{L^2_{g_0}}^2\geq 4\pi-Ckr_1\geq 4\pi-C'k\sqrt{\delta}.
\]
Now, appealing to the second inequality of Lemma \ref{stab.ineqs} with $\psi=\phi^2$, we deduce that
\[
\lambda_1^{N}(g)\int_{\Omega}\phi^2dv_g\geq 2\|\phi\|_{L^2_{g_0}}^2-C\sqrt{\delta}\|\phi^2\|_{W_{g_0}^{1,2}}
\geq 8\pi-C''k\sqrt{\delta}.
\]
If $8\pi-C''k\sqrt{\delta}\leq 0$, then the conclusion of the corollary holds. Assuming the converse, one has
\[
\lambda_1^{D}(N_k,g)\leq \frac{\|d\phi\|_{L^2(\Omega)}^2}{\|\phi\|_{L^2(\Omega,g)}^2}
\leq \frac{Ck}{\log(1/\delta)}\cdot \frac{\lambda_1^{N}(N_k,g)}{8\pi-C''k\sqrt{\delta}}.
\]
For the metric $g$ maximizing $\bar{\mu}(N_k,g)$--i.e., minimizing $\delta(N_k,g)$--we have $\lambda_1^{D}=\lambda_1^{N}$, so setting $\delta_k=\min\left\{\delta(N_k,g)\mid g\in \Met(N_k)\right\}$, we can divide through by $\lambda_1^{N}=\lambda_1^{D}$ above to see that
\[
1\leq \frac{Ck}{\log(1/\delta_k)}\cdot \frac{1}{8\pi-C''\sqrt{\delta_k}},
\]
from which the desired estimate follows easily.\qedhere

\end{proof}

\subsection{Convergence to equators}\label{eq.cvg-1}

Now, let $(N_k,g_k)\subset \mathbb{S}^3_+$ be one half of a reflection-symmetric minimal surface embedded by first eigenfunctions in $\mathbb{S}^3$, conformally equivalent to a balanced domain $\Omega_k=\mathbb{S}^2\setminus \mathcal{D}_k$ with $\mathcal{D}_k=\bigcup\limits_{i=1}^kD_{r_i}(x_i)$ as in the previous subsection, with
\[
8\pi-\bar{\mu}(N_k,g_k)=:\delta_k,
\]
and $\Phi_k\colon\Omega_k\to N_k\subset \mathbb{S}^3_+$ giving the conformal harmonic embedding. In particular, note that $\lambda^N_1(N_k,g_k)=\lambda_1^D(N_k,g_k) = 2$.

As in the proof of Lemma \ref{stab.ineqs}, we begin with the observation that the standard inclusion $I\colon\Sph^2\to \mathbb{R}^3$ satisfies
\[
Q_k(\Phi,I)\leq \sqrt{\delta_k}\|d\Phi\|_{L^2(\Omega_k)}
\]
for any $\Phi\colon \Omega_k\to \mathbb{R}^3$, where
\[
Q_k(\Phi,G):=\int_{\Omega_k}\langle d\Phi,dG\rangle_{g_k}-2\langle \Phi,G\rangle\, dv_{g_k}.
\]
In particular, letting $\Phi$ denote the $L^2(dv_{g_k})$ projection of $I|_{\Omega_k}$ onto the complement of the first Neumann eigenspace for $\Delta_{g_k}$, we deduce that
\begin{equation}\label{id.eigen}
(1-2/\lambda_4^{N}(N_k))^2\|\Phi\|_{W^{1,2}(g_k)}^2\leq C\delta_k.
\end{equation}

Next, we show that the coefficient of the left-hand side of \eqref{id.eigen} is bounded from below when $\delta_k\ll1/k$.

\begin{lemma}\label{l4.gap}
There are constants $c>0$ and $\epsilon>0$ such that if $(N_k,g_k)$ is a sequence satisfying the preceding assumptions with $\delta(N_k,g_k)\leq \frac{\epsilon}{k}$, then $\lambda_4^{N}(N_k,g_k)-2>c$.
\end{lemma}
\begin{proof}
Let $\phi_{k,j}$ be the $j$th Neumann eigenfunction for the Laplacian on $(N_k,g_k)\cong (\Omega_k,g_k)$, normalized so that
\[
\int_{\Omega_k}\phi_{k,j}\phi_{k,l}\,dv_{g_k}=\delta_{jl}.
\]
Note that $\phi_{k,j}$ corresponds to an even eigenfunction on the closed minimal surface $M$ in $\Sph^3$ obtained by doubling $N_k$, and a straightforward variant of the mean value computations in \cite[Section 3.3]{CM} for the function $\phi_{k,j}^2$ lead to an $L^{\infty}$ estimate of the form
\begin{equation}\label{ef.mvi}
\phi_{k,j}(p)^2\leq Ce^{\lambda_j}\int_{\Omega_k}\phi_{k,j}^2\,dv_{g_k}=Ce^{\lambda_j}.
\end{equation}
For completeness let us provide a proof of~\eqref{ef.mvi}. Fix a point $p\in M$ and consider the vector field $X_p(x) = (x-p) = -\nabla h_{p,r}$, where $h_{p,r}(x)=\frac{1}{2}(r^2-|x-p|^2)$ and note that
\[
\langle X_p(x),x\rangle =1 - \langle x,p\rangle = \frac{1}{2}|x-p|^2\text{ on } M.
\]
One then has
\begin{equation*}
\begin{split}
2&\int_{M\cap B_r(p)}\phi_{k,j}^2\,dv_g=\int_{N\cap B_r(p)}\phi_{k,j}^2\mathrm{div}_M(X_p)\,dv_g=\\
&\int_{M\cap B_r(p)}\mathrm{div}_M(\phi_{k,j}^2X_p)+\langle \nabla\phi_{k,j}^2,\nabla h_{p,r}\rangle\,dv_g=\\
&\int_{M\cap B_r(p)}\mathrm{div}_M(\phi_{k,j}^2X_p+h_{p,r}\nabla \phi_{k,j}^2)+2\lambda_j h_{p,r} \phi_{k,j}^2 - 2h_{p,r}\phi_{k,j}^2\,dv_g\leq\\
&\int_{M\cap \partial B_r(p)}\phi_{k,j}^2|(x-p)^T|\,ds_g+\int_{M\cap B_r(p)}\left(\phi_{k,j}^2|x-p|^2+2\lambda_jh_{p,r}\phi_{k,j}^2\right)\,ds_g\leq\\
r&\frac{d}{dr}\left(\int_{M\cap B_r(p)}\phi_{k,j}^2\,dv_g\right)+\left(1+\lambda_j \right) r^2\int_{M\cap B_r(p)}\phi_{k,j}^2\,ds_g.
\end{split}
\end{equation*}
Dividing by $r^3$ we obtain
\[
\frac{d}{dr}\left(\frac{1}{r^2}\int_{M\cap B_r(p)}\phi_{k,j}^2\,dv_g\right)\geq - (1+\lambda_j)\frac{1}{r}\int_{M\cap B_r(p)}\phi_{k,j}^2\,dv_g.
\]
Setting $g(r) = \frac{1}{r^2}\int_{M\cap B_r(p)}\phi_{k,j}^2\,dv_g$ this implies that for $r\leq 1$
\[
\frac{d}{dr}(\ln g(r))\geq -(1+\lambda_j)
\]
Integrating on $[0,1]$ and noting that $g(0) = \pi \phi^2_{k,j}(p)$ completes the proof of~\eqref{ef.mvi}.

Now, to prove the lemma, assume for a contradiction that there exists a sequence $\eps_k\to 0$ and $(N_k,g_k)$ as above such that $\delta_k<\eps_k/k$ but $\lambda_4^{N}(N_k,g_k)\to 2$ as $k\to\infty$. We will obtain a contradiction by showing that $\phi_{k,j}$ converge in $L^2(\Sph^2)$ as $k\to\infty$ to eigenfunctions $\phi_j$ of the standard Laplacian with eigenvalue $\lim\limits_{k\to\infty} \lambda_j^{N}(N_k,g_k)$, such that
\[
\int_{\Sph^2}\phi_j\phi_l\,dv_{g_0}=\delta_{jl},
\]
forcing $\lim \lambda_j^{N}(N_k,g_k)$ to coincide with the $j$-th eigenvalue of the Laplacian on $\Sph^2$, giving the desired contradiction.

First, we observe that the sequence $\{\lambda_j^{N}(N_k,g_k)\}_k$ is bounded. To show this, use the restrictions of eigenfunctions on $\Sph^2$ to 
$\Omega_k$ together with~\eqref{meas.close} to conclude that $\limsup \lambda_j^{N}(N_k,g_k)\leq \lambda_j(\Sph^2,g_0)$.
Then, we check that, for a fixed $j$, the sequence $\phi_{k,j}$ extended by $0$ to $L^2(\Sph^2)$ is precompact in $L^2(\Sph^2,g_0)$. Using linear cutoff functions, it is easy to find $\rho_k\in \mathrm{Lip}_c(\Omega_k; [0,1])$ such that $\rho_k\equiv 1$ on $\Sph^2\setminus \bigcup_{i=1}^kD_{2r_i}(x_i)$ and 
\[
\int_{\Sph^2}|d\rho_k|_{g_0}\,dv_{g_0}\leq C\sum_{i=1}^kr_i\leq C\sqrt{k\delta_k}<C\sqrt{\eps_k}\leq C,
\]
which together with \eqref{ef.mvi} and the obvious bound 
\[
\int_{\Omega_k}|d\phi_{k,j}|_{g_0}\,dv_{g_0} \leq C\sqrt{\lambda_j^{N}(N_k,g_k)}
\]
 yields the uniform $W^{1,1}(\Sph^2,g_0)$ bound
\[
\|\rho_k\phi_{k,j}\|_{W^{1,1}(\Sph^2,g_0)}\leq C_j
\]
for $\rho_k\phi_k$. By the compactness of the embedding $W^{1,1}(\Sph^2)\to L^1(\Sph^2)$, it follows that $\rho_k\phi_{k,j}$ is precompact in $L^1(\Sph^2)$, as $k\to\infty$, and since
\[
\int_{\Omega_k}(1-\rho_k)|\phi_{k,j}|\,dv_{g_0}\leq C\|\phi_{k,j}\|_{L^{\infty}}\sum_{i=1}^k r_i^2\leq C_j\delta_k,
\]
we deduce that $\phi_{k,j}$ is precompact in $L^1(\Sph^2)$ as well. And since \eqref{ef.mvi} gives a uniform $L^{\infty}$ bound for $\phi_{k,j}$ as $k\to\infty$, it follows that $\phi_{k,j}$ is in fact precompact in $L^p(\Sph^2)$ for every $p\in [1,\infty)$, in particular $p=2$. Passing to a subsequence, we let $\phi_j\in L^{\infty}(\Sph^2)$ be the $L^2$-limit of $\phi_{k,j}$. 

Next, we argue that the sequence $\{\phi_j\}$ is orthonormal in $L^2(\Sph^2,g_0)$. Indeed, it follows from~\eqref{meas.close} that
\begin{equation*}
\begin{split}
&\left|\int_{\Omega_k}\phi_{k,j}\phi_{k,l}\,dv_{g_0}-\delta_{jl}\right|=\left|\int_{\Omega_k}\phi_{k,j}\phi_{k,l}\,(dv_{g_0}-dv_{g_k})\right|\leq\\
&C\sqrt{\delta_k}\|\phi_{k,j}\phi_{k,l}\|_{W^{1,2}(\Omega_k,g_0)}
\leq C_{jl}\sqrt{\delta_k},
\end{split}
\end{equation*}
where in the last line we used the $L^{\infty}$ control \eqref{ef.mvi} to get uniform $W^{1,2}(\Sph^2,g_0)$ control on the product $\phi_{k,j}\phi_{k,l}$. Since the right-hand side vanishes as $k\to\infty$, it follows from the strong $L^2(\Sph^2)$ convergence $\phi_{k,j}\to \phi_j$ that
\[
\int_{\Sph^2}\phi_j\phi_l\,dv_{g_0}=\delta_{jl},
\]
as claimed.

We finish the proof by showing that $\phi_j$ is a Laplace eigenfunction with eigenvalue $\lim\limits_{k\to\infty}\lambda_j^{N}(N_k,g_k)$. For any fixed $\psi\in C^{\infty}(\Sph^2)$, we compute
\begin{eqnarray*}
\int \phi_j (\Delta \psi)&=&\lim_{k\to\infty}\int_{\Omega_k}\phi_{k,j}(\Delta_{g_0} \psi)\,dv_{g_0}\\
&=&\lim \int_{\Omega_k}(d^*(\phi_{k,j}d\psi)+\langle d\psi,d\phi_{k,j}\rangle)\,dv_{g_0}\\
&=&\lim \int_{\Omega_k}\langle d\psi, d\phi_{k,j}\rangle\,dv_{g_0}-\lim \int_{\partial\Omega_k}\phi_{k,j}\frac{\partial \psi}{\partial\nu}\,ds_{g_0}.
\end{eqnarray*}
Now, the second term in the last line is bounded in absolute value by $\lim_{k\to\infty}\|\phi_{k,j}\|_{L^{\infty}}\|d\psi\|_{C^0}|\partial\Omega_k|$, which together with the $L^{\infty}$ bound \eqref{ef.mvi} and the estimate
\[
|\partial\Omega_k|\leq C\sum_{i=1}^kr_i\leq C\sqrt{k\epsilon_k}<C\sqrt{\delta_k}
\]
implies
\[
\int_{\Sph^2} \phi_j (\Delta \psi)\, dv_{g_0} = \lim_{k\to\infty}\int_{\Omega_k}\langle d\psi,d\phi_{k,j}\rangle_{g_k}\,dv_{g_k}=\lim_{k\to\infty}\lambda_j^{N}(N_k,g_k)\int_{\Omega_k}\psi \phi_{k,j}\,dv_{g_k}.
\]
Finally, since $\psi \phi_{k,j}$ is uniformly bounded in $W^{1,2}$ as $k\to\infty$, we can appeal to \eqref{meas.close} to see that
\[
\lim_{k\to\infty}\int_{\Omega_k}\psi \phi_{k,j}\,dv_{g_k}=\lim_{k\to\infty}\int_{\Omega_k}\psi \phi_{k,j}\,dv_{g_0}=\int_{\Sph^2}\psi \phi_j\,dv_{g_0},
\]
and, therefore,
\[
\int_{\Sph^2} \phi_j(\Delta_{g_0} \psi)\,dv_{g_0}=\left(\lim_{k\to\infty}\lambda_j^{N}(N_k,g_k)\right)\int_{\Sph^2} \phi_j \psi\,dv_{g_0},
\]
as claimed.
\end{proof}

Let $\Psi_k = (F^1_k,F_k^2,F_k^3)$ be the map given by the first three coordinate functions of the minimal embedding $F_k\colon\Omega_k\to \Sph^3_+$, so that the components of $\Psi_k$ form a basis of the $\lambda_1^N(N_k,g_k)$-eigenspace. It then follows from~\eqref{id.eigen} that there exists a matrix $A_k$ such that
\begin{equation}
\label{ineq:Ak_proj}
\|I-A_k\Psi_k\|_{W^{1,2}(\Omega_k,g_k)}^2\leq C\delta_k.
\end{equation}
In particular, this implies that $A_k\in GL_3(\mathbb{R})$ for large enough $k$. Indeed, otherwise there exists $\xi_k\in\Sph^2$ such that $\langle I,\xi_k\rangle$ is pointwise orthogonal to $A_k\Psi_k$, so that by Lemma~\ref{stab.ineqs} one has
\[
\|I-A_k\Psi_k\|_{W^{1,2}(\Omega_k,g_k)}^2\geq \|d\langle I,\xi_k\rangle\|^2_{L^2(\Omega_k)} = \frac{2}{3}\area(\Omega_k,g_0)\geq\frac{8\pi-\delta_k}{3},
\]
which is a contradiction for large values of $k$.

Thus, we can rewrite~\eqref{ineq:Ak_proj} as 
\[
\|\Psi_k-A_k^{-1}I\|_{W^{1,2}(\Omega_k,g_k)}^2\leq C\delta_k|A_k^{-1}|^2,
\]
and recalling that for any $3\times 3$ matrix $M$, the Hilbert-Schmidt norm $|M|$ can be computed by
\[
|M|^2=\frac{1}{2\pi}\int_{x\in \Sph^2}\left|M|_{T_{x}\Sph^2}\right|^2\,dv_{g_0},
\]
it follows that
\begin{equation*}
\begin{split}
\left|A_k^{-1}\right|^2 &\leq \frac{1}{2\pi}\int_{\Omega_k}|A_k^{-1}dI|_{g_0}^2\,dv_{g_0} + C\delta_k\left|A_k^{-1}\right|^2 \\
&\leq C'\delta_k\left|A_k^{-1}\right|^2+C'\int_{\Omega_k}|d\Psi_k|_{g_0}^2\,dv_{g_0},
\end{split}
\end{equation*}
therefore, $|A_k^{-1}|\leq C$ as $\epsilon_k\to 0$. In particular, writing
\[
\alpha_k:=|A_k^{-1}dI|_{g_0}\leq 2\left|A_k^{-1}\right|^2\leq C,
\]
we deduce that
\begin{equation}\label{key.est}
\||d\Psi_k|_{g_0}-\alpha_k\|_{L^2(\Omega_k)}^2\leq C\delta_k.
\end{equation}

As a corollary of these estimates and Lemma \ref{stab.ineqs}, we have the following result, which implies that $N_k$ must lie close to the equator $\{x_4=0\}$ if $\delta_k=o(1/k)$.

\begin{prop}\label{ht.bd.prop}
Let $x_4$ be the height coordinate in $\Sph^3_+=\left\{x\in \Sph^3\mid x_4\geq 0\right\}$. Then for a free boundary minimal surface $(N_k,g_k)\subset \mathbb{S}^3_+$ as above, we have
\[
\int_{N_k}x_4\,dv_{g_k}\leq C\sqrt{\delta_k k}.
\]
\end{prop}
\begin{proof}
Let $\phi_k=x_4\circ F_k\in W_0^{1,2}(\Omega_k)$ under the conformal identification $F_k\colon \Omega_k\to N_k$, so that $\int_{\Omega_k} \phi_k =\int_{N_k}x_4$. Using linear cutoffs, it is easy to construct a function
\[
\psi\colon \Sph^2\to [0,1]
\]
such that $\psi\equiv 1$ on $\mathcal{D}=\bigcup_{i=1}^kD_{r_i}(x_i)$, $\psi \in C_c^{\infty}\left(\bigcup_{i=1}^kD_{2r_i}(x_i)\right)$, and 
\[
\|d\psi\|_{L^2}^2\leq Ck.
\]
Since $1-\psi=0$ on $\partial\Omega_k$, we then have
\[
0=\int_{\Omega_k}\mathrm{div}([1-\psi]\nabla\phi_k)\,dv_{g_k}=\int_{\Omega_k}2(1-\psi)\phi_k\,dv_{g_k}-\int_{\Omega_k}\langle d\psi,d\phi_k\rangle,
\]
so that by Lemma~\ref{stab.ineqs} one has
\begin{equation*}
\begin{split}
&\int_{\Omega_k}\phi_k\,dv_{g_k}=\int_{\Omega_k}\psi\phi_k\,dv_{g_k}+\frac{1}{2}\int_{\Omega_k}\langle d\psi,d\phi_k\rangle \leq \\
&\int_{\Omega_k}\psi\phi_k\frac{1}{2}|dF_k|_{g_0}^2\,dv_{g_0} + \frac{1}{2}\int_{\Omega_k}|d\psi|_{g_0}|dF_k|_{g_0}\,dv_{g_0}+C\sqrt{\delta_k}\|\psi\phi_k\|_{W^{1,2}(\Omega_k,g_0)}\leq \\
&\|dF_k\|_{L^2(\supp(\psi))}^2+\|d\psi\|_{L^2}\|dF_k\|_{L^2(\supp(\psi))} + C(1+\|d\psi\|_{L^2}).
\end{split}
\end{equation*}
Thus, to complete the proof, it suffices to show that
\[
\|dF_k\|_{L^2(\supp(\psi))}^2\leq C\delta_k.
\]
To this end, notice that the conformality of $F_k$ forces
\[
\frac{1}{2}|dF_k|^2\leq |d\Psi_k|^2,
\]
and by \eqref{key.est}, we have
\begin{eqnarray*}
\int_{\supp(\psi)}|d\Psi_k|^2&\leq &C\delta_k+C\area_{g_0}(\supp(\psi))\\
&\leq & C\delta_k+C\sum_{i=1}^k r_i^2\leq C'\delta_k,
\end{eqnarray*}
where in the last step we used Lemma~\ref{stab.ineqs}.
\end{proof}

With Proposition \ref{ht.bd.prop} in place, Theorem \ref{var.close.thm} follows almost immediately.

\begin{proof}[Proof of Theorem \ref{var.close.thm}]
Let $N\subset \mathbb{S}_+^3$ be a free boundary minimal surface with genus zero and $k$ boundary components, embedded by first eigenfunctions. Then we are in the situation of Proposition \ref{ht.bd.prop} above, with $N=N_k$, $4\pi-|N_k|=\frac{1}{2}\delta_k$, and $\chi(N_k)=2-k$, and Prposition \ref{ht.bd.prop} gives the estimate
$$\int_Nx_4 \leq C\sqrt{1+|\chi(N)|}\sqrt{4\pi-|N|}.$$
The desired estimate follows by noting that $\dist_{\partial \mathbb{S}^3_+}\leq Cx_4$ in $\mathbb{S}^3_+$.
\end{proof}

Moreover, with Theorem \ref{var.close.thm} in place, we can now prove the second part of Theorem \ref{s3.ubd}, showing that $\frac{1}{2}\Lambda_1(M,\Gamma)\leq 8\pi-\frac{C_2}{|\chi(M)|}$ for basic reflection pairs $(M,\Gamma)$ of Scherk type.

\begin{proof}[Proof of Theorem \ref{s3.ubd} completed]
Let $(M,\Gamma)$ be a basic reflection pair of Scherk type; then, by definition, there is a reflection $\tau_1\in \Gamma$ such that, for any $g\in \Met_{\Gamma}(M)$, the fundamental domain of $\tau_1$ in $(M,g)$ is conformally equivalent to $\mathbb{S}^2\setminus \mathcal{D}$, where $\mathcal{D}$ is a union of disks centered on an equator.  In particular, note that reflection across this equator in $\mathbb{S}^2$ gives a conformal automorphism $\tau_2\in \Conf(M,g)$ which commutes with $\tau_1$, with respect to which $(M,\langle \tau_2\rangle)$ is also a basic reflection surface.

Now, suppose $g_{\max}\in \Met_{\Gamma}(M)$ is the metric realizing $\Lambda_1(M,\Gamma)$. Since $(M,g_{\max})$ is isometrically embedded in $\mathbb{S}^3$ by first eigenfunctions and the map $\tau_2\colon (M,g_{\max})\to (M,g_{\max})$ is conformal, it follows from \cite[Theorem 1]{MR} that $\tau_2$ is an isometry. In particular, letting $\bar{\Gamma}$ be the group generated by $\Gamma\cup \{\tau_2\}$, we see that in fact $g_{\max}\in \Met_{\bar{\Gamma}}(M)$, and $(M,g_{\max})$ realizes $\Lambda_1(M,\bar{\Gamma})=\Lambda_1(M,\Gamma)$, so we can assume without loss of generality that $\tau_2\in \Gamma$.

It then follows that the associated minimal surface $M\subset \mathbb{S}^3$ is symmetric with respect to the reflections across two \emph{orthogonal} great spheres $S_1,S_2$ corresponding to $\mathrm{Fix}(\tau_1)$ and $\mathrm{Fix}(\tau_2)$, respectively. Since $M$ is a basic reflection surface with respect to either reflection, it follows from Theorem \ref{var.close.thm} that
\[
\int_M\dist_{S_i}\leq C\sqrt{1+|\chi(M)|}\sqrt{8\pi-|M|}
\]
with respect to either of the orthogonal great spheres $S_1$ and $S_2$. On the other hand, one can see that there is a universal constant $c_0>0$ such that
\[
\int_M(\dist_{S_1}+\dist_{S_2})\geq c_0.
\]
Indeed, otherwise, there exists a sequence of minimal surfaces with bounded area, for which this integral converges to $0$. By a compactness theorem for stationary integral $2$-varifolds, this sequence converges (in the varifold sense) to a stationary integral $2$-varifold with support on the great circle $S_1\cap S_2$, which is clearly impossible.

Hence, any Scherk-type basic reflection surface must have 
\[\sqrt{1+|\chi(M)|}\cdot\sqrt{8\pi-\frac{1}{2}\Lambda_1(M,\Gamma)}\geq \frac{c_0}{C},\]
giving the desired estimate.
\end{proof}

%
%COROLLARY: ALL Z2XZ2 FAMILIES s.t. BRS wrt both reflections NECESSARILY CONVERGING TO PAIR OF ORTHOGONAL SPHERES MUST HAVE AREA$\sim 8\pi-C/\gamma$.

\section{Stability and upper bounds in the Steklov setting}

Let $N_+\subset \B^3_+$ be one half of a free boundary minimal surface $(N,g)\subset \mathbb{B}^3$ by first Steklov eigenfunctions symmetric with respect to reflection across $xy$-plane, so that $\sigma_1(N,g) = 1$, and such that $(N,g)$ is a basic reflection surface with respect to this reflection.
We can then identify $N_+$ isometrically with a conformal metric $(\Omega,g)$ on a domain $\Omega=\DD\setminus \mathcal{D}$ as in Section~\ref{sec:lowbounds_St}, where $\D = \cup_{j=1}^n D_{r_j}(x_j)$ is a collection of disks and half-disks. 

Setting $\Gamma_0:=\partial \mathcal{D}\cap \DD$
and $\Gamma_1:=\partial \Omega\setminus \Gamma_0,$ we can assume moreover that the identity map $I\colon \DD\to \DD$ is balanced along $\Gamma_1$, in the sense that
\[
\int_{\Gamma_1}I ds_g=0,
\]
by applying a conformal automorphism of $\mathbb{D}$. Set $\delta(N): = 2\pi - L_g(\Gamma_1)$.

\begin{lemma}
The set $\D = \cup_{j=1}^n D_{r_j}(x_j)$ satisfies an estimate of the form 
\begin{equation}
\label{hole.est}
\area(D, g_0)\leq C\delta(N)
\end{equation}
for a universal constant $C>0$. Furthermore, for any $\phi\in W^{1,2}(\Omega),$
\begin{equation}
\label{stek.stab}
\left|\int_{\partial \Omega}\phi \langle x,\nu_{g_0}\rangle ds_{g_0}-\int_{\Gamma_1}\phi ds_g\right|\leq \sqrt{\delta(N)}\|\phi\|_{W^{1,2}_{g_0}(\Omega)}.
\end{equation}

\end{lemma}

\begin{proof} The proof is analogous to the proof of Lemma~\ref{stab.ineqs}.

Since the quadratic form
\[
Q(\Phi,\Phi):=\int_\Omega|d\Phi|^2-\int_{\Gamma_1}|\Phi|^2ds_g,
\]
is nonnegative definite on balanced maps in $W^{1,2}(\Omega,\mathbb{R}^2)$, and the identity map $I\colon \DD\to \DD$ is balanced on $\Gamma_1$, we have
\[
0\leq Q(I,I)=2\area_{g_0}(\Omega)-L_g(\Gamma_1)\leq 2\pi\bigg(1-c\sum_{j=1}^nr_j^2\bigg)-L_g(\Gamma_1),
\]
so that
\[
2\pi c\sum_{j=1}^n r_j^2\leq 2\pi-L_g(\Gamma_1)=\delta(N)
\]
as claimed in~\eqref{hole.est}.
Furthermore, for any map $\Phi\colon \Omega\to \mathbb{R}^2$, the Cauchy-Schwarz inequality for $Q$ gives
\begin{equation}\label{stek-cs}
|Q(\Phi,I)|\leq \|d\Phi\|_{L^2(\Omega)}\sqrt{\delta(N)}.
\end{equation}
Taking $\Phi=\phi I$ for $\phi\in C^{\infty}(\Omega)$ in this inequality gives
\[
\left|\int_\Omega(2\phi+\langle x,d\phi\rangle )dx-\int_{\Gamma_1}\phi ds_g\right|\leq \sqrt{\delta(N)} \|\phi\|_{W^{1,2}_{g_0}(\Omega)},
\]
and noting that $2\phi+\langle x,d\phi\rangle=\mathrm{div}(\phi x)$, it follows in particular that
\[
\left|\int_{\partial \Omega}\phi \langle x,\nu_{g_0}\rangle ds_{g_0}-\int_{\Gamma_1}\phi ds_g\right|\leq \sqrt{\delta(N)}\|\phi\|_{W^{1,2}_{g_0}(\Omega)}.
\qedhere
\]

\end{proof}

%
%\color{blue}
%Alternatively, for $\phi$ such that $\phi \equiv 0$ on $\Gamma_0$, we can write
%$$\int_{\Gamma_1}\phi ds_g=\int_{\partial P}\phi \frac{\partial}{\partial\nu}(\frac{1}{2}|F|^2)=\int_P(\phi|dF|^2+\langle d\frac{1}{2}|F|^2,d\phi\rangle),$$
%and the preceding computations give
%\begin{equation}\label{stek.int.stab}
%|\int_P(2-|dF|^2)\phi+\langle \frac{1}{2}d(|x|^2-|F|^2),d\phi\rangle|\leq \sqrt{\epsilon}\|\phi\|_{W_{g_0}^{1,2}}.
%\end{equation}
%\color{black}

As an easy application, we can obtain coarse exponential lower bounds for $\delta$ in terms of the topology of $N$, from which the upper bounds in Theorems \ref{main.stek.bds} and \ref{max.stek.bds} follow.

\begin{prop}
\label{prop:St_lbd}
The gap $\delta=2\pi-L_g((\partial N)_+)$ has a lower bound of the form
\[
\delta\geq e^{-Cn},
\]
where $C$ is a universal constant.
\end{prop}
\begin{proof}

Let $N_+\cong (\Omega,g)$ as above, and let $\phi$ be the same log cutoff function as in the proof of Corollary~\ref{cor:Lap_lbd}, i.e. such that 
\[
\phi\equiv 0\text{ on }\mathcal{D},\qquad
\phi \equiv 1\text{ on }\D\setminus\bigcup_{j=1}^{n}D_{\sqrt{r_j}}(x_j),
\]
and
\[
\|d\phi\|_{L^2(\Omega)}^2\leq\sum_{j=1}^{n}\frac{C}{|\log r_j|}
\]
By~\eqref{hole.est} one has $\sum_j r_j^2\leq C\delta$,
so setting $r=\max\{r_j\}$, we must have $r^2\leq C\delta$. As a result, 
\[
\|d\phi\|_{L^2(\Omega)}^2\leq \frac{Cn}{|\log(C\delta)|}.
\]
On the other hand, since $\sigma_1^D(N_+)=1$, we have
\[
\|d\phi\|_{L^2(\Omega)}^2\geq \int_{\Gamma_1}\phi^2\,ds_g,
\]
while applying \eqref{stek.stab} with the test function $\phi^2$ gives
\[
\int_{\Gamma_1}\phi^2\,ds_g \geq \int_{\Gamma_1}\phi^2\,ds_{g_0}-C\sqrt{\delta}\left(\pi+\frac{n}{|\log(C\delta)|}\right)^{1/2},
\]
so combining all the estimates above gives
\[
\frac{Cn}{|\log(C\delta)|}\geq \int_{\Gamma_1}\phi^2\,ds_{g_0}-C\sqrt{\delta}\left(\pi+\frac{n}{|\log(C\delta)|}\right)^{1/2},
\]
while
\[
\int_{\Gamma_1}\phi^2\,dv_{g_0}\geq 2\pi-C\sum_{j=1}^{n}r_j\geq 2\pi-C\sqrt{n\delta}.
\]
Keeping in mind that $\delta$ tends to $0$ as $n\to\infty$, it follows that for sufficiently large $n$,
\[
\frac{C'n}{-\log(C\delta)}\geq 1,\quad\text{or}\quad \delta\geq Ce^{-Cn}\geq e^{-C'n}
\]
as claimed.
\end{proof}

We observe that one always has $n = k+\gamma$ or $n=k+\gamma-1$ depending on whether there are half-circles present. Thus, the bound in Proposition~\ref{prop:St_lbd} is indeed in terms of topology of $N$, and the upper bounds in Theorems \ref{main.stek.bds} and \ref{max.stek.bds} follow immediately.

\subsection{Convergence to disks}

Next, we show that if the gap $\delta=2\pi-L_g(\Gamma_1)$ vanishes sufficiently rapidly as $n$ gets large, then the associated FBMS must converge to the equatorial disk preserved by the reflection. To begin, note that applying \eqref{stek-cs} where $\Phi$ denotes the $L^2(\Gamma_1,ds_g)$-projection of $I|_{\Gamma_1}$ onto the orthogonal complement of the first Steklov eigenfunctions, we obtain an estimate of the form
\begin{equation}
\label{eq:St_Psi_proj}
\left(1-\frac{1}{\sigma_3^{N}(N_+)}\right)^2\|\Phi\|_{W^{1,2}(\Omega,g)}^2\leq  C\delta.
\end{equation}

\begin{lemma}
Let $N_+$ be as above. There exists $c>0$ and $\eps>0$ such that $\sigma_3^{N}(N_+)-2>c>0$ for $N_+$ as long as $\delta<\frac{\eps}{n}$. 
\end{lemma}
\begin{proof}
The strategy of the proof is the same as for Lemma \ref{l4.gap}. Denote by $\phi_{n,j}$ the $j$th eigenvalue of the mixed Steklov-Neumann problem on $(\Omega,\Gamma_1)$ with respect to the induced metric $g$, normalized so that
\[
\int_{\Gamma_1}\phi_{n,j}\phi_{n,l}\,ds_g=\delta_{jl}.
\]
As in the closed case, we can obtain uniform $L^{\infty}$ estimates for $\phi_{n,j}$ via a mean-value identity, as follows.
%\footnote{Should be standard, but couldn't find a reference.}.
 Doubling $N_+$ to obtain a free boundary minimal surface $N\subset \B^3$, recall that $\phi_{n,j}$ extends to a Steklov eigenfunction with eigenvalue $\sigma_j=\sigma_j^{N}(N_+)$. Fix a point $p\in \partial N$ and consider the vector field 
\[
X_p(x):=(x-p)=-\nabla h_{p,r},
\]
where $h_{p,r}(x):=\frac{1}{2}(r^2-|x-p|^2)$, and note that
\[
DX_p=\Id,
\]
and
\[
\langle X_p(x),x\rangle=1-\langle x,p\rangle=\frac{1}{2}|x-p|^2\text{ on }\partial \B^3.
\]
Recalling that $\phi_{n,j}^2$ is subharmonic on $N$, we then compute
\begin{equation*}
\begin{split}
2&\int_{N\cap B_r(p)}\phi_{n,j}^2\,dv_g=\int_{N\cap B_r(p)}\phi_{n,j}^2\mathrm{div}_N(X_p)\,dv_g=\\
&\int_{N\cap B_r(p)}\mathrm{div}_N(\phi_{n,j}^2X_p)+\langle \nabla\phi_{n,j}^2,\nabla h_{p,r}\rangle\,dv_g=\\
&\int_{N\cap B_r(p)}\mathrm{div}_N(\phi_{n,j}^2X_p+h_{p,r}\nabla \phi_{k,j}^2)+h_{p,r}\Delta \phi_{n,j}^2\,dv_g\leq\\
&\int_{N\cap \partial B_r(p)}\phi_{n,j}^2|(x-p)^T|\,ds_g+\int_{\partial N\cap B_r(p)}\left(\phi_{n,j}^2\frac{1}{2}|x-p|^2+2\sigma_jh_{p,r}\phi_{n,j}^2\right)\,ds_g\leq\\
r&\frac{d}{dr}\left(\int_{N\cap B_r(p)}\phi_{n,j}^2\,dv_g\right)+\left(\frac{1}{2}+\sigma_j \right) r^2\int_{\partial N\cap B_r(p)}\phi_{n,j}^2\,ds_g.
\end{split}
\end{equation*}

Dividing through by $r^3$ and rearranging, we then have
\begin{equation}
\label{ineq:diff_eig}
\frac{d}{dr}\left(\frac{1}{r^2}\int_{N\cap B_r(p)}\phi_{k,j}^2\,dv_g\right)\geq - \left(\frac{1}{2}+\sigma_j \right) \frac{1}{r}
\int_{\partial N\cap B_r(p)}\phi_{n,j}^2\,ds_g.
\end{equation}

\begin{lemma}
There exists a universal constant $C$ and a constant $C'$ depending only on the area of $N$ such that for all $r<2$ one has
\begin{equation}
\label{ineq:bdNNr}
|\partial N\cap B_\frac{r}{2}(p)|\leq \frac{C}{r}|N\cap B_r(p)|\leq C'r.
\end{equation}
\end{lemma}
\begin{proof}
Let us first apply divergence theorem on $N\cap B_r(p)$ to the vector field $Y_{p,r}(x) = h_{p,r}(x)x$, which yields
\begin{equation*}
\begin{split}
\frac{3}{8}r^2|\bd N\cap B_\frac{r}{2}(p)|\leq &\int_{\bd N\cap B_r(p)}h_{p,r}\,ds_g = \\
&\int_{N\cap B_r(p)}2h_{p,r} - \langle x, x-p\rangle\,dv_g\leq (r^2 + r)|N\cap B_r(p)|,
\end{split}
\end{equation*}
which proves the first inequality in~\eqref{ineq:bdNNr} after dividing by $r^2$.

To prove the second part, we apply~\eqref{ineq:diff_eig} to the constant eigenfunction to obtain
\[
\frac{d}{dr}\left(\frac{1}{r^2}|N\cap B_r(p)|\right)\geq -\frac{C}{(2r)^2}|N\cap B_{2r}(p)|.
\]
Denote $g(r) = r^{-2}|N\cap B_r(p)|$ and let $r_*$ be the point, where $g(r)$ achieves its maximum. Integrating the previous inequality on $[r_*,s]$ gives
\[
g(s) + C(s-r_*)g(r_*)\geq g(r_*).
\] 
Setting $s = r_* + \frac{1}{2C}\geq\frac{1}{2C}$ gives
\[
\|g\|_\infty = g(r_*)\leq 2g(s)\leq 8C^2\area(N)
\]
as claimed.
\end{proof}
Note that in the situation we are considering the area is always uniformly bounded, hence, both constants in~\eqref{ineq:bdNNr} are universal.
Substituting~\eqref{ineq:bdNNr} into~\eqref{ineq:diff_eig} and using that $\sigma_j\geq 1$ for $j\geq 1$ we arrive at
\[
\frac{d}{dr}\left(\frac{1}{r^2}\int_{N\cap B_r(p)}\phi_{k,j}^2\,dv_g\right)\geq -C\sigma_j\|\phi_{k,j}\|_{L^{\infty}}^2,
\]
and integrating over $r\in [0,\delta]$ gives an estimate of the form
\[
\phi_{k,j}(p)^2\leq C\delta \sigma_j\|\phi_{k,j}\|_{L^{\infty}}^2+\frac{C}{\delta^2}\int_{N\cap B_r(p)}\phi_{k,j}^2\,dv_g.
\]
In particular, letting $p$ be a maximum point for $\phi_{k,j}^2$ and taking $\delta=\frac{1}{2C\sigma_j}$, we arrive at an estimate of the form
\[
\|\phi_{k,j}\|_{L^{\infty}}^2\leq C'\sigma_j^2\int_N\phi_{k,j}^2\,dv_g\leq C'\sigma_j^2,
\]
where the last inequality follows from an application of the Green's formula
\begin{equation*}
\begin{split}
&\int_N\phi_{k,j}^2\,dv_g = -\frac{1}{4}\int_N\phi_{k,j}^2\Delta_g(|x|^2)\,dv_g = \\
\frac{1}{2}&\int_N|d\phi_{k,j}|^2|x|^2\,dv_g - \frac{1}{2}(\sigma_j - 1)\int_{\bd N} \phi_{k,j}^2\,ds_g\leq  \int_{\bd N} \phi_{k,j}^2\,ds_g.
\end{split}
\end{equation*}
With the $L^\infty$-bound established the rest of the proof is analogous to the closed case.
\end{proof}

Now, let $\Psi_n\colon \Omega\to \DD^2$ be the map given by composing the first two coordinate functions $x_1,x_2$ on $N_+\subset B^3_+$ with the conformal identification $F\colon \Omega\to N_+$, so that the coordinates of $\Psi_n$ span the $\sigma_1^N$-eigenspace of $(\Omega,g)$. It follows from~\eqref{eq:St_Psi_proj} that there exists a matrix $A_n$ such that
\[
\|I-A_n\Psi_n\|_{W^{1,2}(\Omega)}^2\leq C\delta_n.
\]
Similarly to the Laplacian case, this implies $A_n\in GL_2(\mathbb{R})$. Thus, we can rewrite the previous inequality as
\[
\|A_n^{-1} - \Psi_n\|^2_{W^{1,2}(\Omega)}\leq C\delta_n\|A_n^{-1}\|^2
\]
and using the trivial identity
\[
|A_n^{-1}|^2=\frac{1}{\pi}\int_{x\in \mathbb{D}^2}|A_n^{-1}|^2\,dx,
\]
we can argue as in Section \ref{eq.cvg-1} that 
\[
|A_n^{-1}|\leq C
\]
for $\delta_n<c_0$ sufficiently small. As a consequence, we obtain the estimate
\[
\|\Psi-A_n^{-1}I\|_{W^{1,2}(\Omega)}^2\leq C\delta_n.
\]
In particular, one has $\alpha_n=|A_n^{-1}dI|_{g_0}\leq C$ such that
\begin{equation}
\label{stek.dpsi}
\||d\Psi|_{g_0}-\alpha_n\|_{L^2_{g_0}(\Omega)}^2\leq C\delta_n.
\end{equation}
As an application, we have the following, which implies that $N_+$ must lie close to the disk $\{x_3=0\}$ if $\delta_n=o(1/n)$.

\begin{prop}\label{stek.plane.prop}
For $N_+\subset \B^3_+$ as above, whose double is the free boundary minimal surface $N\subset B^3$, the height coordinate $x_3$ satisfies
\[
\int_{\partial N}|x_3|\leq C\sqrt{n\delta_n}.
\]
\end{prop}
\begin{proof}
The proof is similar to the closed case. Let $\phi=x_3\circ F\in W^{1,2}(\Omega)$, so that $\phi$ is the first eigenfunction for the mixed Steklov-Dirichlet problem on $(\Omega,\Gamma_1)$. Using linear cutoffs, construct a function $\psi\colon \DD^2\to [0,1]$ such that $\psi\equiv 1$ on $\D$, $\psi$ vanishes outside $2\D$, the disks of twice the radius, and
%  $\psi\in C_c^{\infty}\left(\bigcup_{i=1}^{\gamma+1}D_{2r_i}(x_i)\cup \bigcup_{j=1}^kD_{2s_j}(y_j)\right),$ and
\[
\|d\psi\|_{L^2(\Omega)}^2\leq Cn.
\]

We then have
\[
\int_{\Gamma_1}(1-\psi)\phi ds_g=\int_\Omega \mathrm{div}((1-\psi)d\phi)\,dv_g=-\int_\Omega\langle d\psi,d\phi\rangle_g\,dv_g,
\]
and combining this with multiple applications of  \eqref{stek.stab} and \eqref{hole.est},
\begin{equation*}
\begin{split}
&\int_{\Gamma_1}\phi\,ds_g=\int_{\Gamma_1}\psi \phi ds_g-\int_\Omega\langle d\psi,d\phi\rangle_g\,dv_g \leq \\
&\int_{\partial \Omega}\psi\phi\langle x,\nu_{g_0}\rangle ds_{g_0}+C\sqrt{\delta_n}\|\psi\phi\|_{W_{g_0}^{1,2}(\Omega)}+\|d\psi\|_{L^2}\|d\phi\|_{L^2(\supp(\psi))}\leq \\
&C\sum_{i=1}^{n}r_i+ C\sqrt{\delta_n} + C\|d\psi\|_{L^2}\left(\sqrt{\delta_n}+\|d\phi\|_{L^2(\supp(\psi))}\right)\leq \\
&C'\sqrt{n\delta_n}+C\sqrt{n}\|d\phi\|_{L^2(\supp(\psi))}.
\end{split}
\end{equation*}
Thus, to complete the proof, it suffices to show that
\[
\|d\phi\|_{L^2(\supp(\psi))}^2\leq C\delta_n.
\]
To this end, again note that conformality of $F\colon \Omega\to N_+$ implies
\[
|d\phi|^2\leq |d\Psi|^2,
\]
and combining this with \eqref{stek.dpsi} gives
\[
\int_{\supp(\psi)}|d\phi|_{g_0}^2\,dv_{g_0}\leq C(\delta_n+\area_{g_0}(\supp(\psi))\leq C\left(\delta_n+\sum_{i=1}^{n}r_i^2\right),
\]
which together with \eqref{hole.est} gives the desired estimate.
\end{proof}

In particular, if $N\subset \mathbb{B}^3$ is a free boundary minimal embedding by first Steklov eigenfunctions which is a basic reflection surface with respect to reflection about the plane $P\subset \mathbb{R}^3$, note that the term $\delta_n$ in Proposition \ref{stek.plane.prop} is exactly $\delta=2\pi-2|\partial N|=2\pi-|N|$, while $n\leq C(1+|\chi(N)|)$. Thus, the conclusion of Proposition \ref{stek.plane.prop} can be recast as
$$\int_{\partial N}\dist_P\leq C\sqrt{2\pi-|N|}\sqrt{1+|\chi(N)|},$$
completing the proof of Theorem \ref{Tfbconv} from the introduction. Finally, we conclude with the proof of Theorem \ref{lawsy.bds}, identifying a family of free boundary minimal surfaces with prescribed Euler characteristic, whose asymptotic behavior resembles that of the Lawson surfaces in $\mathbb{S}^3$. 

\begin{proof}[Proof of Theorem \ref{lawsy.bds}]

For each $\gamma\in  \mathbb{N}\cup\{0\}$ and $k\in\{1,2\}$, we construct a family of basic reflection surfaces $(N,\Gamma)$ with $\Gamma\cong \mathbb{Z}_2\times \mathbb{Z}_2$ as follows. Consider a domain $\Omega \subset \mathbb{D}$ given by removing $\gamma$ disks from the interior of $\mathbb{D}$ with centers along the axis $\{y=0\}$ and $k=1$ or $2$ disks with centers in $\{(1,0),(-1,0)\}$, and set $\Omega^+=\Omega\cap \{y\geq 0\}$. Doubling $\Omega$ across $\partial\Omega\setminus \partial \mathbb{D}$ then gives a compact surface $N_{\gamma,k}$ with genus $\gamma$ and $k$ boundary components, admitting two commuting reflections $\tau_1,\tau_2\in \Diff(N)$ such that $\Omega$ is a fundamental domain for $\tau_1$, while $\Omega^+\cup \tau_1\Omega^+$ is a fundamental domain for $\tau_2$. In particular, we then see that $N_{\gamma,k}$ is a basic reflection surface with respect to both $\tau_1$ and $\tau_2$.

Setting $\Gamma=\langle \tau_1,\tau_2\rangle \leq \Diff(N_{\gamma,k})$, we can apply \cite[Theorem 9.15, item (6)]{KKMS} to deduce the existence of an embedded free boundary minimal surface $N_{\gamma,k}\subset \mathbb{B}^3$ of genus $\gamma$ and $k$ boundary components realizing $\Sigma_1(N_{\gamma,k},\Gamma)$. Moreover, there is a pair of orthogonal planes $P_1,P_2\subset \mathbb{R}^3$ such that $\tau_i$ is induced by the reflection through $P_i$, and an application of Theorem \ref{Tfbconv} forces
\[
\int_{\partial N_{\gamma,k}}\dist_{P_i}\leq C\sqrt{2\pi-|N_{\gamma,k}|}\sqrt{1+\gamma}
\]
for both $i=1$ and $i=2$. 
Similarly to the closed case 
%therOn the other hand, since the boundary of a free boundary minimal surface in $\mathbb{B}^3$ cannot be concentrated arbitrarily close to a pair of antipodal points $P_1\cap P_2\cap \mathbb{S}^2$, it follows that 
there is a universal constant $c_0>0$ such that
\[
c_0\leq \int_{\partial N_{\gamma,k}}\dist_{P_1}+\dist_{P_2}
\]
independent of $\gamma$, and it follows that
\[
2\pi-|N_{\gamma,k}|\geq \frac{c_1}{1+\gamma}
\]
for some $c_1>0$, as claimed.
\end{proof}

%\bibliographystyle{alpha}

%\bibliography{bibliography}
%\bibliographystyle{alpha}

\end{document}